\documentclass[12pt,reqno]{amsart}
\usepackage{amssymb,amsfonts,amsthm,amsmath,mathrsfs}
\usepackage{cite}
\usepackage[shortlabels]{enumitem}
\usepackage[left=1 in,top=1 in,right=1 in, bottom=1 in]{geometry}
\usepackage{graphicx}
\usepackage{color}
\usepackage[pagebackref=false]{hyperref}
\def\d{{\rm d}}
\def\eqdef{\stackrel{\rm def}{=}}
\definecolor{darkred}{rgb}{.70,.12,.20}

\definecolor{darkgreen}{rgb}{.20,.52,.14}

\definecolor{byz}{rgb}{.44,.16,.39}

\numberwithin{equation}{section}

\newtheorem{theorem}{Theorem}[section]
\newtheorem{lemma}[theorem]{Lemma}

\newtheorem{definition}[theorem]{Definition}
\newtheorem{corollary}[theorem]{Corollary}
\newtheorem{assumption}[theorem]{Assumption}

\newtheorem{proposition}[theorem]{Proposition}

\theoremstyle{remark}
\newtheorem{remark}[theorem]{\bf{Remark}}
\newtheorem{example}[theorem]{\bf{Example}}
\usepackage{todonotes}

\newcommand{\varep}{\varepsilon}

\newcommand{\beq}{\begin{equation}}
\newcommand{\eeq}{\end{equation}}
\newcommand{\beqs}{\begin{equation*}}
\newcommand{\eeqs}{\end{equation*}}
\newcommand{\ba}{\begin{array}}
\newcommand{\ea}{\end{array}}
\newcommand{\beas}{\begin{eqnarray*}}
\newcommand{\eeas}{\end{eqnarray*}}
\newcommand{\bea}{\begin{eqnarray}}
\newcommand{\eea}{\end{eqnarray}}
\newcommand{\bal}{\begin{align}}
\newcommand{\eal}{\end{align}}

\newcommand{\bals}{\begin{align*}}
\newcommand{\eals}{\end{align*}}

\newcommand{\tnum}{\rm(\roman*)}
\newcommand{\rnum}{\rm(\alph*)}

\newcommand{\R}{\ensuremath{\mathbb R}}

\newcommand{\N}{\ensuremath{\mathbb N}}

\newcommand{\bds}{\begin{displaystyle}}
\newcommand{\eds}{\end{displaystyle}}

\newcommand{\remove}[1]{} %-- ON
\renewcommand{\remove}[1]{#1} % OFF

\definecolor{darkred}{rgb}{.70,.12,.20}

\definecolor{darkgreen}{rgb}{.20,.52,.14}

\title[Linear non-divergence parabolic equations in non-cylindrical sets]
{Linear non-divergence parabolic equations in non-cylindrical  space-time sets
}

\author[L. Hoang]{Luan Hoang$^{1}$}
\address{$^1$Department of Mathematics and Statistics,
Texas Tech University,
1108 Memorial Circle, Lubbock, TX 79409--1042, U. S. A.}
\email{luan.hoang@ttu.edu}

\author[A. Ibragimov]{Akif Ibragimov$^{1,2,*}$}
\address{$^{2}$Oil and Gas Research Institute, Russian Academy of Sciences, 3 Gubkin Str, Moscow, Russia, 119333 }
\email{ilya1sergey@gmail.com}

\thanks{$^*$Corresponding author.}

\date{\today}

 
\begin{document}

\begin{abstract} 
We study linear parabolic equations of the second order in non-divergence form in a general set which is non-cylindrical with respect to the spatial and time variables. The restriction of the set on any bounded time interval is bounded, but the spatial diameter of each fixed-time cross section can be unbounded as time tends to infinity. For homogeneous problems, we obtain exponential,  power and other intermediate decaying rates for the solutions in different scenarios. For inhomogeneous problems, we obtain all-time and asymptotic, as time tend to infinity, estimates  for the solutions in terms of the data on the parabolic boundary and forcing functions. The analysis requires subtle properties of general space-time sets and an iteration scheme to bootstrap the Growth Lemma. The time steps for such an iteration need not be constant and are adapted to the growth of the diameter.
\end{abstract}

\maketitle 
\tableofcontents 

% \pagestyle{myheadings}\markboth{\sc L.~Hoang and A.~Ibragimov}
% {\sc Linear parabolic equations in non-divergence form  in non-cylindrical  sets}

%%%%%%%%%

\section{Introduction }\label{intro}
 
In the theory of parabolic equations, particularly for long-term dynamics, the following two situations are usually considered. 

\begin{enumerate}
    \item[(a)] Parabolic equations in the divergence form
\beq \label{divform}
\frac{\partial u}{\partial t} - \sum_{i,j=1}^n \frac{\partial}{\partial x_i}\left( a_{i,j}(x,t) \frac{\partial  u}{\partial x_j}\right)+\sum_{i=1}^n b_i(x,t) \frac{\partial u}{\partial x_i}+c(x,t)u(x,t)=f(x,t),
\eeq
where $x\in \R^n$ is the spatial variable, $t\ge 0$ is the time variable, and $u(x,t)$ is the unknown function. (The terminology \textit{divergence form} is referred to  the second term in \eqref{divform}.)

\item[(b)] The space-time cylindrical sets in $\R^{n+1}$, i.e., $(x,t)\in U\times [0,\infty)$, where  $U$ is a fixed open set. 
\end{enumerate}

Regarding (a), equations in the divergence form  can be obtained by some conservation law and formulation of the flux across the boundary. They have been studied intensively with a vast liturature, see e.g. \cite{LadyParaBook68,LiebermanPara96,IKO} and references therein. 
The non-divergence equations are 
\beq \label{nondiv}
\frac{\partial u}{\partial t} - \sum_{i,j=1}^n a_{i,j}(x,t) \frac{\partial ^2 u}{\partial  x_i \partial x_j}+\sum_{i=1}^n b_i(x,t) \frac{\partial u}{\partial x_i}+c(x,t)u(x,t)=f(x,t).
\eeq
The equation \eqref{divform}, in the case $a_{ij}$ are smooth, can be converted to \eqref{nondiv}.
In spite of that, non-divergence equations are not just mathematical generalizations of the divergence form. 
They can come naturally from the probabilistic form of the conservation of mass that was used by Einstein \cite{Einstein1905} in his study of the Brownian motion, see details in \cite{HI3Rus}. 
Technically speaking, the divergence form of the second term in \eqref{divform} allows the energy method to be used in a powerful way. Such a tool is not applicable to the non-divergence equations and, hence,  their analysis is greatly limited. Therefore, the non-divergence equations always pose more challenges and require different approaches from the divergence-form counterpart. For example, Di Giorgi \cite{DeGiorgi57} and Nash \cite{Nash1958} obtained the H\"older regularity for elliptic equations in the divergence-form long before Krylov and Safonov \cite{KrylovSafonov1981} did, with very different techniques, for the non-divergence ones. 

For the long-time dynamics of parabolic equations, most research is devoted to (b). 
However, many phenomena occur in a domain with the boundary changing in time. For example, they can be the melting ice/snow problem, the changing sea level in a long period of time or geological problems of the layering.
They pose the problem of investigating the stability, as $t\to\infty$, of the solutions of the parabolic equation of non-divergence form with $x$ belonging to the sets which can change in $t$. 
Due to the last factor, the results up to now require those sets to have limited geometries \cite{Ushakov1981} even for the divergence form. This is due to the use of the energy method via the integration by parts. 
We aim to obtain more general results by using a different approach.

With these motivations,   we focus on the parabolic equations in non-divergence form in a general non-cylindrical set, i.e., neither (a) and nor (b). 
There are not many papers dedicated to this topic. We review a few of them and other related work here.
The paper \cite{Cerem1968} obtains asymptotic estimates, as $t\to\infty$, using an explicit barrier function for all time. This method leads to constraints on the coefficients of the equation and the shape of the boundary. 
For  the stabilization of the solution  of the Cauchy  problem, i.e., the existence of the limit of $u(x,t)$ as $t\to\infty$, see the review and results in  \cite{Denisov2005}. 
Also, the paper \cite{AD2014} investigates the necessary and sufficient conditions  for the stabilization in the case (b) above with $U$ being a fixed unbounded domain in $\R^n$.

In this article, we extend the method of the Growth Lemma to study the behavior of the solutions $u(x,t)$, as $t\to\infty$, depending on the changes in $t$ of the domain of the function $x\mapsto u(x,t)$.
The  Growth Lemma is a tool that refines the Maximum Principle to evaluate the decay of the solutions either in the spatial or time variable. This technique was initiated by E. M. Landis in a series of papers published in \textit{Uspehi Mat. Nauk.} \cite{Landis1959,Landis1963}, see also \cite{Landis1968}. It plays a crucial role in the celebrated work \cite{KrylovSafonov1981} mentioned above. For its recent developments, see \cite{Safonov2010,KrylovReview2021} are references there in.
In the current study, we will use the variation implemented in \cite{HI4}, see also  \cite{HI3Rus,HI5}, to deal with rather general non-cylindrical space-time sets. The spatial cross section (for fixed time) of the set is allowed to grow infinitely large as time goes to infinity. We will show that if the growth is not too fast, the diffusion is still strong enough to deduce a decaying estimate for the solutions. We restrict only to bounded drifts and treat both homogeneous and inhomogeneous problems.

The paper is organized as follows.
Section \ref{notation} contains the notation and some simple facts which will be  used through out the paper.
In Section  \ref{domainsec}, we specify two types of boundary points for general space-time sets in Definition \ref{parabb}. Namely, the upper-base where the PDE is additionally defined, and the parabolic boundary where the boundary conditions are imposed. Lemma \ref{rmk2} describes the relations between the constituents of a general set $Q$, that is, its interior points and boundary points of each type, with the constituents of its restrictions on bounded  time intervals. Although elementary, they are subtle and crucial to  our later study. To avoid peculiar geometries, our space-time sets of interest will be required to satisfy Assumption \ref{GGamcond}. Some typical but already complicated cases  are presented in Propositions \ref{gset}, \ref{gset1} and \ref{gset2}. 
In Section \ref{MaxGrowth}, we present the main tools to study PDEs in non-divergence form for general sets which are the Maximum Principle -- Theorem \ref{genmax} -- and Growth Lemma -- Lemma \ref{stgrowth}. In Theorem \ref{maxcor}, we establish the decrease in time for the spatial maximum. This is a fundamental property of the solutions.
In the case of cylindrical space-time sets, it is a trivial consequence of the Maximum Principle. However, in the non-cylindrical case, it is not obvious. In fact, the proof is much more delicate and requires Assumption \ref{GGamcond} and subtle relations in Lemma \ref{rmk2}. Theorems \ref{maxprin2} and \ref{lemG2} are applications of  the Maximum Principle and Growth Lemma to estimating the solution of inhomogeneous problems. Their roles are two fold. On the one hand, the obtained estimates are primary which will be improved much more later. On the other hand, they serve as the basic steps for later iterations to obtain large time estimates.
 Section \ref{homsec} is focused on homogeneous problems. We start with a preliminary result -- Theorem \ref{STthm0} -- which derives a power decay in time for very general sets in the case of zero drift. To improve such a decay, we set up a general iteration scheme in subsection \ref{scheme} for both cases of zero drift and bounded drifts utilizing the Growth Lemma.
The case of zero drift is treated in subsection \ref{HZsec} with the main result being  Theorem \ref{STthm1}. Other improvements for more specific situations are in  Theorem \ref{STthm2}, Example \ref{egln} and Corollary \ref{cor1}.
The case of bounded drifts is treated in subsection \ref{HBsec}. Theorem \ref{STthm3} contains the main estimates for various scenarios regarding the spatial diameters' growth as time tends to infinity. An intermediate decaying rate for the solution is established in Theorem \ref{STthm4} which corresponds to an intermediate diameters' growth.
In Section \ref{inhomsec}, we deal with inhomogeneous equations and inhomogeneous Dirichlet boundary conditions.  The preliminary recursive estimates are established  Proposition \ref{genlem}.
In the case of zero drift, they immediately result in the long-time behavior in Theorem \ref{NHthm0}.
In the case of bounded drifts, although they already allow us to obtain estimates for the solutions for all time, it is still difficult to see the long-time behavior. Moreover, Example \ref{egs} shows that, unlike the cylindrical case, the solution may not be bounded although the forcing functions, initial and boundary data are.
In Theorem \ref{NHthm}, under some specified conditions on the monotonicity and/or decaying rate of the forcing function and boundary data, we obtain much more specific estimates for the solution when $t\to\infty$.
Finally, Appendices \ref{apA}--\ref{apD} contain technical proofs of some results in Section \ref{domainsec}.

\section{Notation}\label{notation}

Throughout the paper, the spatial dimension $n\ge 1$ is fixed. 
For a vector $x\in\R^n$, its Euclidean norm is denoted by $|x|$.
 Let $\mathcal M^{n\times n}$ denote the set of $n\times n$ matrices of real numbers, and $\mathcal M^{n\times n}_{{\rm sym}}$ denote the set of symmetric matrices in $\mathcal M^{n\times n}$.
For two matrices $A,B\in \mathcal M^{n\times n}$, their inner product $\langle A,B\rangle$ is the trace ${\rm Tr}(A^{\rm T}B)$.

For a non-empty set $S$ and a function $f:S\to \R$, we recall that the positive and negative parts of $f$ are  $f^+=\max\{0,f\}$ and $f^-=(-f)^+$. Then one has  
\beq   \label{fpm}
 f=f^+-f^-, \quad 
|f|=f^+ + f^-=\max\{f^+,f^-\}.
\eeq
\beq \label{ffrel}
0\le f^+\le |f|, \quad 0\le f^-\le |f|,\quad  -f^-\le f\le f^+.
\eeq 
If $f$ is bounded then
\beq \label{sup3}
\sup_S |f|=\max\left\{\sup_S f^+,\sup_S f^-\right \}.
\eeq

Define the projections $\mathbb P_1:\R^{n+1}\to \R^n$ and $\mathbb P_2:\R^{n+1}\to \R$ by 
\beq\label{PP12}
\mathbb  P_1(x,t)=x\text{ and }\mathbb  P_2(x,t)=t, \text{ respectively, for }(x,t)\in\R^n\times \R.
\eeq 

For $x_0\in\R^n$ and $R>0$, denote by $B_R(x_0)$ the open ball in $\R^n$ of radius $R$ centered at $x_0$.
For $J\subset \R$, define
$\mathcal C_{x_0,R}^J=B_R(x_0)\times J$.

Let $Q$ be a  subset of $\R^{n+1}$, $J$ be a subset of $\R$, and $t\in\R$.
Denote
\beqs
Q_J=Q\cap (\R^n\times J) \text{ and }  Q_t=Q_{\{t\}}. 
\eeqs

For $x,y\in \R^m$, denote by $[x,y]$  the line segment connecting $x$ and $y$, i.e.,  
\beq \label{lineseg}
[x,y]=\{(1-\tau)x+\tau y:0\le \tau\le 1\}.
\eeq 

For  a non-empty, bounded subset $U$  of $\R^n$ and a point $y\in \R^n$, denote 
\begin{align}
\label{di2}
{\rm dist}(y,U)&=\inf\{|x-z|:x\in U\}=\min\{|x-z|:x\in\overline U\},\\
\label{Rs3}
R_*(y,U)&=\max\{|x-y|:x\in\overline U\}.
\end{align}
Recall that the diameter of $U$ is  
\beqs 
{\rm diam} (U)=\sup\{|x-z|:x,z\in  U\} =\max\{|x-y|:x,z\in \overline U\}.
\eeqs 

It is clear that when $U$ has a non-empty interior and $y\not\in \overline U$, one has 
\beq\label{Rrpos}
R_*(y,U)>{\rm dist}(y,U)>0 \text{ and } {\rm diam}(U)>0.
\eeq 

\begin{lemma}[{ \cite[Lemma 4.13]{HI4} }]\label{choicey}
Let $U$ be a non-empty, bounded subset of $\R^n$.  For any $r>0$, there exists a point $y\not \in \overline U$ such that
    \beq\label{rRd}
    {\rm dist}(y,U)=r\text{ and }R_*(y,U)=r+{\rm diam}(U).
    \eeq
\end{lemma}
      Although \cite[Lemma 4.13]{HI4} was proved for an open set $U$, its proof still works in the case ${\rm diam}(U)>0$. In the case ${\rm diam }(U)=0$, the set $U$ is $\{x_0\}$, hence taking $y\in\partial B_r(x_0)$ gives 
      ${\rm dist}(y,U)=R_*(y,U)=r$ and  we obtain \eqref{rRd} again.

\section{Geometry of general space-time sets}  \label{domainsec} 

In the theory of IBVP for parabolic equations, the parabolic boundary plays a special role. Below are its definition and properties.

\begin{definition}\label{parabb}
Let $Q$ be  subset of $\R^n \times \R$. 
The upper base of $Q$, denoted by $\gamma(Q)$, is the set of points $(x,t) \in \R^n\times \R$ such that  there exists a number $h>0$ for which 
\beq\label{CCpt}
\mathcal C^{(t-h,t)}_{x,h}\subset Q \text{ and } \mathcal C^{(t,t+h)}_{x,h}
\subset \R^{n+1}\setminus Q.
\eeq
Define
\beqs 
\widetilde Q=Q\cup \gamma(Q)\text { and } 
\Gamma(Q)=\partial Q\setminus\gamma(Q).
\eeqs 
We call $\Gamma(Q)$ the parabolic boundary of $Q$.
\end{definition}

We clarify the order of operations acting on sets in our notation by 
$$\overline{Q}_J=(\overline{Q})_J,\ 
 \gamma(Q)_J=(\gamma(Q))_J,\  
 \Gamma(Q)_J=(\Gamma(Q))_J,$$ 
$$ \widetilde{Q}_J=(\widetilde{Q})_J,\ 
 \overline{Q}_t=\left (\overline{Q} \right)_t,\ 
 \gamma(Q)_t=(\gamma(Q))_t, \text{ etc.} 
$$
It is cautioned that the closure $\overline{Q_t\,}$ of $Q_t$ is a subset of $\overline Q_t$ but is not necessarily equal to $\overline Q_t$, see Example \ref{egs}\ref{eg0} below.

The following remarks on Definition \ref{parabb} are in order.

\begin{enumerate}[label=\rnum]
\item Obviously, in the case $Q$ is an empty set, one has
$\gamma(\emptyset)=\emptyset$, $\Gamma(\emptyset)=\emptyset$ and $\widetilde \emptyset=\emptyset$.

\item It is clear from \eqref{CCpt} that
\beqs
\gamma(Q)\subset \partial Q\text{ which implies  }
\partial Q=\Gamma(Q)\cup \gamma(Q).
%\text{ and }\overline Q=\widetilde Q\cup \Gamma(Q).
\eeqs

\item\label{rc} Let $Q$ be an open, non-empty set in $\R^{n+1}$. 
Obviously, $Q$ does not intersects  $\partial Q$, hence,  $\gamma(Q)$ and $\Gamma(Q)$. 
Let $(x,t)\in \gamma(Q)$ and $h>0$ be as in Definition \ref{parabb}.
Let $y\in B_h(x)$ and set $\delta=h-|y-x|$. We have $\delta\in(0,h)$ and
\beqs
\mathcal C_{y,\delta}^{(t-\delta,t)}\subset \mathcal C_{x,h}^{(t-h,t)}\subset Q\text{ and }
\mathcal C_{y,\delta}^{(t,t+\delta)}\subset \mathcal C_{x,h}^{(t,t+h)}\subset \R^n\setminus Q.
\eeqs
Thus, $(y,t)\in \gamma(Q)$, and therefore,
\beq\label{Bhg}
B_h(x)\times\{t\}\subset \gamma(Q).
\eeq
As a consequence of \eqref{CCpt} and \eqref{Bhg}, we particularly have
\beqs
    \mathcal C_{x,h/2}^{(t-h/2,t]}=\mathcal C_{x,h/2}^{(t-h/2,t)}\cup (B_{h/2}(x)\times\{t\})\subset Q\cup \gamma(Q),
\eeqs
which yields    
    \beq \label{tilnb}
    \mathcal C_{x,h/2}^{(t-h/2,t]}\subset \widetilde Q.
    \eeq 
    Moreover, since   $\mathcal C_{x,h/2}^{(t,t+h/2)}\cap  Q=\emptyset$, thanks to the second condition in \eqref{CCpt}, and the set $\mathcal C_{x,h/2}^{(t,t+h/2)}$ is open, we have 
$\mathcal C_{x,h/2}^{(t,t+h/2)}\cap \overline Q=\emptyset$,
which gives
    \beq\label{cq1} 
 \mathcal C_{x,h/2}^{(t-h/2,t+h/2)}\cap \overline Q= C_{x,h/2}^{(t-h/2,t]}\cap \overline Q .
    \eeq 
Combining \eqref{cq1} with \eqref{tilnb}, we obtain
    \beq\label{QbQtil}
       \mathcal C_{x,h/2}^{(t-h/2,t+h/2)} \cap \overline Q
      \subset \widetilde Q \cap \overline Q= \widetilde Q.
    \eeq

    \item\label{re} Let $Q$ be an open, non-empty set in $\R^{n+1}$. Let $(x,t)$ be any point in  
    $\widetilde Q$.  In the case $(x,t)\in Q$, then, thanks to $Q$ being open, there is $h>0$ such that
    $\mathcal C_{x,h/2}^{(t-h/2,t+h/2)}
      \subset  Q $, thus,
    \beqs
     \mathcal C_{x,h/2}^{(t-h/2,t+h/2)}\cap \overline Q=\mathcal C_{x,h/2}^{(t-h/2,t+h/2)}
      \subset \widetilde Q.
    \eeqs
     In the case $(x,t)\in \gamma(Q)$, we have \eqref{QbQtil}.
    Therefore, $\widetilde Q$ is an open set in the metric space $\overline Q$. Consequently, $\Gamma(Q)=\overline Q\setminus \widetilde Q$ is a closed set in the metric space $\overline Q$. Since $\overline Q$ is closed in $\R^{n+1}$, we deduce 
    \beq\label{closed} \Gamma(Q) \text{ is closed.}
    \eeq 
    
\item\label{rd} In the case $Q$ is open and bounded, then, thanks to \eqref{closed}, the  set 
    \beq\label{Gcomp} 
    \Gamma(Q) \text{ is compact.}
    \eeq 
    If $J$ is a compact set in $\R$, then \eqref{Gcomp} implies
    \beq \label{GJcomp}
    \Gamma(Q)_J=\Gamma(Q)\cap (\R^n\times J) \text{ is compact.}
    \eeq 
\end{enumerate}

\begin{remark}\label{rmk1}
Let $Q$ be an open, non-empty  set in $\R^{n+1}$ and $u$ is a function from $\widetilde Q$ to $\R$. 
If $(x,t)\in\gamma(Q)$, then, thanks to \eqref{tilnb}, it is valid to define (whenever they exist) the spatial partial derivatives 
\beqs %\label{ddmean}
 \frac{\partial u}{\partial x_i}(x,t),\ 
 \frac{\partial^2 u}{\partial x_i\partial x_j}(x,t)
\eeqs 
while $\frac{\partial u}{\partial t}(x,t)$ must be understood as 
 the left-sided time derivative, i.e., 
 \beqs %\label{dtmean}
\frac{\partial u}{\partial t}(x,t)=\lim_{h\nearrow 0}\frac{u(x,t+h)-u(x,t)}{h}.
\eeqs 
(Clearly, the right-sided time derivative at $(x,t)$ cannot be defined in this case.)
\end{remark}

Thanks to \eqref{tilnb} and Remark \ref{rmk1}, we can have the following definition.

\begin{definition}\label{class21}
Given  an open, non-empty  set $Q$ in $\R^{n+1}$, let $\mathcal Q$ be $\widetilde Q$ or $\widetilde Q_{J}$, where $J\ne \emptyset$  is either an open subset of $\R$ or $J=(-\infty,t_2]$ or $J=(t_1,t_2]$.
Define $C_{x,t}^{2,1}(\mathcal Q)$ to be the collection of functions $u:\mathcal Q\to \R$ such that $u$, $u_{x_i}$, $u_{x_i x_j}$, for $i,j=1\ldots n$, and $u_t$ are continuous.     
\end{definition}

The main tool in analyzing non-divergence PDEs is the Maximum Principle which is established for bounded sets. However, we will study the solutions $u(x,t)$ as $t\to\infty$, that is, $(x,t)$ belongs to an unbounded set $Q$. 
Therefore, we will apply the Maximum Principle to bounded subsets $E$ of $Q$. Doing so, there is a need to understand  the relations between boundary points of $Q$ and $E$. For general sets in $\R^{n+1}$, these relations can be complicated as showed in the next example.

\begin{example}\label{egs}
We present some examples.
%The following examples show how complicated the general space-time sets can cause while applying the Maximum Pricinciple to restriction of a set.
\begin{enumerate}[label=\rnum]
\item\label{eg0} Let $Q=\{(x,t)\in \R^n\times\R:t\in\R,|x|<|t|\}$
or $Q=\{(x,t)\in \R^n\times(0,\infty):|x|<t\}$.
Then $Q_0=\emptyset$ but $\overline{Q}_0=\{0\}$. Thus, $\overline{Q_0\,}=\emptyset\ne \overline{Q}_0$.

    \item\label{ega}    Let $Q$ be the open set in the plane enclosed by the loop consisting of the line segments connecting the points $(0,0),(2,0),(2.2),(3,2),(3,3),(3,2),(1,1),(0,1),(0,0)$. \
    \begin{itemize}
        \item     Let $E=Q_{(1,2)}=(1,2)\times(1,2)$. Then we have
    \beqs   \overline{Q}_{1}=[0,2]\times\{1\},\ 
    \overline{E}_{1}=[1,2]\times\{1\}, \ 
    \overline{Q}_{2}=[1,3]\times\{2\},\ 
    \overline{E}_{2}=[1,2]\times\{2\}.
    \eeqs 
Obviously, $\overline{Q}_{1}\ne \overline{E}_{1}$ and 
    $\overline{Q}_{2}\ne \overline{E}_{2}$.
In fact, $\overline{Q}_{1}\setminus \overline{E}_{1}=[0,2)\times\{1\}$, its furthest left point $(0,1)$ belongs to $\Gamma(Q)$, but the rest, that is  $(0,2)\times\{1\}$, is a subset of $\gamma(Q)$.
Meanwhile,    
    \beq \label{QGex}
    \overline{Q}_{2}\setminus \overline{E}_{2}=(2,3]\times\{2\}\subset \Gamma(Q).
\eeq   
This observation \eqref{QGex} turns out to be true in general, see Lemma \ref{rmk2}\ref{pc}, which plays an important role in Theorem \ref{maxcom} below.

    \item Let $F=Q_{(0,1)}=(0,2)\times(0,1)$. Then the point $(1,1)$ belongs to $\gamma(F)\setminus \gamma(Q)$. Consequently, $\widetilde F$ is not a subset of $\widetilde Q$. This prompts us to impose serious conditions in \ref{asv} of Assumption \ref{GGamcond} below.
    \end{itemize}

\item \label{egb} 
 Let  $S_1=(1,2)\times\{0\}$ and
 $$Q=(-2,2)\times (-2,2)\setminus ([-1,1]\times[-1,1] \cup S_1).$$
 Let $E=Q_{(-2,0)}$. 
 Then $\gamma(E)=S_1\cup S_2\cup S_3$, 
 where  $S_2=(-2,-1)\times \{0\}$ and $S_3=(-1,1)\times\{-1\}$.
 Observe that 
\beq \label{SSS}
S_1\subset \Gamma(Q),\ 
S_2\subset Q\text{ and } 
S_3\subset \gamma(Q).
\eeq  
The first property in \eqref{SSS} shows that $\widetilde E$ is not a subset of $\widetilde Q$. 
  \end{enumerate}
\end{example}

The following are the main relations between a general set and its restrictions.

\begin{lemma}\label{rmk2}
Let $Q$ be an open, non-empty subset of $\R^{n+1}$.
    Let $t_1<t_2$ and $E=Q_{(t_1,t_2)}$. Then the following statements fold true.
\begin{enumerate}[label=\rnum]
    \item\label{pa} If $t\in(t_1,t_2)$, then
    \beq\label{qe1}
    Q_t=E_t, \ 
    \overline{Q}_t=\overline{E}_t,\  (\partial Q)_t=(\partial E)_t,
    \eeq
    \beq\label{qe2}
    \gamma(Q)_t=\gamma(E)_t
    \text{ and }     \Gamma(Q)_t=\Gamma(E)_t  .
    \eeq

    \item \label{pb} $\overline{E}_{t_1}=\Gamma(E)_{t_1}$, 
    $\gamma(E)_{t_1}=\emptyset$,
    $\gamma(Q)_{t_2}\subset \gamma(E)_{t_2}$ and 
    $\Gamma(E)_{t_2}\subset \Gamma(Q)_{t_2}$.
Consequently,
    \beq\label{gEQ} 
\widetilde E=\widetilde E_{(t_1,t_2]}\text{ and }
\gamma(E)=\gamma(Q)_{(t_1,t_2)}\cup \gamma(E)_{t_2}.
    \eeq 

\item\label{pc} 
$\overline{Q}_{t_2}\setminus \overline{E}_{t_2}\subset \Gamma(Q)$.

\item \label{pd} If $E$ is bounded and 
\beq \label{Qnone}
\overline{Q}_t\ne\emptyset\text{ for all }t\in(t_1,t_2),
\eeq 
then
\beq \label{EEne}
\overline{E}_{t_1}\ne \emptyset\text{ and }\overline{E}_{t_2}\ne \emptyset.
\eeq 
\end{enumerate}
\end{lemma}
\begin{proof}%[Proof of Lemma \ref{rmk2}]
   Observe that 
   \beq\label{Ett}
   E_{(t'_1,t'_2)}=Q_{(t'_1,t'_2)} \text{ for all }t'_1, t'_2\in\R \text{ with }t_1\le t'_1<t'_2\le t_2.
   \eeq

\medskip\noindent\ref{pa} Let $t\in(t_1,t_2)$. Select $h>0$ sufficiently small such that 
\beq \label{tth}
t_1<t-h<t+h<t_2.
\eeq 
Hence, thanks to \eqref{Ett},
   \beq\label{Eth}
   E_{(t-h,t+h)}=Q_{(t-h,t+h)}.
   \eeq
With property \eqref{Eth}, we obtain \eqref{qe1}. 

  We prove \eqref{qe2} next. For any small $h>0$ satisfying \eqref{tth}, one has, thanks to \eqref{Eth}, 
  \beq\label{CQE}
  \mathcal C^{(t-h,t)}_{x,h}\cap Q=\mathcal C^{(t-h,t)}_{x,h}\cap E \text{ and } \mathcal C^{(t,t+h)}_{x,h}\cap Q=\mathcal C^{(t,t+h)}_{x,h}\cap E.
  \eeq 
We rewrite the condition \eqref{CCpt} in Definition \ref{parabb} in the following equivalent form
\beq\label{gequiv}
\mathcal C^{(t-h,t)}_{x,h}\cap Q=\mathcal C^{(t-h,t)}_{x,h}\text{ and }\mathcal C^{(t,t+h)}_{x,h}\cap Q=\emptyset.
\eeq
Thanks to \eqref{CQE}, the condition \eqref{gequiv} is equivalent to 
\beqs%\label{gequiv}
\mathcal C^{(t-h,t)}_{x,h}\cap E=\mathcal C^{(t-h,t)}_{x,h}\text{ and }\mathcal C^{(t,t+h)}_{x,h}\cap E=\emptyset.
\eeqs
Therefore, $(x,t)\in\gamma(Q)$ is equivalent to $(x,t)\in\gamma(E)$. This proves the first property in \eqref{qe2}.

Now, we have from the last property in  \eqref{qe1} and the first  property in \eqref{qe2} that
\beqs
\Gamma(Q)_t
= [\partial Q\setminus \gamma(Q)]_t
= (\partial Q)_t\setminus \gamma(Q)_t
=(\partial E)_t\setminus \gamma(E)_t
= [\partial E\setminus \gamma(E)]_t
=\Gamma(E)_t.
\eeqs 
Thus, we obtain the second  property in \eqref{qe2}. 

\medskip\noindent\ref{pb}
The proof is laid out in five steps below.

\medskip\noindent\emph{Step 1: Proof of $\overline{E}_{t_1}=\Gamma(E)_{t_1}$.}    
 Obviously, $\Gamma(E)_{t_1}\subset \overline{E}_{t_1}$. 
For the reverse inclusion, let $(x,t_1)\in \overline{E}_{t_1}$.
For $k\in\N$, take $\tau_k=t_1-1/k$. Clearly, $\tau_k<t_1$. Then
\beq\label{xtau0}
(x,\tau_k)\not\in E\text{ and }(x,\tau_k)\to(x,t_1)\text{  as }k\to\infty.
\eeq
This and the facts $(x,t_1)\in \overline{E}$ and $E$ is open imply $(x,t_1)\in\partial E$. 
Property \eqref{xtau0} also shows that the first requirement in \eqref{CCpt} fails for $t:=t_1$ and $Q:=E$, hence, $(x,t_1)$ does not belong to $\gamma(E)$. 
Thus, $(x,t_1)\in \partial E\setminus \gamma(E)=\Gamma(E)$.
Therefore, $\overline{E}_{t_1} \subset \Gamma(E)$ which implies $\overline{E}_{t_1} \subset \Gamma(E)_{t_1}$.

\medskip\noindent\emph{Step 2: Proof of $\gamma(E)_{t_1}=\emptyset$.}
Note that 
\beqs 
 \gamma(E)_{t_1}=(\partial E)_{t_1}\setminus \Gamma(E)_{t_1}\subset \overline{E}_{t_1}\setminus \Gamma(E)_{t_1}.
 \eeqs 
Then, thanks to Step 1, the last difference set is empty, thus, we obtain $\gamma(E)_{t_1}=\emptyset$.

\medskip\noindent\emph{Step 3: Proof of $\gamma(Q)_{t_2}\subset \gamma(E)_{t_2}$.}
Let $(x,t_2)\in \gamma(Q)$.
Choose $h>0$ small such that $t_2-h>t_1$ and \eqref{CCpt} holds for $t=t_2$, that is,
  \beq\label{CQE0}
  \mathcal C^{(t_2-h,t_2)}_{x,h}\cap Q=\mathcal C^{(t_2-h,t_2)}_{x,h} \text{ and }  \mathcal C^{(t_2,t_2+h)}_{x,h}\cap Q=\emptyset.
  \eeq
On the one hand, we have,  thanks  to \eqref{Ett}, $E_{(t_2-h,t_2)}=Q_{(t_2-h,t_2)}$, and hence, together with the first part of \eqref{CQE0},
  \beq\label{CQE1}
  \mathcal C^{(t_2-h,t_2)}_{x,h}\cap E=\mathcal C^{(t_2-h,t_2)}_{x,h}\cap Q=\mathcal C^{(t_2-h,t_2)}_{x,h}.
  \eeq 
  On the other hand, the fact $E\subset Q$ and the second part of \eqref{CQE0} imply
    \beq\label{CQE2}
 \mathcal C^{(t_2,t_2+h)}_{x,h}\cap E\subset \mathcal C^{(t_2,t_2+h)}_{x,h}\cap Q=\emptyset.
  \eeq
Then it follows from \eqref{CQE1} and \eqref{CQE2} that $(x,t_2)\in \gamma(E)$.

\medskip\noindent\emph{Step 4: Proof of $\Gamma(E)_{t_2}\subset \Gamma(Q)_{t_2}$.}
Consider $(x,t_2)\in \Gamma(E)_{t_2}$.
Clearly,
\beq\label{xQbar}
(x,t_2)\in \Gamma(E)\subset \overline E\subset \overline Q.
\eeq
Observe, for any $h>0$, that
\beq\label{CC2}
\mathcal C^{(t_2,t_2+h)}_{x,h}
\subset \R^{n+1}\setminus E.
\eeq
Because $(x,t_2)\not\in \gamma(E)$, then as a consequence of Definition \ref{parabb} and  \eqref{CC2}, 
there is a sequence  $(h_k)_{k=1}^\infty$ of positive numbers, with $h_k\to 0$ as $k\to\infty$, such that  the first requirement in \eqref{CCpt} fails for $t:=t_2$, $h=h_k$ and $Q:=E$, that is 
$\mathcal C^{(t_2-h_k,t_2)}_{x,h_k}$ is not a subset of $E$. Thus, for each $k\ge 1$, there exists a point $(x_k,\tau_k)\in \mathcal C^{(t_2-h_k,t_2)}_{x,h_k}$  such that  
\beq\label{xtauk} 
(x_k,\tau_k)\not \in E.
\eeq 
For large $k$, we have $\tau_k\in (t_2-h_k,t_2)\subset (t_1,t_2)$, hence, by the first property in \eqref{qe1}, $E_{\tau_k}=Q_{\tau_k}$.
This fact and \eqref{xtauk} imply $(x_k,\tau_k)\not \in Q$.
Clearly, $(x_k,\tau_k)\to (x,t_2)$ as $k\to\infty$. Together with $Q$ being open, the last two properties of $(x_k,\tau_k)$ imply $(x,t_2)\not\in Q$. Combining this with the fact $(x,t_2)\in\overline Q$  from \eqref{xQbar}, we obtain $(x,t_2)\in \partial Q$.
If $(x,t_2)\in \gamma(Q)$, then, by the virtue of the result proved in Step 3 above, we must have $(x,t_2)\in\gamma(E)$ which contradicts the fact $(x,t_2)\in \Gamma(E)$. Therefore, we have $(x,t_2)\in \partial Q\setminus \gamma(Q)=\Gamma(Q)$. Thus, $(x,t_2)\in \Gamma(Q)_{t_2}$.

\medskip\noindent\emph{Step 5: Proof of \eqref{gEQ}.}
Since $\gamma(E)\subset \widetilde E\subset\overline{E}\subset \R^n\times[t_1,t_2]$, we have 
  \beq \label{EtgE}
  \widetilde E=\widetilde E_{[t_1,t_2]}=\widetilde E_{t_1}\cup \widetilde E_{(t_1,t_2]}
  \text{ and }
    \gamma(E)=\gamma(E)_{[t_1,t_2]}=\gamma(E)_{t_1}\cup \gamma(E)_{(t_1,t_2]}.
  \eeq 
Clearly, $E_{t_1}=\emptyset$. Together with Step 2, it implies
$\widetilde E_{t_1}=E_{t_1}\cup \gamma(E)_{t_1}=\emptyset$.
Therefore, $\widetilde E=\widetilde E_{(t_1,t_2]}$ which proves the first property in \eqref{gEQ}.
 Thanks to $\gamma(E)_{t_1} =\emptyset$ from Step 2 again,  it follows from the second part of \eqref{EtgE} that 
    \beq\label{gE} 
    \gamma(E)=\gamma(E)_{(t_1,t_2]}=\gamma(E)_{(t_1,t_2)}\cup \gamma(E)_{t_2}.
    \eeq 
By \eqref{qe2}, one has
\beq\label{gEint} 
\gamma(E)_{(t_1,t_2)}=\gamma(Q)_{(t_1,t_2)}.
\eeq
Then the second property of \eqref{gEQ} follows from \eqref{gE} and \eqref{gEint}.

\medskip\noindent\ref{pc}  Suppose the statement is false.
Then there is 
\beq\label{xt2}
(x_*,t_2)\in (\overline{Q}_{t_2}\setminus \overline{E}_{t_2})\setminus \Gamma(Q) =\overline{Q}_{t_2}\setminus (\overline{E}_{t_2}\cup \Gamma(Q)) .
\eeq
This implies $(x_*,t_2)\in \overline Q\setminus \Gamma(Q) =\widetilde Q=Q\cup \gamma(Q)$.
\begin{itemize}
    \item Consider $(x_*,t_2)\in Q$. Because $Q$ is open, there is a sequence $(x_*,\tau_k)\in Q$, for $k\in\N$,  with $t_1<\tau_k<t_2$ that converges to $(x_*,t_2)$ as $k\to\infty$ . Obviously,  $(x_*,\tau_k)\in E$, hence $(x_*,t_2)\in \overline{E}$. Consequently, $(x_*,t_2)\in \overline{E}_{t_2}$ which contradicts  \eqref{xt2}.
    \item Consider $(x_*,t_2)\in \gamma(Q)$. Then by the virtue of the third property in part \ref{pb}, one must have $(x_*,t_2)\in \gamma(E)\subset \overline E.$ This implies $(x_*,t_2)\in \overline{E}_{t_2}$ which contradicts \eqref{xt2} again.
\end{itemize}
Therefore, we must have  $\overline{Q}_{t_2}\setminus \overline{E}_{t_2}\subset \Gamma(Q)$. 

\medskip\noindent\ref{pd}
We prove  the part $\overline{E}_{t_1}\ne\emptyset$ in \eqref{EEne} first. Take a sequence $(h_k)_{k=1}^\infty$ in the interval $(0,t_2-t_1)$ that converges to $0$ as $k\to\infty$.
By the assumption \eqref{Qnone}, there exists a point  $(x_k',t_1+h_k/2)\in \overline Q$  for each $k\ge 1$.
Consequently, there is $(x_k,\tau_k)\in Q$ sufficiently close to $(x_k',t_1+h_k/2)$ so that 
\beq \label{taulim}
\tau_k\in(t_1,t_1+h_k).
\eeq 
Clearly, $\tau_k\in(t_1,t_2)$ which yields $(x_k,\tau_k)\in E$. 
By the compactness of $\overline E$, there is a subsequence $(x_{k_j},\tau_{k_j})$ of $(x_k,\tau_k)$ that converges to a point $(x_*,t_*)$ as $j\to\infty$. Then $(x_*,t_*)\in \overline{E}$ and $t_*=\lim_{j\to\infty}\tau_{k_j}=t_1$ thanks to \eqref{taulim}. These  imply $(x_*,t_*)\in \overline{E}_{t_1}$. Therefore, the set $\overline{E}_{t_1}$ is not empty.

The proof for the part $\overline{E}_{t_2}\ne\emptyset$ in \eqref{EEne} is similar.
\end{proof}

The following type of set will play an important role in the current work.

\begin{definition}\label{IG}
    Define the set %$\mathcal I(Q)=\mathbb P_2(\overline{Q})$ and 
    \beq\label{mGQ}
    \mathcal G(Q)=\left \{(t_1,t_2)\in \mathbb P_2(\overline{Q})\times \mathbb P_2(\overline{Q}): t_1<t_2, \left[\gamma(Q_{(t_1,t_2)})\right]_{t_2}\subset \widetilde Q \right \}.
    \eeq
\end{definition}

Suppose $(t_1,t_2)\in \mathcal G(Q)$ and denote $E=Q_{(t_1,t_2)}$.
We have from \eqref{mGQ} that $\gamma(E)_{t_2}\subset \widetilde Q_{t_2}$.
Combining this fact with the second property in \eqref{gEQ} gives  
\beq\label{gamEQ}
\gamma(E)\subset \gamma(Q)_{(t_1,t_2)}\cup \widetilde Q_{t_2}\subset \widetilde Q_{(t_1,t_2]}.
\eeq
Together with $\widetilde E=E\cup\gamma(E)=Q_{(t_1,t_2)}\cup\gamma(E)$, \eqref{gamEQ} implies
     \beq\label{tilEQ}
         \widetilde E\subset  \widetilde Q_{(t_1,t_2]}\text{ and, consequently, }
    \widetilde E\subset \widetilde Q.
    \eeq
With \eqref{tilEQ}, an equation which holds in $\widetilde Q$ will also hold in $\widetilde E$.
    
\begin{assumption}\label{GGamcond}
Let $Q$ be a subset of $\R^{n+1}$ and denote $I=\mathbb P_2(\overline{Q})$. Assume the following.
\begin{enumerate}[label=\tnum]
    \item \label{as0} $Q$ is non-empty and open.
    \item\label{asi} The set $I$ is connected.
    \item\label{asii}  For any $t_1,t_2\in I$ with $t_1<t_2$,  the set $Q_{(t_1,t_2)}$ is bounded. 

    \item\label{asv} For any $t_1,t_2\in I$ with $t_1<t_2$, either 
    \begin{enumerate}[label=\rnum]
        \item\label{asa}  $(t_1,t_2)\in \mathcal G(Q)$, or 
        \item\label{asb} there is a sequence $(\tau_k)_{k=1}^\infty\subset (t_1,t_2)$ converging to $t_2$ such that $(t_1,\tau_k)\in \mathcal G(Q)$ for all $k$ and, with the notation $E=Q_{(t_1,t_2)}$,  
    \beq\label{Hsdk}
    \lim_{k\to\infty}\left (\sup_{X\in \overline{E}_{t_2}} {\rm dist}(X,\overline{E}_{\tau_k})\right)=0.
    \eeq
    \end{enumerate}
\end{enumerate}
\end{assumption}

The following remarks on Assumption \ref{GGamcond} are in order.
\begin{enumerate}[label=\rnum]
    \item  Since $Q$ is open and  non-empty, the set $I$ is neither empty nor a single point.
Consequently, the interior set $I^0$ is not empty.

\item Let $t_1$, $t_2$ be as in \ref{asii}. Thanks to the condition \ref{asi}, 
\beq \label{ttnone}
\text{the set $Q_{(t_1,t_2)}$ is not empty.}
\eeq 
Indeed, let $t'\in(t_1,t_2)$, then  \ref{asi} implies $t'\in I$ and, hence, there is a point $(x',t')$ in $\overline{Q}$.
By selecting $(x,t)\in Q$ sufficiently close to $(x',t')$, we have $t\in(t_1,t_2)$, and, therefore,  $(x,t)\in Q_{(t_1,t_2)}$.
\end{enumerate}

\begin{lemma}\label{moreE}
If $Q$  satisfies Assumption \ref{GGamcond}\ref{as0}--\ref{asii} and $t_1,t_2\in \mathbb P_2(\overline Q)$ with $t_1\le t_2$, then
\beq \label{OKbQ}
\text{$\overline{Q}_{[t_1,t_2]}$ is bounded.}
\eeq 
\end{lemma}

We will prove Lemma \ref{moreE} in Appendix \ref{apA}.
Below are some typical cases of sets that satisfy Assumption \ref{GGamcond}. 
    \textit{For the rest of this section, consider $J=\R$ or $J=(t_*,\infty)$ for some $t_*\in\R$.}
We start with a funnel centered about a curvy axis.

\begin{proposition}\label{gset}
Let $X:\overline{J}\to\R^n$ and $R:\overline{J}\to [0,\infty)$ be continuous functions. Assume \beq \label{zeroc}
\text{the set $R^{-1}(\{0\})$ is either empty or consists of only isolated points.}
\eeq 
Define the set
\beq\label{Qeg}
    Q=\{(x,t)\in \R^{n}\times J:|x-X(t)|<R(t)\}.
\eeq 
Then  $Q$ satisfies Assumption \ref{GGamcond} and $\partial Q=\Gamma(Q)$. 
\end{proposition}

The proof of Proposition \ref{gset} will be presented in Appendix \ref{apB}.
Next, we generalize Proposition \ref{gset} to allow the function $R$ to be discontinuous.

For a function $R:\overline{J}\to \R$, define  for $t\in \overline{J}$, 
\beqs
R(t-)=\lim_{\tau\nearrow t}R(\tau),\quad R(t+)=\lim_{\tau\searrow t}R(\tau),
\eeqs
\beq\label{RRpm}
 R_-(t)=\min\left\{R(t-),R(t+)\right\}\text{ and } 
  R_+(t)=\max\left\{R(t-),R(t+)\right\}.
\eeq
whenever the one-sided limits exist and are finite.

\begin{assumption}\label{Rco}
For a function $R:\overline{J}\to \R$, assume the following.
\begin{enumerate}[label=\rnum]
    \item\label{R1} The one-sided, resp. right-handed, limits of $R(t)$ exist and are finite at each $t$ in $J$, resp. $\partial J$.
    \item\label{R2} The set of discontinuity points  of $R$ is a subset of $J$ and  is either empty, or finite, or a sequence strictly increasing to infinity. 
\end{enumerate}
\end{assumption}

\begin{proposition}\label{gset1}
Let $X:\overline{J}\to\R^n$  be a continuous function. 
Let $R:\overline{J}\to [0,\infty)$ be a function that satisfies Assumption \ref{Rco}, condition \eqref{zeroc}, and 
\beq\label{Rmc}
 R(t)=R_-(t)\text{ for all }t\in\overline{J}.
\eeq
Then the set $Q$ defined by \eqref{Qeg} satisfies Assumption \ref{GGamcond}. 
\end{proposition}

In the one-dimensional case, we describe similar sets without the axis $(X(t),t)$.

\begin{proposition}\label{gset2}
Let $R_1,R_2:\overline{J}\to \R$ be given functions  such that $R_1\le R_2$, the condition \eqref{zeroc} is met for $R:=R_2-R_1$,
and Assumption \ref{Rco} is satisfied with $R:=R_i$ for $i=1,2$. 
Assume, for all $t\in\overline{J}$,
\beq\label{RR12}
 R_1(t)=(R_{1})_+(t) \text{ and } R_2(t)=(R_{2})_-(t).
\eeq
Then the set 
$Q\eqdef \{(x,t)\in \R\times J:R_1(t)<x<R_2(t)\}$ 
 satisfies Assumption \ref{GGamcond}.
\end{proposition}

The proofs of Propositions \ref{gset1} and \ref{gset2} will be given in Appendices \ref{apC} and \ref{apD}.
More information about the boundary points of $Q$ and of its particular subsets are showed in those proofs.
For example, for $Q$ in Proposition \ref{gset1}, one has, thanks to \eqref{gQ1},
\beqs 
\text{$\gamma(Q)=\emptyset$ if and only if $R(t-) \le  R(t+)$ for all $t\in J$.}
\eeqs 
Similarly, for $Q$ in Proposition \ref{gset2}, one has, thanks to \eqref{gQ2}, 
\beqs 
\text{$\gamma(Q)=\emptyset$ if and only if $ R_1(t_2-)\ge R_1(t_2+)$ and  $ R_2(t_2-)\le  R_2(t_2+)$ for all $t\in J$.}
\eeqs 
    
\section{Maximum Principle and Growth Lemma}\label{MaxGrowth}

Let $Q$ be  an open, non-empty set  in  $\R^{n+1}$.
For any subset $J$ of $\R$ and number $t\in \R$, denote
\beq\label{SJdef}
\mathcal S_J=\Gamma(Q)_J=\Gamma(Q)\cap (\mathbb R^n\times J)
\text{ and }
\mathcal S_t=\mathcal S_{\{t\}}=\Gamma(Q)\cap (\mathbb R^n\times \{t\}).
\eeq

Given two functions  $A:\widetilde Q\to \mathcal M^{n\times n}_{{\rm sym}}$ and $b:\widetilde Q\to \R^n$, define the linear operator $L$ in non-divergence form by
\beq \label{Ltil}
Lv = v_t-\langle A(x,t),D^2v\rangle  +b(x,t)\cdot \nabla v \text{ for }v\in C_{x,t}^{2,1}(\widetilde Q). 
\eeq 
where $D^2v$ is the Hessian matrix $(\partial^2v/\partial x_i\partial x_j)_{i,j=1,\ldots,n}$.
Recall that the value $Lv(x,t)$ at any $(x,t)\in\widetilde Q$ is understood according to Remark \ref{rmk1}.

We will often refer to the following assumption.

\begin{assumption}\label{secondA}
 Assume the following.
\begin{enumerate}[label=\tnum]
    \item\label{A1} There exists a constant $c_0>0$ such that
\beq\label{Aelip}
\xi^{\rm T} A(x,t)\xi \ge c_0|\xi|^2\text{ for all $(x,t)\in \widetilde Q$ and  $\xi\in \R^n$.}
\eeq 
    \item\label{A2} There exists a constant $M_1>0$ such that
\beqs%\label{TAB2}
{\rm Tr}(A(x,t))\le M_1 \text{ for all }(x,t)\in \widetilde Q.
\eeqs 
\end{enumerate}
\end{assumption}

\subsection{Maximum Principle}
We recall the standard Maximum Principle for a general set from \cite[Chapter 3, Corollary 2.1]{LandisBook}.

\begin{theorem}    [Maximum Principle] \label{genmax}
    Let $Q$ be an open, bounded, non-empty set in $\R^{n+1}$ 
    and suppose Assumption \ref{secondA}\ref{A1} holds.
    If $u\in C(\overline{Q})\cap C_{x,t}^{2,1}(\widetilde Q)$ satisfies $Lu\le 0$ in $\widetilde Q$, then 
\beq\label{maxmax}
\max_{\overline Q} u= \max_{\Gamma(Q)} u.
\eeq
\end{theorem} 

Observe that the set $\Gamma(Q)$ on the right-hand side of \eqref{maxmax} is compact thanks to \eqref{Gcomp}.

As a consequence of Theorem \ref{genmax}, we obtain the following decrease in time for the spatial maximum of $u^+$.

\begin{theorem}\label{maxcor}
 Let $Q$ be a set in $\R^{n+1}$ that satisfies  Assumption \ref{GGamcond}.
Suppose Assumption \ref{secondA}\ref{A1} holds.
Let $t_1$ and $t_2$ be any numbers in $\mathbb P_2(\overline{Q})$ with $t_1<t_2$.
Let $u\in C(\overline{Q}_{[t_1,t_2]})\cap  C_{x,t}^{2,1}( \widetilde Q_{(t_1,t_2]})$ satisfy $Lu\le 0$ in $\widetilde Q_{(t_1,t_2]}$ and 
\beq \label{uSt12}
u\le 0\text{ on }\mathcal S_{(t_1,t_2]}. 
\eeq 
Then one has
\beq\label{maxcom}
\max_{\overline{Q}_{t_2}} u^+\le \max_{\overline{Q}_{t_1}} u^+.
\eeq
\end{theorem}  
\begin{proof}
Denote $I=\mathbb P_2(\overline{Q})$ and $E=Q_{(t_1,t_2)}$. 
Then  $E\ne \emptyset$ thanks to \eqref{ttnone}, and $E$ is bounded thanks to Assumption \ref{GGamcond}\ref{asii}.
In fact, because of \eqref{OKbQ},  $\overline{Q}_t$ is compact for any $t\in I$. 
Clearly $\overline{E}\subset\overline{Q}$. We also observe the following. 
\begin{enumerate}[label=\rnum]
    \item  By the definition of $I$, we already have $\overline{Q}_{t_1}\ne \emptyset$ and $\overline{Q}_{t_2}\ne \emptyset$. Moreover, $(t_1,t_2)\subset I^0$.
    Then, thanks to Lemma \ref{rmk2}\ref{pd}, one has $\overline{E}_{t_1}\ne \emptyset$ and $\overline{E}_{t_2}\ne \emptyset$.
    
    \item Denote $F=\overline{Q}_{t_2}\setminus \overline{E}_{t_2}$.
    On the one hand, thanks to  Lemma  \ref{rmk2}\ref{pc}, we have $F\subset \Gamma(Q)_{t_2}$. On the other hand, it follows from \eqref{uSt12} that $u^+=0$ on $\Gamma(Q)_{t_2}$.
    Hence, in the case $F\ne \emptyset$,  $u^+=0$ on $F$. Therefore,
    \beq\label{mm3}
    \max_{\overline{Q}_{t_2}} u^+=\max_{\overline{E}_{t_2}} u^+.
    \eeq
    \item We decompose
\beq \label{ga}
\Gamma(E)=\Gamma(E)_{t_1}\cup \Gamma(E)_{(t_1,t_2)} \cup \Gamma(E)_{t_2}.
\eeq 
From part \ref{pa} of Lemma \ref{rmk2},  the second property in \eqref{qe2} implies
\beq \label{gb}
\Gamma(E)_{(t_1,t_2)}=\Gamma(Q)_{(t_1,t_2)}.
\eeq 
Recall the first and fourth properties of part \ref{pb} in Lemma \ref{rmk2} 
\beq \label{gc}
\Gamma(E)_{t_1}=\overline{E}_{t_1}\text{ and }\Gamma(E)_{t_2}\subset \Gamma(Q)_{t_2}.
\eeq
Utilizing \eqref{gb} and \eqref{gc} in \eqref{ga} gives
\beq\label{GQG}
 \overline{E}_{t_1}=\Gamma(E)_{t_1}\subset \Gamma(E)\subset \overline{E}_{t_1}\cup \Gamma(Q)_{(t_1,t_2]}=\overline{E}_{t_1}\cup \mathcal S_{(t_1,t_2]}.
\eeq
Since $u^+=0$ on $\mathcal S_{(t_1,t_2]}$, thanks to \eqref{uSt12}, we obtain from \eqref{GQG} that
\beq\label{GEQ}
\max_{\Gamma(E)}u^+=\max_{\overline{E}_{t_1}} u^+.
\eeq
\end{enumerate}

\medskip\noindent\textbf{Part A.} Consider the case when $(t_1,t_2)\in \mathcal G(Q)$, see Definition \ref{IG}.
Thanks to the consequence $\widetilde E\subset \widetilde Q$ from \eqref{tilEQ}, we have $u\in C(\overline{E})\cap C^{2,1}_{x,t}(\widetilde E)$ and $Lu\le 0$ in $\widetilde E$.
 Applying Theorem \ref{genmax} to the function $u$ and the set $E$,  we have 
 \beqs%\label{mm0}
\max_{\overline E} u = \max_{\Gamma(E)} u.
\eeqs
It follows that
 \beq\label{mmp}
\max_{\overline E} u^+ = \max_{\Gamma(E)} u^+.
\eeq
Regarding the left-hand side of \eqref{mmp}, we clearly have
 \beq\label{mEE}
\max_{\overline E} u^+\ge \max_{\overline{E}_{t_2}} u^+ .
\eeq
We deduce from \eqref{mEE}, \eqref{mmp}   and \eqref{GEQ} that
\beq\label{mm1}
\max_{\overline{E}_{t_2}} u^+
\le \max_{\overline E} u^+ = \max_{\Gamma(E)} u^+
= \max_{\overline{E}_{t_1}} u^+.
\eeq
We already have $\overline{E}_{t_i}\subset \overline{Q}_{t_i}$, for $i=1,2$. For the right-hand side of \eqref{mm1}, we we simply estimate
    \beq\label{mm2}
    \max_{\overline{E}_{t_1}} u^+\le \max_{\overline{Q}_{t_1}} u^+.
    \eeq
Combining \eqref{mm3}, \eqref{mm1}  and \eqref{mm2}, we obtain
\beqs
\max_{\overline{Q}_{t_2}} u^+=\max_{\overline{E}_{t_2}} u^+
\le \max_{\overline{E}_{t_1}} u^+\le \max_{\overline{Q}_{t_1}} u^+,
\eeqs
which proves \eqref{maxcom}.

\medskip\noindent\textbf{Part B.}  Consider the case when $(t_1,t_2)\not\in \mathcal G(Q)$. 
Thanks to Assumption \ref{GGamcond}\ref{asv},  there must be a sequence $(\tau_k)_{k=1}^\infty$ in the interval $(t_1,t_2)$ such that $\tau_k\to t_2$ as $k\to\infty$ and 
$(t_1,\tau_k)\in \mathcal G(Q)$ and \eqref{Hsdk} holds.
For any $k\ge 1$,  we apply Part A above to the pair $(t_1,\tau_k)\in \mathcal G(Q)$ to have
 \beq\label{maxtau}
\max_{\overline{Q}_{\tau_k}} u^+ \le \max_{\overline{Q}_{t_1}} u^+.
\eeq

Let $\varep>0$. By the uniform continuity of $u^+$ in $\overline E$, and \eqref{Hsdk}, one has, for sufficiently large $k$,  that
\beq\label{maxtau2}
\max_{\overline{E}_{t_2}} u^+ \le \max_{\overline{E}_{\tau_k}} u^+ +\varep=\max_{\overline{Q}_{\tau_k}} u^+ +\varep.
\eeq
(For the last identity, we used the second property in \eqref{qe1} to have $\overline{E}_{\tau_k}=\overline{Q}_{\tau_k}$.)
Thus, combining \eqref{mm3} with \eqref{maxtau2} gives
\beq\label{QQ0}
\max_{\overline{Q}_{t_2}} u^+ = \max_{\overline{E}_{t_2}} u^+ 
\le \max_{\overline{Q}_{\tau_k}} u^+ +\varep.
\eeq
Fixing a number $k$ in \eqref{QQ0} and applying inequality \eqref{maxtau}, we deduce 
\beq\label{QQe}
\max_{\overline{Q}_{t_2}} u^+ 
\le \max_{\overline{Q}_{t_1}} u^+ +\varep.
\eeq
Passing $\varep\to0$ in \eqref{QQe}, we obtain \eqref{maxcom}.
\end{proof}

The next result consists of estimates for inhomogeneous problems. It is a counter part of \cite[Theorem 3.1]{HI4}. 

\begin{theorem}\label{maxprin2}
Let $Q$ be a subset of $\R^{n+1}$ that satisfies  Assumption \ref{GGamcond}.
Suppose Assumption \ref{secondA}\ref{A1} holds.
Assume $I=[t_*,t^*]$ with $t_*<t^*$, or $I=[t_*,\infty)$ is a subset of $\mathbb P_2(\overline Q)$. Denote $I'=I\setminus\{t_*\}$.
Let $u\in C(\overline{Q}_I)\cap C_{x,t}^{2,1}(\widetilde Q_{I'})$.
Then the following statements hold true.
\begin{enumerate}[label=\tnum]
\item\label{MP1}  Suppose there is a function $f_1\in C(I,[0,\infty))$ such that 
    $Lu(x,t)\le f_1(t)$ for all $(x,t)\in \widetilde Q_{I'}$. 
Then
   \beq\label{max2}
        \max_{\overline{Q}_t} u^+ \le \max_{\overline{Q}_{t_*}\cup \mathcal S_{(t_*,t]}} u^+  +\int_{t_*}^t f_1(\tau)\d \tau \text{ for all }t\in I.
    \eeq
    
\item\label{MP3} Suppose there is a  function $f\in C(I,[0,\infty))$  such that
       $|Lu(x,t)|\le f(t)$ for all $(x,t)\in \widetilde Q_{I'}$. 
Then
    \beq\label{maxabs}
        \max_{\overline{Q}_t} |u| \le \max_{\overline{Q}_{t_*}\cup \mathcal S_{(t_*,t]}} |u| + \int_{t_*}^t f(\tau)\d \tau \text{ for all }t\in I.
        \eeq
\end{enumerate}
\end{theorem}
\begin{proof}
Observe that the set $\overline{Q}_{t_*}\cup \mathcal S_{(t_*,t]}$ on the right-hand side of \eqref{max2} is, in fact,
$\overline{Q}_{t_*}\cup \mathcal S_{[t_*,t]}$, hence, it is compact thanks to \eqref{GJcomp}.
Define the function 
\beqs 
G_1(t)=\max_{\overline{Q}_{t_*}\cup \mathcal S_{(t_*,t]}} u^+\text{ for }t\in I.
\eeqs 
Note that $G_1(t)$ is increasing.

(i) Consider any number $T\in I'$ first. Define 
$v(x,t)=u(x,t)-G_1(T)-\int_{t_*}^t f_1(\tau)\d\tau$.
Then
\beq \label{Lvf} 
L v(x,t)=Lu(x,t)- f_1(t)\le 0 \text{ in }\widetilde Q_{I'}.
\eeq 
For $(x,t)\in \overline{Q}_{t_*}\cup \mathcal S_{(t_*,T]}$, one has, see also \eqref{fpm}, 
    \beq \label{vuG}
    v(x,t)\le u(x,t)-G_1(T)\le u^+(x,t)-G_1(T) \le G_1(t)-G_1(T)\le0.
    \eeq 
   Thanks to \eqref{Lvf} and \eqref{vuG}, we can apply Theorem \ref{maxcor}  to $u:=v$, $t_1:=t_*$ and $t_2:=T$. It results in   
    \beqs 
    \max_{\overline{Q}_T} v^+\le \max_{\overline{Q}_{t_*}} v^+=0.
    \eeqs
(The last identity is due to \eqref{vuG} again.)   
        Thus, $v^+=0$ in $\overline{Q}_T$, which  implies $v\le 0$ in $\overline{Q}_T$, that is,
    \beqs 
    u(x,T)\le G_1(T)+\int_{t_*}^T f_1(\tau)\d\tau
    \text{ for all }x\in \mathbb P_1(\overline{Q}_{T}).
    \eeqs 
    The last inequality proves the inequality in \eqref{max2} for $t:=T$. Because $T$ is arbitrary in $I'$, and, clearly,  the inequality in \eqref{max2} holds for $t:=t_*$, we obtain \eqref{max2}.

(ii) By applying  part (i) to both $u$ and $(-u)$ and $f_1:=f$, and also using \eqref{ffrel}, we have
    \beq\label{maxpm}
        \max_{\overline{Q}_t} u^+,\max_{\overline{Q}_t} u^- \le \max_{\overline{Q}_{t_*}\cup \mathcal S_{(t_*,t]}} |u|+\int_{t_*}^t f(\tau)\d \tau \text{ for all }t\in I.
    \eeq
    Observe from \eqref{sup3} that
\beq \label{maxrel}
\max_{\overline{Q}_t} |u|=\max\left\{ \max_{\overline{Q}_t} u^+,\max_{\overline{Q}_t} u^-\right\}.
\eeq
Then combining the  estimates in \eqref{maxpm} with \eqref{maxrel} yields \eqref{maxabs}.
\end{proof}

\subsection{Growth Lemma}
Next, we establish a growth lemma for a subsolution of $L$ in bounded subsets of $\R^{n+1}$. 
\textit{For the remainder of this section, let $Q$ be a subset of $\R^{n+1}$ that satisfies  Assumption \ref{GGamcond}.}
Let $t_*\in \R$ and $T>0$ such that $I\eqdef [t_*,t_*+T]$ is a subset of $\mathbb P_2(\overline{Q})$. Denote $I'=I\setminus\{t_*\}=(t_*,t_*+T]$.

Let Assumptions \ref{secondA} be satisfied and assume  there is a constant $m_*\ge 0$ so that
\beq\label{TAB3}
 |b(x,t)|\le m_* \text{ for all }(x,t)\in \widetilde Q_{I'}.
\eeq

Referring to definitions \eqref{PP12}, \eqref{di2} and \eqref{Rs3}, let
\begin{align} \label{Vset}
V&=\mathbb P_1(\overline Q_{I})\text{ and  $x_*$ be a point in $\R^n\setminus V$},\\
 r&={\rm dist}(x_*,V)\text{ and }
  R=R_*(x_*,V).\notag
\end{align}
Define
 \beq\label{sTe1}
 \beta=\max\left \{ \frac{M_1+m_*R}{2c_0},\frac{R^2}{4c_0T}\right\}
 \text{ and } \eta_*=1-(r/R)^{2\beta}.
 \eeq

Since $\mathbb P_1$ is continuous and $\overline Q_{I}$ is non-empty and compact, the image set $V$ in \eqref{Vset} is non-empty and  compact. 
Together with the fact $x_*\notin V$, this implies $r>0$. 
Since $Q$ is open and non-empty, the set $V$ has non-empty interior which in turn implies that $R>r$, see \eqref{Rrpos}.
Consequently,
\beqs
\beta>0\text{ and }
\eta_*\in(0,1).
\eeqs 

The following Growth Lemma is a the counter part of \cite[Lemma 4.2]{HI3Rus}.

\begin{lemma}[Growth Lemma]\label{stgrowth}
Let  $u\in C(\overline{Q}_{I})\cap C_{x,t}^{2,1}(\widetilde Q_{I'})$ satisfy $Lu\le 0$ in $\widetilde Q_{I'}$ and $u\le 0$ on $\mathcal S_{I'}$.
Then one has
\beq \label{Tgrow3}
\max_{\overline{Q}_{t_*+T}} u^+\le \eta_* \max_{\overline{Q}_{t_*}} u^+.
\eeq 
\end{lemma}
\begin{proof} 
Observe that both sets $\widetilde Q_{I'}$ and $\mathcal S_{I'}$ are subsets of $\overline{Q}_{I'}$, and hence are subsets of $\overline{Q}_{I}$.
Without loss of the generality, assume $t_*=0$ so that $I=[0,T]$ and $I'=(0,T]$.

\medskip\noindent\textbf{Part A. Case $(0,T)\in\mathcal G(Q)$.}  We modify the proof of \cite[Lemma 4.2]{HI3Rus}.
Because $\overline{Q}_t\subset V\times\{t\}$ for all $t\in[0,T]$,  one has 
\beq\label{QV}
\overline Q_{I}=\bigcup_{t\in[0,T]} \overline{Q}_t\subset V\times[0,T]\text{ and }
\widetilde Q_{I'}\subset \overline Q_{I'}=\bigcup_{t\in(0,T]} \overline{Q}_t\subset V\times (0,T].
\eeq

\medskip\noindent\textit{Step 1.} 
Let $\mu=(4c_0)^{-1}$ and define  a function 
$\varphi(x)=\mu |x-x_*|^2$ for $x\in\R^n$.
Then $\varphi$ belongs to $C^\infty(\R^n)$ and satisfies
\begin{align}\label{genphi1}
&d_0\le \varphi\le d_1,\ 
|\nabla \varphi|\le d_2\text{ and }  
 c_0|\nabla \varphi|^2 =\varphi \text{ in $V$,}\\ 
\label{genphi2}
&|\langle A(x,t),D^2\varphi(x) \rangle|=2\mu |{\rm Tr}(A(x,t))|\le d_3  \text{ for all $(x,t)\in\widetilde Q_{I'}$,}
\end{align}
where
\beq \label{d03}
d_0=\mu r^2>0,\ 
d_1=\mu R^2>0,\ 
d_2=2\mu R>0\text{ and } 
d_3=2\mu M_1>0.
\eeq 
Moreover, thanks to \eqref{Aelip},
\beq
\label{genphi3}
( A\nabla\varphi)\cdot \nabla\varphi \ge c_0|\nabla \varphi|^2=\varphi \text{ in }\widetilde Q_{I'}.
\eeq 

Define a function $W$ on $\R^{n+1}$ by
\beq\label{Wdef}
W(x,t)=\begin{cases}
    t^{-\beta}e^{-\frac{\varphi(x)}{t}}& \text{ if } (x,t)\in \R^n\times(0,\infty),\\
    0,&\text{ if } (x,t)\in \R^n\times(-\infty,0].
\end{cases}
\eeq

On the one hand, it is obvious that
\beq \label{WC2}
W\in C^\infty (\R^n\times (0,\infty)).
\eeq 
On the other hand, because of the lower bound of $\varphi$ in \eqref{genphi1} and \eqref{d03}, and the formula of $W$ in \eqref{Wdef}, the function $W$ is continuous in  $V\times [0,\infty)$.
Together with the first property in \eqref{QV}, it implies that
\beq\label{WC0}
W\in C(\overline Q_{I}).
\eeq 

\medskip\noindent\textit{Step 2.} 
Elementary calculations shows, for $(x,t)\in \widetilde Q_{I'}$, that
\beq\label{LW}
L W(x,t)=\frac{W}{t^2}\left\{ t(-\beta+ \langle A,D^2\varphi \rangle - b\cdot \nabla \varphi) +\varphi   -(A\nabla \varphi)\cdot\nabla \varphi \right\}.
\eeq
Observe from \eqref{TAB3}, \eqref{genphi1} and \eqref{genphi2} that
\beqs %\label{sval}
|\langle A,D^2\varphi \rangle| + |b| |\nabla \varphi|
\le \beta_*\eqdef d_3 + m_* d_2=\frac{M_1+m_*R}{2c_0}\text{ in } \widetilde Q_{I'}.
\eeqs 
Together with the fact $\beta\ge \beta_*$ from \eqref{sTe1}, this implies
\beq\label{spineq}
\beta\ge  \langle A,D^2\varphi \rangle - b\cdot \nabla \varphi 
 \text{ in } \widetilde Q_{I'}.
\eeq
Utilizing inequalities \eqref{spineq} and \eqref{genphi3} in \eqref{LW} yields
 $L W\le 0$ in $\widetilde Q_{I'}$. 

\medskip\noindent\textit{Step 3.} 
Set $M=\max_{\overline Q_0} u^+$ and define the function
\beqs 
\widetilde W(x,t)=M(1-\eta W(x,t))\text{ in }\R^{n+1}, \text{ where }\eta=(d_0 e/\beta)^\beta>0.
\eeqs 
By the conclusion at the end of Step 2, one has 
\beq \label{LtilW}
L \widetilde W\ge 0\text{ in }\widetilde Q_{I'}.
\eeq 

Because $W=0$ on $\R^n\times\{0\}$, we have 
\beq\label{bdv1}
\widetilde W(x,0)=M\ge u(x,0)\text{ for all $x\in \overline{Q}_0$, which implies } \widetilde W\ge u \text{ on }\overline{Q}_0.
\eeq
Observe, for $t>0$ and $x\in V$,  that 
\beqs
\widetilde W(x,t)=M\left (1-\eta t^{-\beta} e^{-\varphi(x)/t}\right)\ge M\left(1-\eta t^{-\beta}e^{-d_0/t}\right).
\eeqs

Elementary calculations show that the function $h_0(t)=t^{-\beta}e^{-d_0/t}$ on $(0,\infty)$ attains  the maximum at $t_0=d_0/\beta$ with the value $h_0(t_0)=\eta^{-1}$.
Thus, one has
\beq\label{Lwtil}
\widetilde W(x,t)\ge M\left(1-\eta h_0(t_0)\right)=0\text{ for all } (x,t)\in V\times(0,\infty).
\eeq
By \eqref{Lwtil}, the facts $u\le 0$ on $\mathcal S_{I'}$ and $\mathcal S_{I'}\subset \overline Q_{I}\times(0,\infty)\subset V\times(0,\infty)$, we imply
\beq\label{bdv2}
\widetilde W\ge u \text{ on }\mathcal S_{I'}.
\eeq

\medskip\noindent\textit{Step 4.} 
Define $v=u-\widetilde W$. By the properties \eqref{WC2} and \eqref{WC0} of $W$, we have  $v\in C(\overline{Q}_{I})\cap  C_{x,t}^{2,1}( \widetilde Q_{I'})$.
By \eqref{LtilW}, we have $Lv\le 0$ in $\widetilde Q_{I'}$. Thanks to \eqref{bdv1} and \eqref{bdv2}, $v\le 0$ on $\overline{Q}_0$ and $\mathcal S_{I'}$.
Applying  Theorem \ref{maxcom}  to the function $v$ yields
\beqs
\max_{\overline{Q}_{t}} v^+\le \max_{\overline{Q}_0} v^+=0\text{ for all $t\in I$.}
\eeqs
Therefore, 
\beq\label{Ww}
u\le \widetilde W\text{ in } \overline{Q}_{I}.
\eeq

For $(x,t)\in V\times (0,T]$, one has from the upper bound of $\varphi(x)$ in \eqref{genphi1} that
\beq\label{We}
\widetilde W(x,t)\le M\left(1-\eta t^{-\beta} e^{-d_1/t}\right).
\eeq

Set $T_*=R^2/(4c_0\beta)=d_1/\beta\in(0,T]$.
It follows from \eqref{Ww} and \eqref{We}, for all $(x,T_*)\in \overline{Q}_{T_*}$, that 
\beqs
u(x,T_*)\le \widetilde W(x,T_*)
\le M\left[1-\left(\frac{d_0e}{\beta}\right )^\beta \left(\frac{d_1}{\beta}\right)^{-\beta} e^{-d_1(\beta/d_1)}\right]
=M\left[1-\left(\frac{d_0}{d_1}\right)^{\beta}\right]=\eta_* M.
\eeqs
Thus, we obtain the inequality 
\beq \label{Tgrow}
 \max_{\overline{Q}_{T_*}} u^+\le \eta_* \max_{\overline{Q}_0} u^+.
 \eeq

If $T_*=T$, then inequality \eqref{Tgrow3} immediately follows from \eqref{Tgrow}.
 In the case $T_*<T$, by applying Theorem \ref{maxcor} to $t_1:=T_*$ and $t_2=T$ and then combining its inequality \eqref{maxcom} with \eqref{Tgrow}, we obtain
\beqs 
\max_{\overline{Q}_{T}} u^+\le  \max_{\overline{Q}_{T_*}} u^+\le \eta_* \max_{\overline{Q}_0} u^+,
\eeqs 
which yields inequality  \eqref{Tgrow3}.

\medskip\noindent\textbf{Part B. Case $(0,T)\not\in\mathcal G(Q)$.}  Let  $(\tau_k)_{k=1}^\infty$ be the sequence in \ref{asb} of Assumption \ref{GGamcond}\ref{asv} for $t_1=0$ and $t_2=T$.
For $k\ge 1$, define
\begin{align*} 
V_k&=\mathbb P_1(\overline Q_{[t_*,\tau_k]}), \
 r_k={\rm dist}(x_*,V_k),\  
  R_k=R_*(x_*,V_k),\\
 \beta_k&=\max\left \{ \frac{M_1+m_*R_k}{2c_0},\frac{R_k^2}{4c_0\tau_k}\right\}
 \text{ and } \eta_{*,k}=1-(r_k/R_k)^{2\beta_k}.
 \end{align*}
For $k\ge 1$, we have $V_k\subset V$ which implies $x_*\in \R^n\setminus V_k$. Consequently, $r_k\ge r$ and $R_k\le R$, which in turn yield
 \begin{align*}
 \beta_k\le \bar \beta_k \eqdef \max\left \{ \frac{M_1+m_*R}{2c_0},\frac{R^2}{4c_0\tau_k}\right\}
 \text{ and } 
\eta_{*,k}\le \bar \eta_{*,k}\eqdef 1-(r/R)^{2\bar \beta_k}.
 \end{align*}
For $k\ge 1$, applying the result in Part A to $T:=\tau_k$ gives 
\beq\label{Tgtau}
 \max_{\overline{Q}_{\tau_k}} u^+\le \eta_{*,k} \max_{\overline{Q}_0} u^+
 \le \bar\eta_{*,k} \max_{\overline{Q}_0} u^+.
 \eeq
 
 Given $\varep>0$, we apply inequality \eqref{QQ0} to $t_2:=T$ to have, for sufficiently large $k$,
 \beq \label{Tge}
 \max_{\overline{Q}_{T}} u^+\le \max_{\overline{Q}_{\tau_k}} u^++\varep.
 \eeq
Combining \eqref{Tgtau} with \eqref{Tge} yields
\beq \label{Tgtau2}
 \max_{\overline{Q}_T} u^+\le \bar\eta_{*,k} \max_{\overline{Q}_0} u^++\varep.
 \eeq
 Note that $\bar\beta_k\to\beta$ and $\bar\eta_{*,k}\to\eta_*$.
By passing $k\to\infty$ and then $\varep\to 0$ in \eqref{Tgtau2}, we obtain \eqref{Tgrow3}.
\end{proof}

For inhomogeneous problems, we obtain the following counter part of \cite[Lemma 4.4]{HI4} with a similar proof.

\begin{theorem}\label{lemG2}
Let  $u$ belong to $C(\overline{Q}_{I})\cap C_{x,t}^{2,1}(\widetilde Q_{I'})$.
Suppose  there is a function $F\in C(I,[0,\infty))$ such that 
    $Lu(x,t)\le F(t)$ for all $(x,t)\in \widetilde Q_{I'}$.     
Then one has
 \beq \label{Tgrow4}
\max_{\overline{Q}_{t_*+T}} u^+
\le \eta_* \max_{\overline{Q}_{t_*}} u^+
+\sup_{\mathcal S_{(t_*,t_*+T]}} u^+ +\int_{t_*}^{t_*+T}F(t)\d t.
\eeq 
\end{theorem} 
\begin{proof}
Thanks to \eqref{OKbQ}, the maxima and supremum in \eqref{Tgrow4} exist.
Denote 
$$\delta=\sup_{\mathcal S_{(t_*,t_*+T]}} u^+,\quad D=\int_{t_*}^{t_*+T}F(t)\d t,$$
and define the function 
\beqs 
v(x,t)=u(x,t)-\delta-\int_{t_*}^t F(\tau)\d\tau \text{ for }(x,t)\in \overline Q_{I}.
\eeqs 
Then $v\in C(\overline{Q}_{I})\cap C_{x,t}^{2,1}(\widetilde Q_{I'})$ and  
\beq \label{GLv}
L v=Lu -F(t)\le 0\text{ in }\widetilde Q_{I'}.
\eeq 
Moreover, for the boundary values of $v$, one has
\beqs%\label{Gvgam}
v(x,t)\le u(x,t)-\delta\le u^+(x,t)-\delta\le 0
\text{ for all } (x,t)\in\mathcal S_{(t_*,t_*+T]}.
\eeqs 
By Lemma \ref{stgrowth} applied to the subsolution $v$ of $L$, see \eqref{GLv}, we have 
\beq\label{Gvmax1}
\max_{\overline{Q}_{t_*+T}} v^+\le \eta_* \max_{\overline{Q}_{t_*}} v^+.
\eeq
For $(x,t_*)\in \overline{Q}_{t_*}$, one has  $v(x,t_*)= u(x,t_*)-\delta$, 
hence, 
\beq\label{GvJ}
\max_{\overline{Q}_{t_*}} v^+\le \max_{\overline{Q}_{t_*}} u^+.
\eeq 
For $(x,t_*+T)\in \overline{Q}_{t_*+T}$, one has  
\beqs 
u(x,t_*+T)=v(x,t_*+T)+\delta+D\le \max_{\overline{Q}_{t_*+T}} v^+ +\delta +D.
\eeqs
Thus,
\beq\label{GJb}
\max_{\overline{Q}_{t_*+T}} u^+\le  \max_{\overline{Q}_{t_*+T}} v^+ +\delta+D.
\eeq
Applying inequality \eqref{Gvmax1} to estimate the  maximum on the right-hand side of \eqref{GJb}, and then utilizing inequality \eqref{GvJ},  we obtain 
$$\max_{\overline{Q}_{t_*+T}} u^+ \le \eta_* \max_{\overline{Q}_{t_*}} v^+ +\delta+D
\le \eta_* \max_{\overline{Q}_{t_*}} u^+ +\delta+D,$$ 
which proves \eqref{Tgrow4}.
\end{proof}

\section{Homogeneous problems}\label{homsec}

Throughout this section, $Q$ is  a subset of $\R^n\times(0,\infty)$ that satisfies Assumption \ref{GGamcond} with $\mathbb P_2(\overline{Q})=[0,\infty)$.
Given two functions  $A:\widetilde Q\to \mathcal M^{n\times n}_{{\rm sym}}$ and $b:\widetilde Q\to \R^n$, define  the linear operator $L$ on $C_{x,t}^{2,1}(\widetilde Q)$ by \eqref{Ltil}.
When $b\equiv 0$, the operator $L$ becomes 
\beq \label{Lz}
 L_0 v = v_t-\langle A(x,t),D^2 v\rangle \text{ for }v\in C_{x,t}^{2,1}(\widetilde Q).
\eeq

We first obtain a very general estimate in the case of zero drift, i.e., $b\equiv 0$.

\begin{theorem}\label{STthm0}
Suppose $A(x,t)$ is bounded in $\widetilde Q$ and satisfies Assumption \ref{secondA}\ref{A1}. Then there exist $C_0\ge 1$ and  $\beta>0$ such that 
the following statement holds true.
If $u\in C(\overline{Q})\cap C_{x,t}^{2,1}(\widetilde Q)$ satisfies $L_0 u\le 0$ in $\widetilde Q$ and $u\le 0$ on $\Gamma(Q)\setminus \overline{Q}_0$, then 
  \begin{equation}\label{beta-est}
        \max_{\overline{Q}_t} u^+\le C_0\left(\max_{\overline{Q}_0} u^+\right) (t+1)^{-\beta} \text{ for all }t\ge 0.
    \end{equation}
\end{theorem}
\begin{proof}
Let
\beq\label{mucond1}
M_*=\sup_{(x,t)\in \widetilde Q} \|A(x,t)\|_{\rm op}, \  
\mu =\frac1{4M_*} \text{ and }
\beta= 2nc_0\mu=\frac{nc_0}{2M_*}, 
\eeq
 where  $\|\cdot\|_{\rm op}$ denotes the norm of a bounded linear mapping.
Then $M_*$, $\mu$, $\beta$ are positive numbers.
Define 
\beqs
    W(x,t)=(t+1)^{-\beta}e^{-\mu\frac{|x|^2}{t+1}}
    \text{ for } (x,t)\in \R^n\times[0,\infty).
\eeqs
Then $W\in C^\infty (\R^n\times[0,\infty))$.
Using similar calculations to \eqref{LW} with $b=0$ and $\varphi(x)=\mu |x|^2$, we have
\beqs
   L_0 W(x,t)
= \frac{W_{\beta,\mu}(x,t)}{(t+1)^2}
   \left\{(t+1)\left[2 \mu{\rm Tr}(A(x,t))-\beta\right]
+\mu\left[|x|^2-4\mu x^{\rm T}A(x,t)x\right]\right\}.
\eeqs 
Because  ${\rm Tr}(A)\ge n c_0$, thanks to \eqref{Aelip}, and 
$x^{\rm T}Ax\le M_* |x|^2$, we have from the choice of $\mu$ and $\beta$ in \eqref{mucond1} that
 \beq \label{superW}
 L_0  W(x,t)\geq 0 \text{ in }\widetilde Q.
 \eeq
Let $R=\max\{|x|:x\in \mathbb P_1(\overline Q_0)\}$ and define
\beq\label{vWdef} 
v(x,t)=C_0\left (\max_{\overline{Q}_0} u^+ \right)W(x,t), 
\text{ where } 
C_0=\exp(\mu  R^2).
\eeq 
By \eqref{superW}, we have $Lv\ge 0\ge Lu$ in $\widetilde Q$. Moreover,  $v\ge 0\ge u$ on $\Gamma(Q)\setminus \overline{Q}_0$, and, for $(x,0)\in \overline{Q}_0$, 
\beq\label{uv0}
v(x,0)
=C_0  e^{-\mu|x|^2}\max_{\overline{Q}_0} u^+
\ge C_0  e^{-\mu R^2}\max_{\overline{Q}_0} u^+
= \max_{\overline{Q}_0} u^+\geq u(x,0).
\eeq
Let $w=u-v$. Then $Lw\le 0$ in $\widetilde Q$ and $w\le 0$ on $\Gamma(Q)
$.
Applying Theorem \ref{maxcor} and using \eqref{uv0} give, for any $t\ge 0$,
\beqs
\max_{\overline{Q}_t} w^+\le \max_{\overline{Q}_0} w^+=0. 
\eeqs
Thus,  $w^+(x,t)=0$ in $\overline Q$, which implies 
\beqs
u(x,t)\le v(x,t)\le C_0\left (\max_{\overline{Q}_0} u^+ \right) (t+1)^{-\beta} \text{ in }\overline Q.
\eeqs
(The last inequality comes from the use of the simple inequality $W(x,t)\le (t+1)^{-\beta}$  to estimate $v(x,t)$ given in \eqref{vWdef}.)
Therefore, we obtain  \eqref{beta-est}.
\end{proof}

Note that $\beta>0$ in the above Theorem \ref{STthm0} can be very small.
In order to  obtain better decaying estimates than \eqref{beta-est}, we will utilize the Growth Lemma -- Lemma \ref{stgrowth}.
We will impose a condition on $Q$ of the following form 
\beq\label{diamcond}
{\rm diam} \left(\mathbb P_1\left(\bar  Q_{[\Phi(t),t]}\right)\right)\le  \psi(t)
\text{ for all }t\ge T,
\eeq 
where $T\ge 0$ is some number,  $0\le \Phi(t)\le t$ and $\psi(t)\ge 0$. 
The condition \eqref{diamcond}  means that, for each $t$, the diameter of $\overline{Q}_\tau$ can be controlled by $\psi(t)$ uniformly in $\tau\in[\Phi(t),t]$.

If $Q$ satisfies 
\beq \label{dincrease}
Q_t\subset Q_s \text{ for }t<s,
\eeq
then  \eqref{diamcond} becomes
\beq\label{dcond2}
{\rm diam}\left(\overline{Q}_t\right)\le  \psi(t). 
\eeq 
Even more general than \eqref{dincrease} and \eqref{dcond2} is the following situation. If there is a family of sets $\mathcal Q_t$ in $\R^{n+1}$, for $t\ge T$, such that 
\beq\label{bigQ}
Q_t\subset \mathcal Q_t \text{ for all $t$, and } \mathcal Q_t \subset \mathcal Q_s \text{ for all $t<s$,}
\eeq
then  a sufficient condition for  \eqref{diamcond} is
\beq\label{dcmQ}
{\rm diam}\left(\overline{\mathcal{Q}}_t\right)\le  \psi(t). 
\eeq 

\subsection{Iteration scheme}\label{scheme}
We describe the general scheme for obtaining estimates for a subsolution $u$ of $L$.
We fix an appropriate number $T_0\ge T$, choose $\tau_k>0$ for $k\ge 1$, and set
\beq \label{Tksum}
T_k=T_0+\sum_{j=1}^k \tau_j.
\eeq 
For $k\ge 1$, let 
   \begin{align}\label{Vdk}
   V_k&=\mathbb P_1 (\bar  Q_{[T_{k-1},T_k]}),\quad d_k={\rm diam}(V_k),\\
   \label{rRk0}
r_k&>0\text{  and } R_k=r_k+d_k.
    \end{align}
By the virtue of Lemma \ref{choicey}, there is $x_k\in \R^n\setminus V_k$ such that
\beq \label{xchoice}
{\rm dist}(x_k,V_k)=r_k\text{ and }R_*(x_k,V_k)=R_k.
\eeq 

Under Assumption \ref{secondA}, assume, for $k\ge 1$, that there is a constant $m_k\ge 0$ so that
\beq\label{TABk}
 |b(x,t)|\le m_k \text{ on }\widetilde Q_{(T_{k-1},T_k]}.
\eeq
Taking into account the formulas in \eqref{sTe1} together with the choice \eqref{xchoice} and the bound \eqref{TABk}, we define
\beq\label{sTek0}
 \beta_k=\frac{1}{2c_0}\max\left \{M_1+m_k R_k,\frac{R_k^2}{2\tau_k}\right\}, \quad 
 \eta_k=1-\left(\frac{r_k}{R_k}\right)^{2\beta_k}=1-\left(\frac{1}{1+d_k/r_k}\right)^{2\beta_k}.
 \eeq
 
 Suppose $u\in C(\overline{Q})\cap C_{x,t}^{2,1}(\widetilde Q)$ satisfies $L u\le 0$ in $\widetilde Q$ and $u\le 0$ on $\Gamma(Q)\setminus \overline{Q}_0$. 
 Given an integer $k\ge 1$, let $E:=Q_{(T_{k-1},T_k)}$. By Lemma \ref{rmk2}\ref{pd}, the sets $\overline{E}_{T_{k-1}}$ and $\overline{E}_{T_{k}}$ are non-empty.
  Applying Lemma \ref{stgrowth} to the interval $I:=[T_{k-1},T_k]$, we have from \eqref{Tgrow3} that
 \beq \label{Tgrowk}
\max_{\overline{Q}_{T_{k}}} u^+\le \eta_k \max_{\overline{Q}_{T_{k-1}}} u^+.
\eeq
Iterating inequality \eqref{Tgrowk} gives, for any $k>k_*\ge 0$,
 \beq \label{ugk0}
\max_{\overline{Q}_{T_{k}}} u^+\le \eta_{k_*+1}\eta_{k_*+2}\ldots \eta_k \max_{\overline{Q}_{T_{k_*}}} u^+.
\eeq
In particular, when $k_*=0$ and $k\ge 1$ one has
 \beq \label{ugrowk}
\max_{\overline{Q}_{T_{k}}} u^+\le \eta_1\eta_2\ldots \eta_k \max_{\overline{Q}_{T_{0}}} u^+.
\eeq

More specific estimates of $d_k$, $R_k$ and $\beta_k$, $\eta_k$ will be obtained thanks to \eqref{diamcond} and later choices of $\tau_k$ and $r_k$.
They will result in more specific estimates in \eqref{ugrowk}, and later for 
$\max_{\overline{Q}_t} u^+$ with arbitrary real number $t\ge T$.
Below we investigate the  zero drift case and general bounded drift case separately.

\subsection{Zero drift}\label{HZsec}
We focus on the case of zero drift and obtain faster decaying estimates than \eqref{beta-est}. 
However, it requires some restriction on the growth of the spatial diameters, when $t\to\infty$, as mentioned in \eqref{diamcond}.

\begin{theorem}\label{STthm1}
Under Assumption \ref{secondA}, suppose $u\in C(\overline{Q})\cap C_{x,t}^{2,1}(\widetilde Q)$ satisfies $L_0 u\le 0$ in $\widetilde Q$ and $u\le 0$ on $\Gamma(Q)\setminus \overline{Q}_0$. 
\begin{enumerate}[label=\tnum]
    \item\label{st4i} Given $p\in[0,1/2)$, assume there exist numbers $\delta_*>0$, $T\ge \delta_*^\frac{1}{1-2p}$ and $K>0$ such that 
    \beq\label{dcond3}
{\rm diam}\left(\mathbb P_1\left(\overline{Q}_{\left[t-\delta_* t^{2p},t\right]}\right)\right)\le  K t^p \text{ for all }t\ge T.
\eeq 
Then
\beq\label{udecay2}
\max_{\overline{Q}_t}u^+\le C\left(\max_{\overline{Q}_T}u^+\right) \exp\left(-\nu t^{1-2p}\right) \text{ for all }t\ge T, 
\eeq
where $C$ and $\nu$ are positive constants independent  of $u$.
    \item\label{st4ii} Assume  there exist $\delta_*\in(0,1)$, $T>0$ and $K>0$ such that
    \beq\label{dcond4}
{\rm diam}\left(\mathbb P_1\left(\bar  Q_{[\delta_* t,t]}\right)\right)\le  K t^{1/2} \text{ for all }t\ge T.
\eeq 
Then
\beq\label{udecay3}
\max_{\overline{Q}_t}u^+\le \left(\max_{\overline{Q}_T}u^+\right) (t/\mu)^{-\theta} \text{ for all }t\ge T, 
\eeq
where $\mu>0$ depends on $T$ and $\delta_*$, while $\theta>0$  depends on $K$, $c_0$, $M_1$, $T$, $\delta_*$.
If the last four parameters are fixed, then 
$\theta=\theta(K)$
satisfies
\beq \label{thelim}
\lim_{K\to 0} \theta(K)=\infty.
\eeq 
More specifically, there is $K_0\in(0,1)$ depending on $c_0$, $M_1$, $T$ and $\delta_*$ such that if $K\le K_0$ then 
\beq\label{thest}
\theta(K)\ge \frac{\ln K}{4\ln \delta_*}.
\eeq 
\end{enumerate}
\end{theorem}
\begin{proof}
For $k\ge 1$, let 
   \beq
   \label{rRk}
r_k=Nd_k \text{ for some positive number $N$. }
    \eeq
In this case, $m_*=0$ in \eqref{TAB3} and $m_k=0$ in \eqref{TABk}. Hence \eqref{sTek0} becomes
\beq\label{bez}
 \beta_k=\frac{1}{2c_0}\max\left \{M_1,\frac{R_k^2}{2\tau_k}\right\},\quad 
  \eta_k=1-\left(\frac{1}{1+1/N}\right)^{2\beta_k}.
\eeq

\medskip
\ref{st4i} 
Given $p\in[0,1/2)$, set 
    \beq\label{sigdef}
    \sigma=\frac{2p}{1-2p}>0 \text{ and } 
\delta_0=[\delta_*(1-2p)^{2p}]^\frac{1}{1-2p}>0.
\eeq
Then we have
\beq\label{d0def}
2p(\sigma+1)=\sigma \text{ and }  \delta_*=\delta_0^\frac1{\sigma+1} (\sigma+1)^\frac\sigma{\sigma+1} .
\eeq

Take 
\beq \label{Ttau}
T_0=T\text{ and   }\tau_k=\delta_0 k^\sigma\text{ for }k\ge 1.
\eeq 
From \eqref{Tksum}, we estimate $T_k=T_0+ \delta_0  \sum_{j=1}^kj^\sigma $, for $k\ge 1$, from above and below  by
\begin{align*}
T_k&\le T_0+\delta_0  \sum_{j=1}^k \int_j^{j+1} \tau^\sigma\d\tau = T_0+\delta_0 \int_{1}^{k+1} \tau^\sigma\d\tau,\\
T_k&\ge T_0+\delta_0  \sum_{j=1}^k \int_{j-1}^{j} \tau^\sigma\d\tau = T_0+\delta_0 \int_{0}^{k} \tau^\sigma\d\tau.
\end{align*}
Correspondingly, we obtain
\begin{align}\label{tk1}
T_k&\le T_0+\frac{\delta_0 }{\sigma+1}(k+1)^{\sigma+1},\\
\label{tk2}
T_k&\ge T_0+\frac{\delta_0 }{\sigma+1}k^{\sigma+1}.
\end{align}
Note that \eqref{tk1} and \eqref{tk2} also hold for $k=0$.
Thus,  it follows from \eqref{tk1} and \eqref{tk2} that
\beq\label{Tkk0}
\left[\frac{\sigma+1}{\delta_0}(T_k-T_0)\right]^\frac1{\sigma+1} -1 \le k\le \left[\frac{\sigma+1}{\delta_0}(T_k-T_0)\right]^\frac1{\sigma+1}
\text{ for all integer } k\ge 0.
\eeq
Using the second inequality in \eqref{Tkk0} and then \eqref{sigdef}, \eqref{d0def}, we have, for all integer $k\ge 1$, 
\beq\label{tau01}
\tau_k=\delta_0 k^\sigma \le \delta_0 \left(\frac{\sigma+1}{\delta_0}T_k\right)^\frac\sigma{\sigma+1}
=\delta_0^\frac1{\sigma+1} (\sigma+1)^\frac\sigma{\sigma+1} T_k^\frac\sigma{\sigma+1}.
\eeq

Now, utilizing specific values in \eqref{d0def}, we have from \eqref{tau01}, for $k\ge 1$, 
\beqs %\label{tau1}
\tau_k\le \delta_* T_k^{2p}.
\eeqs
This implies 
\beq\label{tau2}
[T_{k-1},T_k]=[T_k-\tau_k,T_k]\subset [T_k-\delta_* T_k^{2p},T_k].
\eeq
Since $T_k\ge T\ge \delta_*^\frac1{1-2p}$ we have $T_k-\delta_* T_k^{2p}\ge 0$.
To estimate $d_k$ in \eqref{Vdk}, we use \eqref{tau2},  \eqref{dcond3} and \eqref{tk1} to obtain, for $k\ge 1$,
\begin{align*}
d_k&\le {\rm diam}\left(\bar  Q_{[T_k-\delta_* T_k^{2p},T_k]}\right)
\le K T_k^p
\le K\left[T_0+\frac{\delta_0}{\sigma+1}(k+1)^{\sigma+1}\right]^p\\
&\le K\left[\left(1+\frac{\delta_0}{\sigma+1}\right)(T_0+k+1)^{\sigma+1}\right]^p 
= K\left(1+\frac{\delta_0}{\sigma+1}\right)^p(k+T_0+1)^{(\sigma+1)p}.
\end{align*} 
With $R_k$ in \eqref{rRk0}, selecting $N=1$ in \eqref{rRk} yields
\beqs 
R_k=2d_k\le 2K\left(1+\frac{\delta_0}{\sigma+1}\right)^p(k+T_0+1)^{(\sigma+1)p}.
\eeqs 
It follows from this estimate of $R_k$, the choice of $\tau_k$ in \eqref{Ttau}  and the first identity in \eqref{d0def} that 
\begin{align*}
    \frac{R_k^2}{2\tau_k}&\le \frac{4K^2\left(1+\frac{\delta_0}{\sigma+1}\right)^{2p}(k+T_0+1)^{2p(\sigma+1)}}{2\delta_0 k^\sigma}
    = \frac{2K^2}{\delta_0}\left(1+\frac{\delta_0}{\sigma+1}\right)^{2p}\frac{(k+T_0+1)^\sigma}{k^\sigma}\\
    &=\frac{2K^2}{\delta_0}\left(1+\frac{\delta_0}{\sigma+1}\right)^{2p}\left(1+\frac{T_0+1}{k}\right)^\sigma.
    \end{align*}
With $k\ge 1$, this yields
\beq\label{R2t0}
    \frac{R_k^2}{2\tau_k}
    \le c_1\eqdef \frac{2K^2}{\delta_0}\left(1+\frac{\delta_0}{\sigma+1}\right)^{2p}(T_0+2)^\sigma.
\eeq
Combining \eqref{R2t0} with the definitions of $\beta_k$ and $\eta_k$ in \eqref{bez}, we have
\beq\label{the1}
\beta_k\le \bar\beta\eqdef \frac{1}{2c_0}\max\left\{M_1,c_1\right\}\text{ and }
    \eta_k\le \bar\eta\eqdef 1-(1/2)^{2\bar\beta}.
\eeq
By \eqref{ugrowk} and \eqref{the1},
    \beq\label{exkdec}
    \max_{\overline{Q}_{T_k}} u^+\le \left (\max_{\overline{Q}_{T_0}} u^+\right) \bar \eta^k=\left (\max_{\overline{Q}_{T_0}} u^+\right)  e^{-\nu_0 k},\text{ where } \nu_0=-\ln \bar \eta >0.
    \eeq
    Clearly, \eqref{exkdec} holds also for $k=0$.
    
    Consider $t\ge T_0$. Then there is an integer $k\ge 1$ such that $t\in[T_{k-1},T_{k})$. 
By Theorem \ref{maxcor}, estimate \eqref{exkdec} for $k:=k-1\ge 0$, and the first inequality in  \eqref{Tkk0}, we have
    \begin{align*}
    \max_{\overline{Q}_{t}} u^+
    &\le \max_{\overline{Q}_{T_{k-1}}} u^+
    \le \left( \max_{\overline{Q}_{T_0}} u^+\right)e^{-\nu_0 (k-1)}
    =\left( \max_{\overline{Q}_{T_0}} u^+\right)e^{2\nu_0 }e^{-\nu_0 (k+1)} \\
    &\le \left( \max_{\overline{Q}_{T_0}} u^+\right)e^{2\nu_0 } \exp\left\{ -\nu_0\left[\frac{\sigma+1}{\delta_0}(T_{k}-T_0)\right]^\frac1{\sigma+1}\right\},
    \end{align*}
    thus,
       \beq\label{uQt0}
    \max_{\overline{Q}_{t}} u^+
\le \left( \max_{\overline{Q}_{T_0}} u^+\right)e^{2\nu_0 } \exp\left\{ -\nu(t-T_0)^\frac1{\sigma+1} \right\},
\text{    where }
\nu=\nu_0\left(\frac{\sigma+1}{\delta_0}\right)^\frac1{\sigma+1}.
    \eeq 
Because $1/(\sigma+1)\in(0,1)$ and $t\ge T_0$,  we have
    \beq\label{tT0}
    (t-T_0)^\frac1{\sigma+1}\ge t^\frac1{\sigma+1}-T_0^\frac1{\sigma+1}.
    \eeq
Combining \eqref{uQt0} with \eqref{tT0} yields
    \beq\label{uQtest}
    \max_{\overline{Q}_{t}} u^+\le C\left( \max_{\overline{Q}_{T_0}} u^+\right)\exp\left(-\nu t^\frac1{\sigma+1}\right),
   \text{ where }C=\exp(\nu T_0^\frac1{\sigma+1}+2\nu_0).
    \eeq
   Noting that $1/(\sigma+1)=1-2p$ and recalling $T_0=T$,  we obtain \eqref{udecay2} from \eqref{uQtest}.

\medskip    
\ref{st4ii}
Set $\lambda=1/\delta_*>1$. Take 
\beq \label{Ttau2}
T_0=\max\{T, \lambda/(\lambda-1)\}>0\text{ and }\tau_k= \lambda^k \text{ for }k\ge 1.
\eeq 
Then $T_k$ in \eqref{Tksum} becomes
\beq\label{tk3}
T_k=T_0+ \frac{\lambda^{k+1}-\lambda}{\lambda-1}\text{ which gives }
\lambda^{k+1}=(T_k-T_0)(\lambda-1) +\lambda.
\eeq
Notice that $T_0 (\lambda-1)\ge \lambda$.
We have from the last identity in \eqref{tk3} that
\beq\label{tk5}
(T_k-T_0)(\lambda-1) \le \lambda^{k+1}\le T_k(\lambda-1) .
\eeq
The last inequality yields $\lambda^k\le T_k(\lambda-1)/\lambda$, which we use to estimate
$$T_{k-1}=T_k- \lambda^k\ge T_k-T_k (\lambda-1)/\lambda=T_k/\lambda=\delta_* T_k.$$ 
Therefore, 
\beq \label{tau3}
[T_{k-1},T_k]\subset [\delta_* T_k,T_k].
\eeq 
From \eqref{Vdk}, \eqref{tau3} and \eqref{dcond4}, we have 
\beqs
d_k\le {\rm diam}\left(\mathbb P_1\left(\overline{Q}_{[\delta_* T_k,T_k]}\right)\right)
\le K T_k^{1/2}.
\eeqs 
Then we can estimate $R_k$ from \eqref{rRk0} and \eqref{rRk} by
\beq \label{Rk1}
R_k=(1+N)d_k\le (1+N)K T_k^{1/2}.
\eeq 
Combining \eqref{Rk1} with the formulas of $\tau_k$ and $T_k$ in \eqref{tk3} and \eqref{Ttau}, respectively, yields
    \beq\label{prec2}
    \frac{R_k^2}{2\tau_k}
    \le \frac{(1+N)^2K^2T_k}{2\lambda^k}
    = \frac{(1+N)^2K^2\left(T_0+\frac{\lambda^{k+1}-\lambda}{\lambda-1}\right)}{2\lambda^k}
    \le \frac12(1+N)^2K^2\left (\frac{T_0}{\lambda^k}+\frac{\lambda}{\lambda-1}\right).
    \eeq
Simply using $\lambda^k\ge 1$, we obtain
    \beq\label{c2e}
    \frac{R_k^2}{2\tau_k}
        \le c_2\eqdef \frac12(1+N)^2K^2\left (T_0+\frac{\lambda}{\lambda-1}\right).
    \eeq 
We have from this upper bound \eqref{c2e} of $R_k^2/(2\tau_k)$ and \eqref{bez} that
\beq\label{the2}
\beta_k\le \bar\beta\eqdef \frac{1}{2c_0}\max\left\{M_1,c_2\right\},\quad 
    \eta_k\le \bar\eta\eqdef 1-\left(\frac{1}{1+1/N}\right)^{2\bar\beta}\in(0,1).
    \eeq
By \eqref{ugrowk} again
    \beq\label{exkdec7}
    \max_{\overline{Q}_{T_k}} u^+\le \left (\max_{\overline{Q}_{T_0}} u^+\right) \bar \eta^k=\left (\max_{\overline{Q}_{T_0}} u^+\right)  \lambda^{-\nu k},\text{ where } \nu=\frac{\ln(1/\bar \eta)}{\ln \lambda}>0.
    \eeq
    Obviously, inequality \eqref{exkdec7} holds also for $k=0$.
    
For $t> T_0$, there is an integer $k\ge 1$ such that $t\in(T_{k-1},T_{k}]$. By Theorem  \ref{maxcor}, estimate \eqref{exkdec7} for $k:=k-1\ge 0$ and the first inequality in  \eqref{tk5}, we obtain
    \begin{align*}
    \max_{\overline{Q}_{t}} u^+
    &\le \max_{\overline{Q}_{T_{k-1}}} u^+ \le \left (\max_{\overline{Q}_{T_0}} u^+\right)\lambda^{-\nu (k-1)}
    =\left (\max_{\overline{Q}_{T_0}} u^+\right)\lambda^{2\nu}\lambda^{-\nu (k+1)}\\
    &\le  \left (\max_{\overline{Q}_{T_0}} u^+\right)\lambda^{2\nu}\left\{ (T_{k}-T_0)(\lambda-1)\right\}^{-\nu}.
    \end{align*}
With $T_k\ge t>T_0$, it follows that
\beq\label{uQ0}
    \max_{\overline{Q}_{t}} u^+
    \le \left (\max_{\overline{Q}_{T_0}} u^+\right)\lambda^{2\nu}\left\{ (t-T_0)(\lambda-1)\right\}^{-\nu}.
\eeq
 Let 
 \beq\label{Tstar}
 T_*=2T_0=2\max\left\{T,\frac{\lambda}{\lambda-1}\right\}.
 \eeq
 For $t\ge T_*>T_0$,  we have $t-T_0\ge t/2$, and, hence, by \eqref{uQ0} and the fact  $T_0\ge T$ in \eqref{Ttau2} and Theorem \ref{maxcor},  
    \begin{equation}\label{uQ2}
    \max_{\overline{Q}_{t}} u^+
    \le   \left(\max_{\overline{Q}_{T_0}} u^+\right)\lambda^{2\nu}\left(\frac{t(\lambda-1)}{2}\right)^{-\nu}
    \le      \left(\max_{\overline{Q}_{T}} u^+\right) \left(\frac{2\lambda^2}{\lambda-1}\right)^{\nu}t^{-\nu}.
    \end{equation}
For $t\in[T,T_*]$, we have, by Theorem \ref{maxcor} again, 
   \begin{equation}\label{uQ3}
    \max_{\overline{Q}_{t}} u^+
    \le   \left(\max_{\overline{Q}_{T}} u^+\right)
    \le   \left(\max_{\overline{Q}_{T}} u^+\right) t^{-\nu} T_*^{\nu}.
    \end{equation}

Take $N=K^{-3/2}$ now. From \eqref{uQ2} and \eqref{uQ3}, we obtain inequality \eqref{udecay3} where, recalling $\nu$ defined in \eqref{exkdec7},  and $T_*$ defined by \eqref{Ttau2} and \eqref{Tstar},
\beq\label{themu}
    \theta=\nu=-\frac{\ln \bar \eta}{\ln \lambda} \text{  and } \mu =\max\left\{\frac{2\lambda^2}{\lambda-1},T_*\right\}=2\max\left\{T,\frac{\lambda^2}{\lambda-1}\right\}.
\eeq 

Next, we prove \eqref{thelim}. Consider $K\in(0,1)$. We have 
\beq \label{NK1}
(1+N)K>NK = 1/\sqrt K\ge 1.
\eeq 
We further estimate $\bar\beta$ and $\bar\eta$ in \eqref{the2} by
\begin{align}\label{be2}
\bar\beta&\le \frac1{2c_0}(M_1+c_2)\le M(1+N)^2 K^2,\text{ where }M=\frac{M_1+[T_0+\lambda/(\lambda-1)]/2}{2c_0},\\
\label{ebar}
    \bar\eta&\le 1-\left(\frac{1}{1+1/N}\right)^{2M(1+N)^2K^2}
    =1-\left[(1+1/N)^{1+N}\right]^{-2M(1+N)K^2}.
\end{align}    
Note that the function $(1+x)^{1+1/x}$ is increasing for $x>0$. Hence, one has, for $x\in(0,1]$,
\beq\label{exineq}
(1+x)^{1+1/x}\le 2^2=4=e^{\kappa_*}\text{ with } \kappa_*=\ln 4.
\eeq 
Since $1/N=K^{3/2}<1$, we can apply \eqref{exineq} to estimate $(1+1/N)^{1+N}$ in \eqref{ebar} and obtain
    \beq\label{bem1}
    \bar \eta\le 1-e^{-2\kappa_* M(1+N)K^2}=1-e^{-2\kappa_* M(K^2+\sqrt K)}.
    \eeq
Combining estimate \eqref{bem1} with the formula of $\theta$ in \eqref{themu} gives
\beq \label{teKes}
\theta(K)=-\frac{\ln \bar \eta}{ \ln \lambda}
\ge -\frac{1}{\ln\lambda} \ln\left( 1-e^{-2\kappa_* M(K^2+\sqrt K)} \right).
\eeq 
Therefore, we obtain the limit  \eqref{thelim}.

We establish \eqref{thest} now.
Using the inequality $1-e^{-x}\le 2x$ for $x\ge 0$, we have
\beq\label{ebar2}
     1-e^{-2\kappa_* M(K^2+\sqrt K)}\le 4M\kappa_* (1+N)K^2=4\kappa_* M(K^2+K^{1/2})\le 8M\kappa_* K^{1/2}.
\eeq
Define 
\beq\label{Kodef}
K_0=(8M\kappa_*+1)^{-4}\in(0,1).
\eeq 
For $K\le K_0$, we have $K<1$ and $8M\kappa_*K^{1/2}\le K^{1/4}$, hence, together with \eqref{teKes} and \eqref{ebar2}, 
\beq\label{prethe}
\theta(K) \ge -\frac{1}{\ln\lambda} \ln (8M\kappa_*K^{1/2})
\ge -\frac{1}{\ln\lambda} \ln (K^{1/4}) =\frac{\ln K}{4\ln \delta_*}.
\eeq
This proves \eqref{thest}.
\end{proof}

\begin{remark}\label{samecyl}
(a) In Theorem \ref{STthm1}(i), when $p=0$, condition \eqref{dcond3} becomes
        \beq\label{scond}
{\rm diam}\left(\mathbb P_1 \left(\bar  Q_{[t-\delta_*,t]}\right)\right)\le  K,
\eeq 
and estimate \eqref{udecay2} reads as
\beq\label{scex}
    \max_{\overline{Q}_{t}} u^+
\le C\left(\max_{\overline{Q}_T}u^+\right) \exp(-\nu t) \text{ for all }t\ge T. 
\eeq
Compared with the the cylinder case, the condition \eqref{scond} is not as strict, but the decay in \eqref{scex} is just of the same exponential type, see e.g. \cite[Proposition 4.1]{HI3Rus}.

(b) Regarding the rate of decay in $t$, both  \eqref{beta-est} in   Theorem \ref{STthm0} and \eqref{udecay3} in Theorem \ref{STthm1}(ii) are of the same power-type.
However, the latter estimate  is stronger thanks to the limit \eqref{thelim}. 
This, of course, requires stricter conditions on the set $Q$ and is achieved by using a more technical tool -- the Growth Lemma.
\end{remark}

We now give examples for sets satisfying \eqref{dcond3} or \eqref{dcond4}.
In Examples \ref{eg1c}--\ref{eg1b} below, the number $T\ge 0$ can be taken sufficiently large if needed. 
All the sets $Q$ will satisfy Assumption \ref{GGamcond} thanks to Proposition \ref{gset}. Although not showed here, these sets  can certainly be made more complicated and verified by using Propositions \ref{gset1} and \ref{gset2} instead. They also serve as prototypes for later subsection \ref{HBsec} and Section \ref{inhomsec}.

\begin{example}\label{eg1c}
Given  numbers $p\in[0,1/2]$, $\kappa\ge 0$, let $X(t)$ be continuous in $[T,\infty)$ with 
$|X(t)|\le \kappa t^p$ for all $t\in[T,\infty)$.
    Let 
    $$Q=\{(x,t)\in\R^n\times(T,\infty):|x-X(t)|<\mathcal K t^p\}.$$
    Then $Q\subset \mathcal Q$ where 
        $$\mathcal Q=\{(x,t)\in\R^n\times(T,\infty):|x|<(\kappa+\mathcal K) t^p\}.$$
    Using \eqref{bigQ} and \eqref{dcmQ}, we have \eqref{dcond3} when $p\in[0,1/2)$ or \eqref{dcond4} when $p=1/2$ with $K=\kappa+\mathcal K$.
    Obviously,  $K$ is small whenever $\kappa$ and $\mathcal K$ are small.
\end{example}

\begin{example}\label{eg1}
 Consider the case $n=1$, $\delta_*=1$ and $p=0$ in \eqref{dcond3}. Let   
    \beqs 
    Q=\{(x,t)\in \R\times(T,\infty):|x-t|<1\}.
    \eeqs 
    Then 
${\rm diam}\left(\mathbb P_1 \left(\bar  Q_{[t-1,t]}\right)\right)
={\rm diam} ([t-2,t+1])=3.$
Hence, the condition \eqref{dcond3} is met, but the set $Q$ is not contained in a space-time cylinder. (However, this is a slanted cylinder which can be converted to a cylinder by a linear transformation, see e.g. \cite[Chap. 3, Sec. 2, Lemma 2.2]{LandisBook}.)

Another example, which is not  a slanted cylinder, is
  \beqs 
    Q=\{(x,t)\in \R\times(T,\infty):|x-t-\kappa(\ln t)^q|<1\},\  \kappa >0, \  q\in[0,1).
    \eeqs 
Observe, for each $\tau$, that
\beq\label{P1tau}
\mathbb P_1(\overline{Q}_\tau)=\left[\tau+\kappa(\ln \tau)^q-1,\tau+\kappa(\ln \tau)^q+1\right].
\eeq
For large $\tau$, the function $\tau+\kappa(\ln \tau)^q$ is increasing. Therefore, based on \eqref{P1tau}, the maximum and minimum of $\mathbb P_1 \left(\bar  Q_{[t-1,t]}\right)$ are
$t+\kappa(\ln t)^q+1$ and $(t-1)+\kappa(\ln (t-1))^q-1$, respectively.
As a consequence,
\begin{align*}
{\rm diam}\left(\mathbb P_1 \left(\bar  Q_{[t-1,t]}\right)\right)
&=[t+\kappa(\ln t)^q+1]-[t-1+\kappa(\ln (t-1))^q-1]\\
&=3+\kappa[(\ln t)^q-(\ln (t-1))^q]\le 3+\kappa q.
\end{align*}
(The last inequality is obtained by the Mean Value Theorem with $t$ large.) 
Therefore, \eqref{dcond3} is satisfied. 
\end{example}

\begin{example}\label{eg1b}
Consider the case $n=1$, $\delta_*=1$ and $p\in(0,1/2)$ in \eqref{dcond3}.
    Given numbers  $\kappa> 0$, $\mathcal K>0$, $m\in[0,1)$ and $q\ge 0$, let $X(t)=\kappa t^{m}(\ln t)^q$ and 
    \beqs
      Q=\{(x,t)\in\R\times(T,\infty):|x-X(t)|< \mathcal Kt^p\}.
    \eeqs  
Note that the case $m<p$ falls into Example \ref{eg1c}.    
   From now on, we consider  $m\ge p$. 
   Let $Z(t)=\mathbb P_1\left(\overline{Q}_{\left[t-t^{2p},t\right]}\right)$.
Observe, for each $\tau$, that
\beq\label{PQtau}
\mathbb P_1(\overline{Q}_\tau)=\left[X(\tau)-\mathcal K \tau^p,X(\tau)+\mathcal K\tau^p\right].
\eeq

Regarding the right end point of the interval in \eqref{PQtau}, the function $X(\tau)+\mathcal K\tau^p$ is increasing. Therefore, taking $\tau\in[t-t^{2p},t]$ in \eqref{PQtau}, it implies that the right end point of the set 
$Z(t)$  is $X(t)+\mathcal K t^p$.

Regarding the left end point of the interval in \eqref{PQtau}, let $g(\tau)= X(\tau)-\mathcal K\tau^p$ and we have the following cases.
\begin{itemize}
\item If $m>p$, the function $g(\tau)$ is increasing for $\tau$ large.
\item If $m=p$,  the function $g(\tau)$ is increasing for $\tau$ large when $q>0$, and is monotone when $q=0$.
\end{itemize}

When the function $g(\tau)$ is increasing, the left end point of $Z(t)$ is  $X(t-t^{2p})-\mathcal K(t-t^{2p})^p$.
When $g(\tau)$ is decreasing, the left end point of $Z(t)$ is  $X(t)-\mathcal Kt^p$.
Correspondingly, the diameter of $Z(t)$ is
\beqs
d(t)=[  X(t)+\mathcal Kt^p] -[X(t-t^{2p})-\mathcal K(t-t^{2p})^p]
=\mathcal Kt^p+\mathcal K(t-t^{2p})^p +  X(t)-X(t-t^{2p}),
\eeqs
or,
\beqs
d(t)=[ X(t)+\mathcal Kt^p]-[ X(t)-\mathcal Kt^p]
=2\mathcal Kt^p.
\eeqs
Using the Mean Value Theorem to estimate $X(t)-X(t-t^{2p})$, we have, in both scenarios of $g(\tau)$, 
\beq\label{degrate}
d(t)\le 2\mathcal Kt^p + \mathcal O(\kappa\cdot t^{m-1} (\ln t)^q \cdot t^{2p}).
\eeq

We choose $m\in[p,1-p)$ now. This yields $m-1+2p<p$, and \eqref{degrate} implies that \eqref{dcond3} is satisfied for large $T$.
If we further restrict $m\in(p,1-p)$, then $X(t)$  grows much faster than $t^p$.

When  $m=1/2$ and $\kappa=1$, if we use \eqref{bigQ} and \eqref{dcmQ} with 
$$\mathcal Q=\{(x,t)\in\R^n\times(T,\infty):|x|<t^{1/2}(\ln t)^q+\mathcal Kt^{p}\},$$
 then the diameter of $\bar {\mathcal Q}_t$ is $2(t^{1/2}(\ln t)^q+\mathcal Kt^{p})$ which does not satisfy \eqref{dcond3}.
\end{example}

The next result deals with the solutions instead of subsolutions of $L_0$.

\begin{theorem}\label{STthm1b}
Under Assumption \ref{secondA}, suppose $u\in C(\overline{Q})\cap C_{x,t}^{2,1}(\widetilde Q)$ satisfies $L_0 u= 0$ in $\widetilde Q$ and $u= 0$ on $\Gamma(Q)\setminus \overline{Q}_0$. 
\begin{enumerate}[label=\tnum]
    \item\label{st4a} If  $p\in[0,1/2)$ and  \eqref{dcond3} holds, then 
\beqs%\label{udecay2b}
\max_{\overline{Q}_t}|u|\le C\left(\max_{\overline{Q}_T}|u|\right) \exp(-\nu t^{1-2p}) \text{ for all }t\ge T,
\eeqs
where $C$ and $\nu$ are positive constants independent  of $u$.
    \item\label{st4b} If \eqref{dcond4} holds,
then
\beq\label{udecay3b}
\max_{\overline{Q}_t}|u|\le \left(\max_{\overline{Q}_T}|u|\right) (t/\mu)^{-\theta} \text{ for all }t\ge T,
\eeq
where $\mu>0$ and $\theta>0$ are the same as in Theorem \ref{STthm1}\ref{st4ii}.
\end{enumerate}
\end{theorem}
\begin{proof}
    We apply Theorem \ref{STthm1} to $u$ and $(-u)$ to have estimates for $u^+$ and $u^-$, and then use identity \eqref{maxrel}. We omit the details.
\end{proof}

Theorem \ref{STthm1} prompts a question that what estimate one can expect when the diameters grow slower than $t^{1/2}$ in \eqref{dcond4} but not as slow as \eqref{dcond3}. 
The following Theorem \ref{STthm2} and Corollary \ref{cor1} investigate this issue.

\begin{theorem}\label{STthm2}
Assume there are constants $\delta_*\in(0,1)$ and $T\ge e$ such that 
    \beq\label{dc8}
{\rm diam}\left(\mathbb P_1 \left(\overline{Q}_{[\delta_* t,t]}\right)\right)\le  \frac{t^{1/2}}{\varphi(t)},
\eeq 
for all $t\ge T$, where $\varphi(t)>0$ is increasing to infinity.

Under Assumption \ref{secondA}, there are  numbers $\mu_1=\mu_1(\delta_*)>0$, $\mu_0=\mu_0(\delta_*)\in(0,1)$ and $t^*>T$  such that the following statements hold true.
\begin{enumerate}[label=\tnum]
\item If $u\in C(\overline{Q})\cap C_{x,t}^{2,1}(\widetilde Q)$ satisfies $L_0 u\le 0$ in $\widetilde Q$ and $u\le 0$ on $\Gamma(Q)\setminus \overline{Q}_0$, then 
\beq\label{udec8}
\max_{\overline{Q}_t}u^+\le \left (\max_{\overline{Q}_{T}} u^+\right) t^{-\mu_1  \ln\left(\varphi(\mu_0\sqrt t)\right)} \text{ for all }t\ge t^*.
\eeq

\item If $u\in C(\overline{Q})\cap C_{x,t}^{2,1}(\widetilde Q)$ satisfies $L_0 u= 0$ in $\widetilde Q$ and $u=0$ on $\Gamma(Q)\setminus \overline{Q}_0$, then 
\beq\label{udec8b}
\max_{\overline{Q}_t}|u|\le \left (\max_{\overline{Q}_{T}} |u|\right)  t^{-\mu_1 \ln\left(\varphi(\mu_0\sqrt t)\right)} \text{ for all }t\ge t^*.
\eeq
\end{enumerate}
\end{theorem}
\begin{proof}Same as in the proof of Theorem \ref{STthm1}\ref{st4ii}, denote 
$\lambda=1/\delta_*>1$ and $\kappa_*=\ln 4$.
Set 
\beq\label{barMK}
\bar M=\frac{M_1+\lambda/(\lambda-1)}{2c_0}  \text{ and }
\bar K_0=(8\bar M\kappa_*+1)^{-4}\in(0,1).
\eeq
Let 
$\bar K(t)=1/\varphi(t)$.
Then $\bar K(t)$ is positive, decreasing, for $t\ge T$, and goes to zero as $t\to \infty$.
Let $\bar T>\max\{\lambda,\lambda/(\lambda-1)\}$ be sufficiently large such that
\beq\label{barKK}
\bar K(t)\le \bar K_0\text{ for all }t\ge \bar T.
\eeq 

Let $\bar t$ be any number in $[\bar T,\infty)$.
Note that $\bar T>\lambda/(\lambda-1)>1$. 
For all $t\ge \bar t$, we have, by \eqref{dc8} that
   \beqs%\label{dc1}
{\rm diam}\left(\mathbb P_1 \left(\overline{Q}_{[\delta_* t,t]}\right)\right)
\le \bar K(t) t^{1/2}
\le \bar K(\bar t) t^{1/2}.
\eeqs 
Mimicking \eqref{Ttau2} and \eqref{Tstar}, let 
\beqs
\bar T_0=\max \left\{\bar t,\frac{\lambda}{\lambda-1}\right\}=\bar t \text{ and } 
\bar T_*=2\max\left\{\bar t,\frac{\lambda}{\lambda-1}\right\}
=2\bar t.
 \eeqs
We follow the proof of Theorem \ref{STthm1}(ii) using 
\beqs% \label{newpars}
T_0=\bar T_0=\bar t,\ 
K=\bar K(\bar t), \ 
T_*=\bar T_* \text{ and, again, } N=K^{-3/2}.
\eeqs 
Instead of \eqref{c2e}, we return to \eqref{prec2} and utilize the lower bound of $\lambda^k$ in \eqref{tk5}. It results in 
    \beqs
    \frac{R_k^2}{2\tau_k}
 \le \frac12(1+N)^2K^2\left (\frac{\lambda T_0}{(T_k-T_0)(\lambda-1)}+\frac{\lambda}{\lambda-1}\right) .    
 \eeqs
 With this and \eqref{NK1}, we re-estimate $\beta_k$ by
\beq\label{newbk}
\beta_k= \frac{1}{2c_0}\max\left\{M_1,\frac{R_k^2}{2\tau_k}\right\}\le  \frac{1}{2c_0}(1+N)^2 K^2 \left[M_1+\frac{\lambda}{2(\lambda-1)}\left (\frac{T_0}{T_k-T_0}+1\right)\right].
\eeq

 Referring to $T_k$ in \eqref{tk3}, we momentarily assume 
\beq\label{T2T}
T_k\ge 2T_0=2\bar T_0=2\bar t,
\eeq
Then $T_0/(T_k-T_0)\le 1$, and, hence, we have from \eqref{newbk} that
\beqs
\beta_k\le \frac{1}{2c_0}(1+N)^2 K^2 \left[M_1+\frac{\lambda}{\lambda-1}\right]
= \bar M(1+N)^2 K^2.
\eeqs
Therefore, under the assumption \eqref{T2T}, we can replace \eqref{the2} with 
\beq\label{been}
\beta_k\le \bar \beta \text{ and }\eta_k\le \bar \eta,
\eeq
where the new numbers $\bar \beta$ and $\bar\eta$ satisfy inequalities \eqref{be2} and \eqref{ebar}, with $M$ being replaced with $\bar M$ defined in \eqref{barMK}.

By replacing $K_0$ in \eqref{Kodef} with $\bar K_0$ defined in \eqref{barMK}, and observing that, thanks to \eqref{barKK},  $K=\bar K(\bar t)\le \bar K_0$, we have that the number $\nu=\theta$ in \eqref{themu}, now denoted by $\theta(K)$, still satisfies the last inequality in \eqref{prethe}, i.e.,  
\beq\label{theb}
\nu=\theta(K)\ge -\frac 1{4\ln\lambda}\ln \bar K(\bar t)
=\frac1{4\ln\lambda}\ln \varphi(\bar t).
\eeq

Suppose $k$ and $k_0$ are two integers satisfying $k>k_0\ge 1$ and $T_{k_0}\ge 2\bar t$.
For any integer $j\in[k_0,k]$, we have $T_j\ge T_{k_0}\ge 2\bar t$.
Then, thanks to \eqref{T2T} and \eqref{been}, we obtain $\eta_j\le \bar \eta$ 
Therefore,  same as the inequality \eqref{exkdec7} but with the use of \eqref{ugk0} instead of \eqref{ugrowk} in deriving it, we have 
\beq\label{uQk}
    \max_{\overline{Q}_{T_k}} u^+\le \left (\max_{\overline{Q}_{T_{k_0}}} u^+\right)  \lambda^{-\nu (k-k_0)}.
\eeq
In particular, taking $k=2k_0$ in \eqref{uQk} gives
\beq\label{uQk2}
    \max_{\overline{Q}_{T_{2k_0}}} u^+\le \left (\max_{\overline{Q}_{T_{k_0}}} u^+\right)  \lambda^{-\nu k_0}.
\eeq

 Let $k_0\ge 1$ be an integer satisfying
 \beq\label{k0}
 2\bar t \le T_{k_0}=\bar t+\frac{\lambda(\lambda^{k_0}-1)}{\lambda-1}< (2+\lambda)\bar t.
 \eeq
Observe that \eqref{k0}  is equivalent to 
 \beq\label{tbar0}
 \frac{\lambda-1}{\lambda}\bar t +1\le \lambda^{k_0}<  \frac{\lambda-1}{\lambda}(1+\lambda)\bar t+1,
 \eeq
which, in turn, is equivalent to
  \beq\label{bart}
\log_\lambda \left[ \frac{\lambda-1}{\lambda}\bar t +1\right]\le k_0< \log_\lambda \left[ \frac{\lambda-1}{\lambda}(1+\lambda)\bar t+1\right].
 \eeq
To guarantee \eqref{bart}, we need 
\beq \label{tb1}
\log_\lambda \left[ \frac{\lambda-1}{\lambda}\bar t +1\right]\ge 1
\eeq 
and 
\beq\label{tb2}
\log_\lambda\left[ \frac{\lambda-1}{\lambda}(1+\lambda)\bar t+1\right]-\log_\lambda\left[ \frac{\lambda-1}{\lambda}\bar t +1\right]\ge 1.
\eeq 
Both conditions \eqref{tb1} and \eqref{tb2} are equivalent to $\bar t\ge \lambda$ which is satisfied thanks to the fact $\bar t\ge \bar T>\lambda$. Therefore, such an integer $k_0$ satisfying \eqref{k0} exists.

By \eqref{uQk2}, the estimate of $\nu$ in \eqref{theb}, and then the first inequality in \eqref{tbar0} as well as Theorem \ref{maxcor}, we have 
\beq\label{uQp}
    \max_{\overline{Q}_{T_{2k_0}}} u^+\le \left (\max_{\overline{Q}_{T_{k_0}}} u^+\right)  \left(\lambda^{k_0}\right)^{-\frac1{4\ln\lambda} \ln(\varphi(\bar t))}
    \le \left (\max_{\overline{Q}_T} u^+\right)  \left(\frac{\lambda-1}{\lambda}\bar t\right)^{-\frac1{4\ln\lambda} \ln(\varphi(\bar t))}.
\eeq
Using the second inequality in  \eqref{tbar0} and also the fact $\bar t> 1$, we estimate, see \eqref{tk3} with $k:=2k_0$,  
\begin{align*}
T_{2k_0}&=\bar t+\frac{\lambda(\lambda^{2k_0}-1)}{\lambda-1}
\le \bar t +\frac{\lambda(\frac{\lambda^2-1}{\lambda}\bar t+1)^2}{\lambda-1}\\
&\le \bar t +\frac{\lambda}{\lambda-1}(\lambda\bar t+1)^2
\le \bar t +\frac{\lambda}{\lambda-1}[(\lambda +1)\bar t]^2.
\end{align*}
This yields
\beq\label{mu0}
T_{2k_0}\le \left[1+\frac{\lambda(\lambda +1)^2}{\lambda-1}\right]\bar t^2
=\frac{\bar t^2}{\mu_0^2} ,
\text{ where }
\mu_0=\left[1+\frac{\lambda(\lambda +1)^2}{\lambda-1}\right]^{-1/2}.
\eeq
Clearly, $\mu_0$ in \eqref{mu0} depends only $\delta_*$ only and belongs to the interval $(0,1)$.
By  Theorem \ref{maxcor}, \eqref{mu0} and \eqref{uQp}, we have
\beq\label{uQmu}
    \max_{\overline{Q}_{ \bar t^2/\mu_0^2}} u^+\le    \max_{\overline{Q}_{T_{2k_0}}} u^+\le \left (\max_{\overline{Q}_{T}} u^+\right)  \left(\frac{\lambda-1}{\lambda}\bar t\right)^{-\frac1{4\ln\lambda} \ln(\varphi(\bar t))}.
\eeq

For any sufficiently large $t$ with $\mu_0\sqrt{t}\ge \bar T$ and 
$(\lambda-1)\mu_0\sqrt t/\lambda\ge t^{1/4}$,  taking $\bar t =\mu_0\sqrt{t}$ in \eqref{uQmu} yields
\begin{align*}
    \max_{\overline{Q}_{t }} u^+ 
    &\le \left (\max_{\overline{Q}_{T}} u^+\right)  \left(\frac{(\lambda-1)\mu_0\sqrt t}{\lambda}\right)^{-\frac1{4\ln\lambda} \ln\left(\varphi(\mu_0\sqrt t)\right)}\\
&\le \left (\max_{\overline{Q}_{T}} u^+\right)  \left(t^{1/4}\right)^{-\frac1{4\ln\lambda} \ln\left(\varphi(\mu_0\sqrt t)\right)},
\end{align*}
which implies  \eqref{udec8} with $\mu_1=1/(16\ln \lambda)$.

(ii) Apply part (i) to $u$ and $-u$, and then use \eqref{maxrel}. 
\end{proof}

\begin{example}\label{egln}
    Assume there are numbers $\delta_*\in(0,1)$ and $z>0$ such that 
\beqs%\label{dcond8}
{\rm diam}\left(\mathbb P_1 \left(\overline{Q}_{[\delta_* t,t]}\right)\right)=\mathcal O\left(\frac{ t^{1/2}}{(\ln t)^z}\right)\text{ as }t\to\infty.
\eeqs 
Then the estimate \eqref{udec8}, resp., \eqref{udec8b}, can be replaced with 
\begin{align}\label{ud}
\max_{\overline{Q}_t}u^+&=\mathcal O(t^{-\frac{z\mu_1}{2}\ln\ln t})  \text{ as } t\to\infty,\\
\label{udb}
\text{resp., }
\max_{\overline{Q}_t}|u|&=\mathcal O(t^{-\frac{z\mu_1}{2}\ln\ln t})  \text{ as } t\to\infty .
\end{align}
Obviously, the explicit decaying rate in \eqref{ud} is faster than any rate in \eqref{udecay3}, but slower than any rate in \eqref{udecay2}.

To prove \eqref{ud} and \eqref{udb}, we apply Theorem \ref{STthm2} and take
$\varphi(t)=(\ln t)^z/K$
for some constant $K>0$ and all sufficiently large $t$.
Then, for sufficiently large $t$,
\beq\label{phie}
\ln\left(\varphi(\mu_0\sqrt t)\right)=z\ln\ln (\mu_0\sqrt t)-\ln K=z\ln \left(\frac12\ln t+\ln\mu_0\right)-\ln K
\ge \frac z 2 \ln\ln t.
\eeq
With estimate \eqref{phie}, the statements \eqref{ud} and \eqref{udb} follow from \eqref{udec8} and \eqref{udec8b}, respectively. 
\end{example}

\begin{corollary}\label{cor1}
Assume there exists $\delta_*\in(0,1)$ such that
  \beq\label{dcondo}
{\rm diam}\left(\mathbb P_1 \left(\overline Q_{[\delta_* t,t]}\right)\right)=o(t^{1/2}) \text{ as }t\to\infty.
 \eeq 
Under Assumption \ref{secondA}, one has the following.

\begin{enumerate}[label=\tnum]
    \item If $u\in C(\overline{Q})\cap C_{x,t}^{2,1}(\widetilde Q)$ satisfies $L_0 u\le 0$ in $\widetilde Q$ and $u\le 0$ on $\Gamma(Q)\setminus \bar Q_0$,  
then
\beq\label{mxo1}
\max_{\bar Q_t}u^+ =o(t^{-\theta}) \text{ for all }\theta>0.
\eeq

    \item If $u\in C(\overline{Q})\cap C_{x,t}^{2,1}(\widetilde Q)$ satisfies $L_0 u=0$ in $\widetilde Q$ and $u=0$ on $\Gamma(Q)\setminus \bar Q_0$,  
then
\beq\label{mxo2}
\max_{\bar Q_t}|u| =o(t^{-\theta}) \text{ for all }\theta>0.
\eeq
\end{enumerate}
\end{corollary}
\begin{proof}
For sufficiently large $t$, define
\beq\label{phinf}
\varphi(t)=\inf_{\tau\ge t} \frac{\tau^{1/2}}{{\rm diam}\left(\mathbb P_1 \left(\overline Q_{[\delta_* \tau,\tau]}\right)\right)+1}.
\eeq
Clearly, $\varphi(t)$ is increasing.
Thanks to \eqref{dcondo},
\beqs 
\lim_{\tau\to\infty}\frac{\tau^{1/2}}{{\rm diam}\left(\mathbb P_1 \left(\overline Q_{[\delta_* \tau,\tau]}\right)\right)+1}=\infty,
\eeqs 
thus, $\varphi(t)\to\infty$ as $t\to\infty$. By \eqref{phinf}, 
\beqs
\varphi(t)\le \frac{t^{1/2}}{{\rm diam}\left(\mathbb P_1 \left(\overline Q_{[\delta_* t,t]}\right)\right)+1},
\eeqs
which implies \eqref{dc8}. 
Then we can apply Theorem \ref{STthm2}.
Since $\ln(\varphi(\mu_0\sqrt t))\to\infty$ as $t\to\infty$, we obtain \eqref{mxo1}, resp. \eqref{mxo2} from \eqref{udecay3}, resp. \eqref{udecay3b}.
\end{proof}

\subsection{Bounded drifts}\label{HBsec}
We now turn to the case of general bounded drifts.

\begin{assumption}\label{condB}
Assume there is a constant $M_2\ge 0$ so that
\beq\label{BM2}
 |b(x,t)|\le M_2 \text{ for all }(x,t)\in \widetilde Q.
\eeq
\end{assumption}

We obtain estimates for different types of space-time sets.

\begin{theorem}\label{STthm3}
Under Assumptions \ref{secondA} and \ref{condB}, consider the following cases. 
\begin{itemize}
    \item Case 1. There are $\delta_*>0$, $T\ge \delta_*$ and $K>0$ such that
    \beq\label{dcond5}
{\rm diam}\left(\mathbb P_1 \left(\bar  Q_{[t-\delta_*,t]}\right)\right)\le  K \text{ for all }t\ge T.
\eeq

\item Case 2. There are $\delta_*>0$, $T\ge\max\{\delta_*, e\}$ and 
\beq\label{Kc2} 
K\in\left(0,K_*\right), \text{ where }
K_*=\frac1{1+\delta_*}\sqrt{\frac{2c_0\delta_*}{\ln(1+1/\delta_*)}},
\eeq  
such that 
 \beq\label{dcond60}
{\rm diam}\left(\mathbb P_1 \left(\bar  Q_{[t-\delta_*,t]}\right)\right)\le  K(\ln t)^{1/2} \text{ for all }t\ge T.
\eeq

\item Case 3. There are $\delta_*>0$, $z\in(0,1)$, $T\ge \max\{\delta_*^\frac1{1-z},e\}$ and  
\beq\label{Kc3}
  K\in\left(0,(1-z)K_*\right),\text{ where }K_*=\frac{c_0}{2(\ln 4) (M_1+M_2+1)},
\eeq
 such that 
    \beq\label{dcond6}
{\rm diam}\left(\mathbb P_1 \left(\overline{Q}_{\left[t-\delta_* t^z,t\right]}\right)\right)\le  K\ln t
\text{ for all }t\ge T.
\eeq
\end{itemize}
Define
\beq \label{theta}
\theta=\begin{cases}
   1 & \text{in Case 1},\\ 
1-K^2/K_*^2 & \text{in Case 2},\\
1-z-K/K_*&\text{in Case 3}. 
\end{cases}
\eeq 
Then there exist numbers $C>0$ and $\nu>0$ such that   one has the following.
\begin{enumerate}[label=\tnum]
    \item  If $u\in C(\overline{Q})\cap C_{x,t}^{2,1}(\widetilde Q)$ satisfies  $Lu\le 0$ in $\widetilde Q$ and $u\le 0$ on $\Gamma(Q)\setminus \overline{Q}_0$, then 
\beq\label{udecay0}
\max_{\overline{Q}_t}u^+\le C \left(\max_{\overline{Q}_T}u^+\right) e^{-\nu t^\theta} \text{ for all } t\ge T.
\eeq

\item  If $u\in C(\overline{Q})\cap C_{x,t}^{2,1}(\widetilde Q)$ satisfies  $Lu=0$ in $\widetilde Q$ and $u= 0$ on $\Gamma(Q)\setminus \overline{Q}_0$, then 
\beq\label{udecay1}
\max_{\overline{Q}_t}|u|\le C \left(\max_{\overline{Q}_T}|u|\right) e^{-\nu t^\theta} \text{ for all }t\ge T.
\eeq
\end{enumerate}
\end{theorem}
\begin{proof}
Note that the number $\theta$ defined by \eqref{theta} belongs to the interval $(0,1)$ in Cases 2 and 3.

\medskip\noindent
Part (i). We prove the estimate \eqref{udecay0} for each case.

\medskip\noindent
\textbf{Case 1.}
Let $T_0=T$ and $\delta_0=\delta_*$. Let $\tau_k=\delta_0$ for all $k\ge 1$ and define $T_k$ as in \eqref{Tksum},
 $d_k$ as in  \eqref{Vdk},  $r_k$  as in \eqref{rRk} with $N=1$, $R_k$ as in \eqref{rRk0}, and $\beta_k$, $\eta_k$ as in \eqref{sTek0}. 
In particular, $r_k=d_k$ and $R_k=2d_k$.

 Thanks to Assumption \ref{condB} with the value $M_2$ in \eqref{BM2},  we can take $m_k=M_2$ for all $k$ in \eqref{TABk},  hence, $\beta_k$ and $\eta_k$ in \eqref{sTek0} become 
\beq\label{sTek}
 \beta_k=\frac{1}{2c_0}\max\left \{M_1+M_2 R_k,\frac{R_k^2}{2\tau_k}\right\}, \quad 
 \eta_k=1-\left(\frac{1}{2}\right)^{2\beta_k}.
 \eeq

Observe that $T_k=T_0+\delta_0 k$, hence, 
\beq \label{Tinv0}
[T_{k-1},T_k]=[T_k-\delta_*,T_k].
\eeq 
Thanks to \eqref{dcond5}, we have $d_k\le K$, which yields $R_k=2d_k\le 2K$ and 
$R_k^2/(2\tau_k)\le 2K^2/\delta_0$.
Therefore, 
\beq\label{the3}
\beta_k\le \bar\beta\eqdef \frac{1}{2c_0}\max\left\{M_1+2M_2 K,\frac{2K^2}{\delta_0}\right\}\text{ and }
    \eta_k\le \bar\eta\eqdef 1-(1/2)^{2\bar\beta}.
\eeq
Hence, same as \eqref{exkdec}, 
\beq\label{ux1}
\max_{\overline{Q}_{T_k}} u^+\le \left( \max_{\overline{Q}_{T_0}} u^+\right) e^{-\nu_0 k},\text{ where }\nu_0=-\ln\bar \eta.
\eeq
Clearly, \eqref{ux1} is also true for $k=0$.

For $t\ge T=T_0$, there is an integer $k\ge 1$ such that $t\in[T_{k-1},T_k)$. 
By Theorem \ref{maxcor} and \eqref{ux1}, we have
\beqs
\max_{\overline{Q}_{t}} u^+\le \max_{\overline{Q}_{T_{k-1}}} u^+\le  \left( \max_{\overline{Q}_{T_0}} u^+\right) e^{-\nu_0 (k-1)}
= \left( \max_{\overline{Q}_T} u^+\right) e^{-\nu_0 (\frac{T_k-T_0}{\delta_0}-1)}
\le \left( \max_{\overline{Q}_T} u^+\right) e^{\nu_0} e^{-\nu_0 \frac{t-T_0}{\delta_0}}.
\eeqs
Then  we obtain \eqref{udecay0} with $\nu=\nu_0/\delta_0$ and $C=e^{\nu_0+\nu T_0}$.

\medskip\noindent
\textbf{Case 2.} 
Let $\delta_0=\delta_*$ and fix a number 
 \beq\label{T0}
 T_0\ge \max\left\{ T,e^\frac{1}{\delta_0^2 K^2},e^{ 4(M_1+M_2)^2/K^2}\right\}.
 \eeq
With $T_0$ in \eqref{T0}, define, for $k\ge 1$, the numbers $\tau_k$, $T_k$ and $d_k$ as in Case 1, 
   \beq\label{drR}
r_k=\frac{\tau_k K^2\ln T_k}{d_k}\text{ and } 
R_k=d_k+r_k.
\eeq 
We again have the formula for $\beta_k$ as in \eqref{sTek}, and then define $\eta_k$ as in \eqref{sTek0}.
For $k\ge 1$,
by applying inequality \eqref{ugrowk} and Theorem \ref{maxcor}, we obtain
\beq\label{uiter}
\max_{\overline{Q}_{T_k}} u^+\le \left( \max_{\overline{Q}_{T_0}} u^+ \right) \eta_1\eta_2\ldots \eta_k
\le \left( \max_{\overline{Q}_{T}} u^+ \right) \eta_1\eta_2\ldots \eta_k.
\eeq

By \eqref{Tinv0} and \eqref{dcond60}, one has 
\beq\label{dk1}
d_k\le K (\ln T_k)^{1/2}.
\eeq 
On the one hand, squaring both sides of \eqref{dk1} gives a lower bound of $K^2\ln T_k$ which we use for the numerator of $r_k$ in \eqref{drR}. It results in
\beq\label{dk2}
r_k\ge \frac{\tau_k d_k^2}{d_k}=\delta_0 d_k \text{ and, hence, }
\ R_k=r_k+d_k\le (1+1/\delta_0) r_k.
\eeq
On the other hand, it is clear from the second relation in \eqref{drR} that $R_k\ge r_k$.
Then using \eqref{dk1} as an upper bound of $d_k$ which is the denominator of $r_k$ in \eqref{drR},  we deduce 
\beq\label{RkTk}
R_k\ge r_k\ge \frac{\tau_k K^2 \ln T_k}{K (\ln T_k)^{1/2}}=\delta_0 K (\ln T_k)^{1/2}
\ge \delta_0 K (\ln T_0)^{1/2}  .
\eeq 
Using the last two numbers in \eqref{T0} to bound $\ln T_0$ from below, we further estimate 
\beq\label{Rkdel}
R_k\ge \delta_0 \max\left\{\frac{1}{\delta_0},2(M_1+M_2)\right\} .
\eeq 
Consequently, $R_k\ge 1$ and $R_k/(2\delta_0)\ge M_1+M_2$, which in turn imply
    \beqs 
    \frac{R_k^2}{2\tau_k}=\frac{R_k}{2\delta_0}\cdot R_k \ge (M_1+M_2) R_k\ge M_1+M_2 R_k,
    \eeqs
Therefore, $\beta_k$ in \eqref{sTek} is 
\beq\label{bquot}
\beta_k=\frac1{2c_0}\cdot \frac{R_k^2}{2\tau_k}=\frac{(r_k+d_k)R_k}{4c_0\tau_k}
=\left(1+\frac{d_k}{r_k}\right) \frac{r_k R_k}{4c_0\tau_k}.
\eeq

With $d_k/r_k\in(0,1/\delta_0]$ from \eqref{dk2}, we have, same as inequality \eqref{exineq},
\beq \label{qexq}
\left(1+\frac{d_k}{r_k}\right)^{1+\frac{r_k}{d_k}}\le \left(1+\frac{1}{\delta_0}\right)^{(1+\delta_0)}=e^{\kappa_*} \text{ with }
\kappa_*=(1+\delta_0)\ln(1+1/\delta_0).
\eeq 
From \eqref{sTek0}, \eqref{bquot} and \eqref{qexq}, it follows that
\beq\label{etk1}
\eta_k
= 1-\left[\left(1+\frac{d_k}{r_k}\right)^{1+\frac{r_k}{d_k}}\right]^{-\frac{d_k R_k}{2c_0\tau_k}}
\le 1-e^{-\frac{\kappa_* d_k R_k}{2c_0\tau_k}}.    
\eeq
Note from the second property in \eqref{dk2} and the definition of $r_k$ in \eqref{drR}  that
\beq\label{dRkln}
\frac{d_k R_k}{2c_0\tau_k}\le  \frac{d_k(1+1/\delta_0) r_k}{2c_0\tau_k}=\frac{(1+1/\delta_0)K^2}{2c_0} \ln T_k.
\eeq
Combining this estimate with \eqref{etk1} gives 
\beq\label{eTk1}
\eta_k\le 1-e^{-\varep \ln T_k}=1-T_k^{-\varep}, \text{ where } \varep=\frac{(1+1/\delta_0)K^2\kappa_*}{2c_0}=K^2/K_*^2\in(0,1).
\eeq

Observe that  $T_j\ge T_0>1$ for all $j\ge 0$. Thus, the inequality \eqref{eTk1} implies
\begin{align}
\ln \prod_{j=1}^k\eta_j
&\le \sum_{j=1}^k \ln (1-T_j^{-\varep})
\le -\sum_{j=1}^k T_j^{-\varep}
\le - \sum_{j=1}^k\frac1{\tau_{j+1}}\int_{T_j}^{T_{j+1}} \tau^{-\varep}\d\tau\notag \\
&= -\frac1{\delta_0}\int_{T_1}^{T_{k+1}} \tau^{-\varep}\d\tau
=\frac{1}{\delta_0 (1-\varep)} (-T_{k+1}^{1-\varep}+T_{1}^{1-\varep}).
\label{lneta}
\end{align}
With $\theta$ defined in \eqref{theta} and $\varep$ defined in \eqref{eTk1}, we have 
$1-\varep=\theta$ and obtain
\beq\label{exkdec2}
\prod_{j=1}^k\eta_j \le \exp\left\{\frac1{\delta_0 \theta}T_{1}^{\theta}\right\} \exp\left\{-\frac1{\delta_0 \theta}T_{k+1}^{\theta}\right\}
=C \exp\{-\nu T_{k+1}^{\theta}\},
\eeq
where
\beq\label{nudef2}
\nu=\frac{1}{\delta_0 \theta} \text{ and } C=e^{\nu T_{1}^{\theta}}.
\eeq
Combining \eqref{uiter} with \eqref{exkdec2} gives
\beq\label{uQ4}
\max_{\overline{Q}_{T_k}} u^+
\le C\left( \max_{\overline{Q}_{T}} u^+ \right) \exp\{-\nu T_{k+1}^{\theta}\}.
\eeq

For $t\ge T_1$, there is $k\ge 1$ such that $t\in[T_{k},T_{k+1})$. By Theorem \ref{maxcor} and \eqref{uQ4}, 
\beq \label{uQ5}
\max_{\overline{Q}_{t}} u^+\le \max_{\overline{Q}_{T_{k}}} u^+ 
\le C\left( \max_{\overline{Q}_{T}} u^+ \right)  \exp\{-\nu T_{k+1}^{\theta}\}
\le C \left( \max_{\overline{Q}_{T}} u^+ \right) e^{-\nu t^{\theta}}.
\eeq 
For $t\in[T,T_1]$, by Theorem \ref{maxcor} again,
\beq\label{uQ6}
\max_{\overline{Q}_{t}} u^+\le \max_{\overline{Q}_{T}} u^+ \le \left( \max_{\overline{Q}_{T}} u^+ \right) e^{-\nu t^{\theta}}e^{\nu T_1^{\theta}}
=C \left( \max_{\overline{Q}_{T}} u^+ \right) e^{-\nu t^{\theta}}.
\eeq 
Combining \eqref{uQ5} and \eqref{uQ6}, we obtain \eqref{udecay0}.

\medskip\noindent
\textbf{Case 3.} 
Set 
\beq\label{szd0}
\sigma=\frac{z}{1-z}\text{ and }\delta_0= \left[\delta_*(1-z)^z\right]^\frac{1}{1-z},
\eeq
so that 
\beq\label{szd}
\frac{\sigma}{\sigma+1}=z\text{ and } \delta_*=\delta_0^\frac1{\sigma+1} (\sigma+1)^\frac\sigma{\sigma+1}.
\eeq
Same as \eqref{Ttau} in the proof of part \ref{st4i} of Theorem \ref{STthm1}, we take $T_0=T$, which gives $T_0>1$, and $\tau_k=\delta_0 k^\sigma$ for $k\ge 1$.
For $k\ge 1$, define $T_k$ and $d_k$ as in \eqref{Tksum} and  \eqref{Vdk},
\beq\label{dk3}
r_k=\frac{\tau_k K\ln T_k }{d_k} \text{ and } R_k=r_k+d_k.
\eeq 
Again, $\beta_k$ is now given as in \eqref{sTek}, and $\eta_k$ is defined by \eqref{sTek0}.

For estimates of $T_k$, $k$ and $\tau_k$, we have the same inequalities \eqref{tk1}--\eqref{tau01}.
By the inequality \eqref{tau01} and values in \eqref{szd}, one has
\beqs
\tau_k\le \delta_0^\frac1{\sigma+1} (\sigma+1)^\frac\sigma{\sigma+1} T_k^\frac\sigma{\sigma+1}=\delta_* T_k^z.
\eeqs
Thus, 
\beq \label{Tinv}
[T_{k-1},T_k]=[T_k-\tau_k,T_k]\subset [T_k-\delta_* T_k^z,T_k].
\eeq

Thanks to \eqref{Tinv} and \eqref{dcond6}, we have
\beq\label{rd3cond}
d_k\le K\ln  T_k.
\eeq
Note from \eqref{dk3} and \eqref{rd3cond} that
\beq\label{rkprop}
r_k\ge \tau_k,
\eeq 
and from \eqref{rd3cond} and \eqref{tk1} that
\beq\label{dkprop}
d_k\le K\ln\left\{T_0+\frac{\delta_0 }{\sigma+1}(k+1)^{\sigma+1}\right\}.
\eeq 

Because $\tau_k=\delta_0 k^\sigma$ and $T_k$ can be estimated by \eqref{tk1},  
there exists an integer  $k_0\ge 1$ such that one has, for all $k\ge k_0$,   
\begin{align} \label{rdp0}
\tau_k&\ge 1 \text{ and } \tau_k\ge K\ln\left\{T_0+\frac{\delta_0 }{\sigma+1}(k+1)^{\sigma+1}\right\},\\
\label{rdp2}
T_k&\le \frac{2\delta_0 }{\sigma+1}(k+1)^{\sigma+1}.
\end{align}

Let $k\ge k_0$.
By \eqref{rkprop}, \eqref{dkprop} and  the second property in \eqref{rdp0}, we have
\beq \label{rdp1}
r_k\ge d_k \text{ and, hence,  }
R_k=r_k+d_k\le 2r_k.
\eeq 
It is clear from \eqref{dk3} and  \eqref{rkprop} that
\beq \label{Rrtau}
R_k\ge r_k\ge \tau_k,
\eeq 
hence,
\beq\label{RMR}  \frac{R_k^2}{2\tau_k}
 \ge \frac{R_k \tau_k }{2\tau_k}=\frac{R_k}{2}.
\eeq  
Let $M=2(M_1+M_2+1)$.
Then $M\ge 1$ and $M\ge 2(M_1+M_2)$. Observe from the first inequality in \eqref{rdp0} and \eqref{Rrtau} that
 $R_k \ge 1$. Combining inequality \eqref{RMR} with these facts, we can estimate $\beta_k$ in \eqref{sTek} by  
\beq\label{be3} 
\beta_k\le  \frac1{2c_0}\max\left \{(M_1+M_2)R_k,\frac{MR_k^2}{2\tau_k}\right\}
\le \frac1{2c_0}\max\left \{\frac{M R_k}{2},\frac{MR_k^2}{2\tau_k}\right\}=\bar \beta_k\eqdef \frac{MR_k^2}{4c_0\tau_k}.
 \eeq 
 
For $k\ge k_0$, using \eqref{be3}, the fact $d_k\le r_k$ from \eqref{rdp1} and inequality \eqref{exineq},  we estimate $\eta_k$ in \eqref{sTek0} by
 \begin{align*}
\eta_k
&\le 1-\left(1+\frac{d_k}{r_k}\right)^{-2\bar \beta_k}
=1-\left(1+\frac{d_k}{r_k}\right)^{ -\frac{M(d_k+r_k)R_k}{2c_0\tau_k}}\\
&= 1-\left[\left(1+\frac{d_k}{r_k}\right)^{1+\frac{r_k}{d_k}}\right]^{-\frac{M d_kR_k}{2c_0\tau_k}}
\le 1-e^{-\frac{M\kappa_* d_kR_k}{2c_0\tau_k}},
 \end{align*}
 where $\kappa_*=\ln 4$.
For the last exponent, we have from  the second property in \eqref{rdp1} and definition of $r_k$ in \eqref{dk3} that
\beqs
\frac{d_kR_k}{2c_0\tau_k}\le  \frac{d_k (2r_k)}{2c_0\tau_k}=\frac{K}{c_0} \ln T_k.
\eeqs
It results in 
\beq\label{eTk2}
\eta_k\le 1-e^{-\varep \ln T_k}=1-T_k^{-\varep}, \text{ where } \varep=M\kappa_*K/c_0=K/K_*.
\eeq

For $k\ge j\ge k_0$, we have $T_j\ge T_{0}>1$, and, by \eqref{eTk2}, 
\beqs
\ln \prod_{j=k_0}^k\eta_j\le \sum_{j=k_0}^k \ln (1-T_j^{-\varep})
\le -\sum_{j=k_0}^k T_j^{-\varep}.
\eeqs
Using the fact $T_j\le \frac{2\delta_0 }{\sigma+1}(j+1)^{\sigma+1}$ from \eqref{rdp2},
 we obtain 
\begin{align}
\ln \prod_{j=k_0}^k\eta_j
&\le -\sum_{j=k_0}^k\left( \frac{2\delta_0 }{\sigma+1}(j+1)^{\sigma+1}\right)^{-\varep}
\le -\left( \frac{\sigma+1}{2\delta_0 }\right)^{\varep} \sum_{j=k_0}^k \int_{j+1}^{j+2} \tau^{-\varep(\sigma+1)} \d\tau \notag \\
&= -\left( \frac{\sigma+1}{2\delta_0 }\right)^{\varep} \int_{k_0+1}^{k+2} \tau^{-\varep(\sigma+1)}\d\tau.\label{lnprod}
\end{align}
Thanks to \eqref{Kc3}, \eqref{szd} and the formula of $\varep$ in \eqref{eTk2}, we have
\beqs 
K<(1-z)K_*=\frac{K_*}{\sigma+1}=\frac{K}{\varep(\sigma+1)}\text{ which  implies }\varep(\sigma+1)<1.
\eeqs 
Thus, computing the last integral in \eqref{lnprod} gives
\beq\label{lneta3}
\ln \prod_{j=k_0}^k\eta_j\le \left( \frac{\sigma+1}{2\delta_0 }\right)^{\varep}\frac1{1-\varep(\sigma+1)}\left[-(k+2)^{1-\varep(\sigma+1)}+(k_0+1)^{1-\varep(\sigma+1)}\right] .
\eeq 
It follows that
\beq \label{exkdec6}
\prod_{j=k_0}^k\eta_j\le C_0  \exp\{-\nu_0(k+2)^{1-\varep(\sigma+1)}\},
\eeq
where
\beq\label{nuc0}
\nu_0=\left( \frac{\sigma+1}{2\delta_0 }\right)^{\varep}\frac1{1-\varep(\sigma+1)}\text{ and }
C_0=\exp\left[ \nu_0 (k_0+1)^{1-\varep(\sigma+1)}\right]>1.
\eeq
Note from the first inequality in \eqref{Tkk0} for $k:=k+1$ that
\beq\label{Tkk5}
\left[\frac{\sigma+1}{\delta_0 }(T_{k+1}-T_0)\right]^\frac1{\sigma+1} \le k+2.
\eeq
From the choice of $\sigma$ in \eqref{szd0}, the formula of $\theta$ in \eqref{theta} and the formula of $\varep$ in \eqref{eTk2}, we have the following relations
\beq \label{esz}
\frac{1-\varep(\sigma+1)}{\sigma+1}=(1-z)-\varep=\theta,\ 
\theta+\varep=1-z,\ 
1-\varep(\sigma+1)=\theta(\sigma+1)=\frac{\theta}{1-z}.
\eeq 
Then we have from \eqref{exkdec6} and \eqref{Tkk5} that
\beq\label{exkdec5}
\begin{aligned}
\prod_{j=k_0}^k\eta_j
&\le C_0 \exp\left\{-\nu_0\left[\frac{\sigma+1}{\delta_0 }(T_{k+1}-T_0)\right]^\frac{1-\varep(\sigma+1)}{\sigma+1}\right\}  \\
&=C_0 \exp\left\{-\nu(T_{k+1}-T_0)^\theta\right\},
\end{aligned}
\eeq
where
\beq\label{nudef3}
\nu=\nu_0 \left(\frac{\sigma+1}{\delta_0 } \right)^\theta
={\frac{1}{2^{\varep}(1-\varep(\sigma+1))}}\left(\frac{\sigma+1}{\delta_0 } \right)^{\theta+\varep}
={\frac{1-z}{2^{K/K_*}\theta }}\left(\frac{1}{\delta_0(1-z) } \right)^{1-z}. 
\eeq
(We  have used \eqref{esz} for the last two identities.)
By inequality \eqref{ugk0} for $k_*=k_0-1\ge 0$, estimate \eqref{exkdec5} and Theorem \ref{maxcor}, we have 
\beq\label{exkdec8}
\max_{\overline{Q}_{T_k}}u^+
\le \left(\prod_{j=k_0}^k \eta_j\right)\left(\max_{\overline{Q}_{T_{k_0-1}}} u^+\right)
\le C_0 \exp\left\{-\nu(T_{k+1}-T_0)^\theta\right\} \left(\max_{\overline{Q}_{T_{0}}} u^+\right) .
\eeq

Let $t$ be any number in $[T_{k_0},\infty)$. There is an integer $k\ge k_0$ such that $t\in[T_{k},T_{k+1})$. Thanks to Theorem \ref{maxcor} applied to $t_1:=T_k$ and $t_2:=t$,  
inequality \eqref{exkdec8}, and the facts $T_0=T$ and $T_{k+1}>t$,   we have
\beqs
\max_{\overline{Q}_t}u^+\le \max_{\overline{Q}_{T_{k}}}u^+
\le C_0\left(\max_{\overline{Q}_{T}} u^+\right) \exp\{-\nu(t-T_0)^\theta\}.
\eeqs
With $\theta<1$, we have $(t-T_0)^\theta\ge t^\theta-T_0^\theta$. Therefore,
\beq\label{uQ8}
\max_{\overline{Q}_t}u^+
\le C_0 \left(\max_{\overline{Q}_{T_{0}}} u^+\right) \exp\{-\nu t^\theta+\nu T_0^\theta\}
=  C_0 e^{\nu T_0^\theta} \left(\max_{\overline{Q}_{T}} u^+\right) e^{-\nu t^\theta}.
\eeq

For $t\in[T,T_{k_0}]$, by Theorem \ref{maxcor},
\beq\label{uQ9}
\max_{\overline{Q}_{t}} u^+\le \max_{\overline{Q}_{T}} u^+ \le \left( \max_{\overline{Q}_{T}} u^+ \right) e^{-\nu t^{\theta}}e^{\nu T_{k_0}^{\theta}}.
\eeq 

Combining \eqref{uQ8} and \eqref{uQ9} yields \eqref{udecay0} with 
$C=\max\{C_0 e^{\nu T_0^\theta},e^{\nu T_{k_0}^{\theta}}\}$.
The proof of Part (i) is complete.

\medskip\noindent
Part (ii). We apply part (i) to $u$ and $(-u)$  to derive estimates for $u^+$ and $u^-$,  and then use relation \eqref{maxrel}. We omit the details.
\end{proof}

It is worth noticing that condition \eqref{dcond6} is not a consequence of \eqref{dcond60} just based on their right-hand sides. It is due to the much larger interval $[t-\delta_* t^z,t]$ required on the left-hand side of \eqref{dcond6}  compared to $[t-\delta_*,t]$ in \eqref{dcond60}.

The following result is similar to Theorem \ref{STthm2}, but based on Cases 2 and 3 of Theorem \ref{STthm3} instead of part \ref{st4ii} of Theorem \ref{STthm1}.

\begin{corollary} \label{cor2}
Under Assumptions \ref{secondA} and \ref{condB}, assume either
\begin{enumerate}[label=\rnum]
    \item there is $\delta_*>0$ such that      
  \beq\label{dco1}
{\rm diam}\left(\bar  Q_{[t-\delta_*,t]}\right)=o ((\ln t)^{1/2}) \text{ as }t\to\infty,
\eeq
or
    \item  there are $\delta_*>0$ and $z\in(0,1)$ such that
\beq\label{dco2}
{\rm diam}\left(\overline{Q}_{[t-\delta_* t^z,t]}\right)=o (\ln t)\text{ as }t\to\infty.
\eeq
\end{enumerate}
Let $\theta_*=1$ in the case \eqref{dco1}, and 
$\theta_*=1-z$ in the case \eqref{dco2}.
Then one can replace the estimate \eqref{udecay0}, resp., \eqref{udecay1}, with
\beq \label{at1}
\forall \theta\in(0,\theta_*):\max_{\overline{Q}_t}u^+=\mathcal O(e^{-t^\theta}) \text{ as }t\to\infty,
\eeq 
resp.,
\beq \label{at2}
\forall \theta\in(0,\theta_*):\max_{\overline{Q}_t}|u|=\mathcal O(e^{-t^\theta})\text{ as }t\to\infty. 
\eeq 
\end{corollary}
\begin{proof}
Let $K_*$ be defined by \eqref{Kc2} in the case \eqref{dco1}, and 
by \eqref{Kc3} in the case \eqref{dco2}.
Let $\theta\in(0,\theta_*)$.

We prove \eqref{at1} first.
Take $K>0$ sufficiently small so that the number 
\beq \label{Ket}
\eta\eqdef
\begin{cases}
    \theta_*-K^2/K_*^2&\text{ in the case \eqref{dco1}},\\
    \theta_*-K/K_* &\text{ in  the case \eqref{dco2}},
\end{cases} 
\text{ is larger that $\theta$.}
\eeq 

Let $T>0$ be sufficiently large such that 
  \beq\label{dco1T}
  T\ge\max\{\delta_*, e\}\text{ and }
{\rm diam}\left(\bar  Q_{[t-\delta_*,t]}\right)\le K(\ln t)^{1/2} \text{ for all  }t\ge T
\eeq
in case the \eqref{dco1}, or 
\beq\label{dco2T}
T\ge \max\{\delta_*^\frac1{1-z},e\}\text{ and }
{\rm diam}\left(\overline{Q}_{[t-\delta_* t^z,t]}\right)\le K \ln t\text{ for all  }t\ge T
\eeq 
in case the \eqref{dco2}.
(Because $K_*$ is \textit{independent} of $T$ in Theorem \ref{STthm3}, the choices of the above $K$ in \eqref{Ket} and then $T$ in \eqref{dco1T} and \eqref{dco2T} are valid.)
Applying Theorem \ref{STthm3} for Case 2 when \eqref{dco1} holds, or Case 3 when \eqref{dco2} holds, we have  from \eqref{udecay0} that
 \beqs 
\max_{\overline{Q}_t}u^+=\mathcal O(e^{-\nu t^\eta})=\mathcal O(e^{-t^\theta}) \text{ as }t\to\infty. 
\eeqs 
Thus, we obtain \eqref{at1}.

The proof of \eqref{at2} is similar with the use of estimate \eqref{udecay1} for $\theta:=\eta$ instead of \eqref{udecay0}.
\end{proof}

\begin{example}\label{eg2}
Considering $n=1$, let $\kappa> 0$, $z\in[0,1)$ and $p\ge 0$.
Take $X(t)=\kappa t^m (\ln t)^q$, where  
\beq \label{mzpq}
(m=1-z,\ 0\le q<p)
\text{ or }(m=1-z,\ q=p,\text{ small }\kappa)
\text{ or } (0\le m<1-z,\  q\ge 0).
\eeq 
With $T\ge e$ and $\mathcal K>0$,  let
    \beqs
 Q=\{(x,t)\in\R^2:  t>T,\   |x-X(t)|< \mathcal K (\ln t)^p\}.
    \eeqs
    By the virtue of Proposition \ref{gset}, $Q$ satisfies Assumption \ref{GGamcond}.
    Next, we examine $Q$ for Theorem \ref{STthm3}.
Set $\delta_*=1$. 
Given any number  $K>0$. We claim, for  sufficiently large $T$ and small $\mathcal K$, that 
\begin{enumerate}[label=\rnum]
    \item\label{de1}  the condition \eqref{dcond60} is satisfied for $z=0$, $p=1/2$, and
    \item\label{de2} the condition \eqref{dcond6} is satisfied for $z\in(0,1)$, $p=1$.
\end{enumerate}

Indeed, let $Z(t)=\mathbb P_1 \left(\overline{Q}_{\left[t-t^z,t\right]}\right)$ which is the set on the left-hand side of \eqref{dcond60} in case \ref{de1}, and of \eqref{dcond6} in case \ref{de2}. 
For each $\tau$,
\beq\label{P1int}
\mathbb P_1(\overline{Q}_\tau)=\left[X(\tau)-\mathcal K(\ln \tau)^p,X(\tau)+\mathcal K(\ln \tau)^p\right].
\eeq

Clearly, $X(\tau)+\mathcal K(\ln \tau)^p$ is increasing. Hence,  the right end point of $Z(t)$ is $X(t)+\mathcal K(\ln t)^p$.

Regarding the left end point in \eqref{P1int}, denote $g(\tau)=X(\tau)-\mathcal K(\ln \tau)^p$. For large $\tau $, by computing the derivative of $g(\tau)$, we have the function $g(\tau)$ is increasing for ($m>0$) or ($m=0$, $q>p$),  and decreasing for ($m=0$, $q<p$), and monotone for ($m=0$, $q=p$).
In all cases,  $g(\tau)$ is monotone. Hence, the left end point of $Z(t)$ is either $X(t)-\mathcal K(\ln t)^p$ if $g(\tau)$  is decreasing, or 
$X(t-t^z)-\mathcal K(\ln (t-t^z))^p$ if $g(\tau)$  is increasing.
Therefore, the diameter $d(t)$ of $Z(t)$ is either
\beq\label{df1}
d(t)=[X(t)+\mathcal K(\ln t)^p]-[X(t)-\mathcal K(\ln t)^p]
=2\mathcal K(\ln t)^p,
\eeq
or 
\begin{align}
d(t)&=[X(t)+\mathcal K(\ln t)^p]-[X(t-t^z)-\mathcal K(\ln (t-t^z))^p]\notag \\
&=\mathcal K(\ln t)^p+\mathcal K(\ln (t-1))^p+X(t)-X(t-t^z).\label{df2}
\end{align} 
Considering \eqref{df2}, one has, by the Mean Value Theorem,
\beqs
0<X(t)-X(t-t^z)=\mathcal O\left(\kappa\left[t^{m-1}(\ln t)^q+t^{m-1}(\ln t)^{q-1}\right]\cdot t^z\right)\text{ as } t\to\infty.
\eeqs
Thus, for both \eqref{df1} and \eqref{df2},
\beq\label{df3}
d(t)\le 2\mathcal K(\ln t)^p+\mathcal O(\kappa t^{m-1+z}(\ln t)^q).
\eeq
With \eqref{df3} and \eqref{mzpq}, we obtain the statements (a) and (b). 
\end{example}

The next result improves the decaying rate in \eqref{udecay0} and \eqref{udecay1} under a significantly slower growth than the one in \eqref{dcond60}.

\begin{theorem}\label{STthm4}
Under Assumptions \ref{secondA} and \ref{condB}, assume there are $\delta_0>0$, $T\ge \max\{\delta_0,e^e\}$ and $K>0$ such that, for any $t\ge T$,  
    \beq\label{dcond9}
{\rm diam}\left(\mathbb P_1 \left(\bar  Q_{[t-\delta_0,t]}\right)\right)\le  K (\ln\ln t)^{1/2}.
\eeq
Set 
\beq\label{epdef}
\varep=\frac{(1+\delta_0)^2\ln(1+1/\delta_0)K^2}{2c_0\delta_0}.
\eeq
Then there is a positive number  $C$ depending on $c_0$, $M_1$, $M_2$, $K$ and $T$ such that the following statements hold true.
\begin{enumerate}[label=\tnum]
\item If $u\in C(\overline{Q})\cap C_{x,t}^{2,1}(\widetilde Q)$ satisfies $Lu\le 0$ in $\widetilde Q$ and $u\le 0$ on $\Gamma(Q)\setminus \overline{Q}_0$, then 
\beq\label{udecay8c}
\max_{\overline{Q}_t} u^+\le C\left(\max_{\overline{Q}_T}u^+\right)\exp\left\{-\frac{1}{\delta_0 }\int_e^t (\ln \tau)^{-\varep}\d\tau\right\}\text{ for all }t\ge T.
\eeq

\item If $u\in C(\overline{Q})\cap C_{x,t}^{2,1}(\widetilde Q)$ satisfies $Lu= 0$ in $\widetilde Q$ and $u=0$ on $\Gamma(Q)\setminus \overline{Q}_0$, then 
\beq\label{udecay8d}
\max_{\overline{Q}_t}|u|\le C\left(\max_{\overline{Q}_T}|u|\right)\exp\left\{-\frac{1}{\delta_0 }\int_e^t (\ln \tau)^{-\varep}\d\tau\right\}\text{ for all }t\ge T.
\eeq
\end{enumerate}
\end{theorem}
\begin{proof}
We prove part (i) first.
We fix a number $T_0\ge T$ sufficiently large such that 
\beq\label{TT0}
(\ln\ln T_0)^{1/2}\ge \frac{\delta_0}{K} \max \left\{ \frac1{\delta_0},2(M_1+M_2)\right\}.
\eeq
Same as in the proofs for Cases 1 and 2  of Theorem \ref{STthm3}, let $\tau_k=\delta_0$, for $k\ge 1$, and define $T_k$ as in \eqref{Tksum},  $d_k$ as in  \eqref{Vdk},  and set 
\beq\label{rRlln}
r_k=\frac{\tau_k K^2\ln\ln T_k}{d_k},\quad  R_k=r_k+d_k,
\eeq
and define the numbers $\beta_k$ and $\eta_k$ as in \eqref{sTek0}.
We again have $\beta_k$ rewritten as in \eqref{sTek}.

Clearly, $T_k=T_0+\delta_0 k$ and $[T_{k-1},T_k]=[T_k-\delta_0,T_k]$.
Then we have from \eqref{dcond9} that 
\beq \label{dklln}
d_k\le K (\ln\ln T_k)^{1/2}.
\eeq 
With \eqref{rRlln} and \eqref{dklln}, we have, same as \eqref{dk2} and \eqref{RkTk},
\beq\label{rkp2}
r_k\ge \frac{\tau_k d_k^2}{d_k}=\delta_0 d_k, \quad R_k\le (1+1/\delta_0)r_k,\quad R_k\ge  r_k\ge \delta_0 K(\ln\ln T_k)^{1/2}.
\eeq
Using \eqref{TT0} and the last lower bound of $R_k$ in \eqref{rkp2}, we have
\beqs
R_k\ge \delta_0  K(\ln\ln T_0)^{1/2}\ge \delta_0 \max \left\{ \frac1{\delta_0},2(M_1+M_2)\right\}.
\eeqs
We then proceed same as from \eqref{Rkdel} to \eqref{etk1}.
As a result, one has, same as estimate \eqref{etk1}, 
 \beq\label{etkap}
\eta_k\le 1-e^{-\frac{\kappa_* d_k R_k}{2c_0}},\text{ where }\kappa_*=(1+\delta_0)\ln(1+1/\delta_0).
\eeq
Using the first inequality of \eqref{dRkln} and then definition of $r_k$ in \eqref{rRlln}, we have
\beq\label{quotdR}
\frac{d_k R_k}{2c_0}\le  \frac{ d_k(1+1/\delta_0) r_k}{2c_0\tau_k}=\frac{(1+1/\delta_0)K^2}{2c_0} \ln\ln T_k.
\eeq
Observe from \eqref{epdef} and the formula of $\kappa_*$ in \eqref{etkap} that  $\kappa_* (1+1/\delta_0)K^2/(2c_0)=\varep$.
Combining this with \eqref{etkap} and \eqref{quotdR}, we obtain 
\beqs
\eta_k\le 1-e^{-\varep \ln\ln T_k}=1-(\ln T_k)^{-\varep}.
\eeqs
It follows that
\begin{align*}
\ln \prod_{j=1}^k\eta_j
&\le \sum_{j=1}^k \ln (1-(\ln T_j)^{-\varep})
\le -\sum_{j=1}^k (\ln T_j)^{-\varep}\\
&\le -\sum_{j=1}^k  \frac{1}{\tau_{j+1}}\int_{T_j}^{T_{j+1}} (\ln \tau)^{-\varep}\d\tau
=-\frac{1}{\delta_0 }\int_{T_1}^{T_{k+1}} (\ln \tau)^{-\varep}\d\tau.
\end{align*}
Combining this estimate with \eqref{ugrowk} gives
\begin{align}
\max_{\overline{Q}_{T_k}} u^+&\le \left(\max_{\overline{Q}_{T_0}}u^+\right)\prod_{j=1}^k\eta_j 
\le  \left(\max_{\overline{Q}_{T_0}}u^+\right)\exp\left\{-\frac{1}{\delta_0 }\int_{T_1}^{T_{k+1}} (\ln \tau)^{-\varep}\d\tau\right\}\notag \\
&=C \left(\max_{\overline{Q}_{T_0}}u^+\right)\exp\left \{-\frac{1}{\delta_0 }\int_{e}^{T_{k+1}} (\ln \tau)^{-\varep}\d\tau\right\},\label{uQ10}
\end{align}
  where
$  C=\exp\left \{\frac{1}{\delta_0 }\int_{e}^{T_{1}} (\ln \tau)^{-\varep}\d\tau\right\}.$

    For $t\ge T_1$, let $k\ge 1$ be the integer such that $t\in[T_{k},T_{k+1})$. By Theorem  \ref{maxcor} and estimate \eqref{uQ10}, one has 
\begin{align}
\max_{\overline{Q}_t}u^+
&\le \max_{\overline{Q}_{T_{k}}}u^+
\le C \left(\max_{\overline{Q}_{T_0}}u^+\right)\exp\left\{-\frac{1}{\delta_0 }\int_{e}^{T_{k+1}} (\ln \tau)^{-\varep}\d\tau\right\}\notag \\
&\le C \left(\max_{\overline{Q}_T}u^+\right)\exp\left\{-\frac{1}{\delta_0 }\int_{e}^{t} (\ln \tau)^{-\varep}\d\tau\right\}.\label{uQ11}
\end{align}
For $t\in[T,T_1]$, we have, by Theorem  \ref{maxcor} again, 
\begin{align}
\max_{\overline{Q}_t}u^+
&\le \left(\max_{\overline{Q}_T}u^+\right) \
\le \left(\max_{\overline{Q}_T}u^+\right)
\cdot \exp\left\{\frac{1}{\delta_0 }\int_e^{T_1} (\ln \tau)^{-\varep}\d\tau\right\}\exp\left\{-\frac{1}{\delta_0 }\int_e^t (\ln \tau)^{-\varep}\d\tau\right\}\notag  \\
&= \left(\max_{\overline{Q}_T}u^+\right)\cdot C\exp\left\{-\frac{1}{\delta_0 }\int_e^t (\ln \tau)^{-\varep}\d\tau\right\}.\label{uQ12}
\end{align}
Combining \eqref{uQ11} and \eqref{uQ12} gives \eqref{udecay8c}.

Part (ii) is obtained by applying part (i) to $u$ and $(-u)$.
\end{proof}

\begin{remark}\label{hrate}
Denote $h(t)=\exp\left\{-\frac{1}{\delta_0 }\int_e^t (\ln \tau)^{-\varep}\d\tau\right\}$ which is the decaying function in \eqref{udecay8c} and \eqref{udecay8d}. Then one has, for all $M>0$ and $\theta\in(0,1)$, that
\beqs
e^{-Mt}=\mathcal O(h(t)) \text{ and }
h(t)=\mathcal O(e^{-t^\theta})\text{ as }t\to\infty.
\eeqs
Therefore, the decaying rate $h(t)$ obtained in Theorem \ref{STthm4} is an intermediate one between those in Case 1 and Cases 2, 3 of Theorem \ref{STthm3}. This is caused by the intermediate growth rate of the diameters  in \eqref{dcond9} compared to those in \eqref{dcond5}, \eqref{dcond60} and \eqref{dcond6} of Theorem \ref{STthm3}.    
\end{remark}

\section{Inhomogeneous problems}\label{inhomsec}

Let $Q$ be  a subset of $\R^n\times(0,\infty)$ that satisfies Assumption \ref{GGamcond} with $\mathbb P_2(\overline{Q})=[0,\infty)$.
Recall that $\mathcal S_J$ and $\mathcal S_t$ are defined in \eqref{SJdef}.
Let  $A:\widetilde Q\to \mathcal M^{n\times n}_{{\rm sym}}$ and $b:\widetilde Q\to \R^n$ be two given  functions. Define the linear operators $L$ by \eqref{Ltil}, and $L_0$ by \eqref{Lz}.

In all statements in the remainder of this section, $u$ is any function in $C(\overline{Q})\cap C_{x,t}^{2,1}(\widetilde Q)$.

We apply  Lemma \ref{lemG2} to restrictions of $Q$ on consecutive time intervals to obtain some recursive estimates. It results in the following counter part of \cite[Proposition 4.5]{HI4}.

\begin{theorem}\label{genlem}
Under Assumptions \ref{secondA},
let $T_k$, $m_k$, $\beta_k$ and $\eta_k$ be the same as those from \eqref{Tksum} to \eqref{sTek0}.
Suppose there is a function $F\in C([0,\infty),[0,\infty))$ such that 
  \beq\label{LwF}   
   Lu(x,t)\le F(t)\text{ for all } (x,t)\in \widetilde Q.
  \eeq 
Define 
\begin{align}
J_k&=\max_{\overline{Q}_{T_k}} u^+\text{ for $k\ge 0$,}\notag \\
 \label{lamk}
\Lambda_0&=J_0\text{ and }
\Lambda_k=\sup_{\mathcal S_{(T_{k-1},T_k]}} u^+ +\int_{T_{k-1}}^{T_k}F(t)\d t
 \text{ for $k\ge 1$.  }
\end{align}
Then one has
\beq\label{wv1}
\max_{\overline{Q}_{[T_{k-1},T_k]}} u^+
\le J_{k-1}+\Lambda_k  \text{ for all }k\ge 1,
\eeq 
\beq \label{Jkstar}
J_k\le \eta_k\eta_{k-1}...\eta_{k_*+1} J_{k_*} +\sum_{m=k_*+1}^k \left[ \left(\prod_{j=m+1}^k \eta_j\right) \Lambda_m\right]  \text{ for all } k>k_*\ge 0,
\eeq 
and, consequently,
\beq \label{Jkest0}
J_k\le  \sum_{m=0}^k \left[ \left(\prod_{j=m+1}^k \eta_j\right) \Lambda_m\right] 
\text{ for all }k\ge 1.
\eeq
\end{theorem}
\begin{proof}
Thanks to \eqref{LwF}, we can apply, for each $k\ge 1$, Theorem \ref{maxprin2}\ref{MP1} to the interval $I=[T_{k=1},T_k]$, and function $f_1(t):=F(t)$. It follows from inequality \eqref{max2} that        
\begin{align*}
        \max_{\overline{Q}_{[T_{k-1},T_k]}} u^+ 
        &\le \max_{\overline{Q}_{T_{k-1}}\cup\mathcal S_{(T_{k-1},T_k]}} u^+ 
        +\int_{T_{k-1}}^{T_k} F(\tau)\d \tau \\
        &\le \left (J_{k-1}+\sup_{\mathcal S_{(T_{k-1},T_k]}} u^+ \right)+\int_{T_{k-1}}^{T_k} F(\tau)\d \tau
        \le J_{k-1}+\Lambda_k
\end{align*}
which implies \eqref{wv1}.

Let $k\ge 1$. We apply Lemma \ref{lemG2} to the interval $I=[T_{k-1},T_k]$, with the same $M_1$, but, referring to the numbers appearing from \eqref{Tksum} to \eqref{sTek0}, 
\beqs 
M_2:=m_k,\  
x_*:=x_k,\ 
r:=r_k,\ R:=R_k,\ 
\beta=\beta_k,\ 
\eta_*:=\eta_k.
\eeqs 
We obtain from \eqref{Tgrow4} that
\beqs
J_k\le \eta_k J_{k-1}+\Lambda_k.
\eeqs 
Iterating this estimate for $k-1, k-2, \ldots, k_*+1$, we have
\begin{align*}
J_k&\le \eta_k(\eta_{k-1}J_{k-2}+\Lambda_{k-1})+\Lambda_k
=\eta_k\eta_{k-1}J_{k-2}+\eta_k\Lambda_{k-1}+\Lambda_k\\
&\le \eta_k\eta_{k-1}\eta_{k-2}J_{k-3}+\eta_k\eta_{k-1}\Lambda_{k-2}+\eta_k\Lambda_{k-1}+\Lambda_k\\
&\le \ldots \le \eta_k\eta_{k-1}...\eta_{k_*+1} J_{k_*} + \sum_{m=k_*+1}^k \left( \prod_{j=m+1}^k \eta_j \right)\Lambda_m,
\end{align*}
which implies \eqref{Jkstar}. Taking $k_*=0$ in \eqref{Jkstar} yields \eqref{Jkest0}.
\end{proof}

Thanks to \eqref{wv1} and \eqref{Jkest0}, we can estimate $u^+(x,t)$ for all $(x,t)\in \overline Q$ in terms of the values of $u$ on $\Gamma(Q)$. 
We now turn to the question of the behavior of $u(x,t)$ for large $t$.

\subsection{Zero drift}\label{NHZsec}
To estimate the series in \eqref{Jkest0} in some particular cases, we will use the following simple result.
\begin{lemma}[{ \cite[Lemma 4.6]{HI4} }]\label{dH}
    Given a number  $\eta\in(0,1)$ and a sequence $(\Lambda_k)_{k=0}^\infty$ of nonnegative numbers. For $k\ge 0$, let 
$ a_k= \sum_{j=0}^k \eta^{k-j}\Lambda_j$.
Then
\beqs%\label{alim}
\limsup_{k\to\infty} a_k\le \frac1{1-\eta} \limsup_{k\to\infty}\Lambda_k.
\eeqs
\end{lemma}

Theorem \ref{genlem} and Lemma \ref{dH} immediately result in the following explicit asymptotic estimates.

\begin{theorem}\label{NHthm0}
Under  Assumptions \ref{secondA}, considering the cases in Theorem \ref{STthm1}, let
\begin{itemize}
    \item $\Phi(t)=t-\delta_* t^{2p}$ and  $\bar\eta$ be defined by \eqref{the1} in the  case \eqref{dcond3}, or 
\item $\Phi(t)=\delta_* t$ and  $\bar\eta$ be defined by \eqref{the2} with $N=1$ in the  case \eqref{dcond4}.
\end{itemize}
\begin{enumerate}[label=\tnum]
    \item \label{dec1}    Suppose there is a function $f_1\in C([0,\infty),[0,\infty))$ such that 
    $L_0u(x,t)\le f_1(t)$ for all $(x,t)\in \widetilde Q$.  
Then 
\beq\label{lsup1}
\limsup_{t\to\infty} \left(\max_{\overline{Q}_t} u^+\right)
\le \frac{2-\bar\eta}{1-\bar\eta}\left[
\limsup_{t\to\infty}\left(\max_{\mathcal S_t} u^+\right)
+\limsup_{t\to\infty}\int_{\Phi(t)}^{t} f_1(\tau)\d\tau 
\right].
\eeq

\item \label{dec2}     Suppose there is a function $f\in C([0,\infty),[0,\infty))$ such that 
    $|L_0u(x,t)|\le f(t)$ for all $(x,t)\in \widetilde Q$. 
Then 
\beq\label{lasb1}
\limsup_{t\to\infty} \left(\max_{\overline{Q}_t} |u|\right)
\le \frac{2-\bar\eta}{1-\bar\eta}\left[
\limsup_{t\to\infty}\left(\max_{\mathcal S_t} |u|\right)
+\limsup_{t\to\infty}\int_{\Phi(t)}^{t} f(\tau)\d\tau 
\right].
\eeq
\end{enumerate}
\end{theorem}
\begin{proof}
We use the same $T_k$, $\beta_k$ and $\eta_k$ as in the proof of Theorem \ref{STthm1}.

\ref{dec1} 
We apply  Lemma \ref{genlem} taking into account the estimate of $\eta_j$ by \eqref{the1} in the case \eqref{dcond3}, and by \eqref{the2} in the case \eqref{dcond4}. It follows from inequality \eqref{Jkest0} that 
\beqs%\label{Jkest}
J_k\le \sum_{j=0}^{k} \bar\eta^{k-j} \Lambda_{j}.
\eeqs 
By Lemma \ref{dH},
\beq\label{Jk1}
\limsup_{k\to\infty} J_k\le \frac1{1-\bar\eta}\limsup_{k\to\infty}\Lambda_k.
\eeq
By \eqref{wv1} and \eqref{Jk1},
\beq\label{limQ1}
\limsup_{t\to\infty}\left(\max_{\overline{Q}_t} u^+\right)\le \limsup_{k\to\infty} J_{k-1}+\limsup_{k\to\infty} \Lambda_k
\le \frac{2-\bar\eta}{1-\bar\eta}\limsup_{k\to\infty} \Lambda_k.
\eeq
Note from \eqref{tau2} and \eqref{tau3} that
$(T_{k-1},T_k]\subset (\Phi(T_k),T_k]$. 
Together with the definition of $\Lambda_k$ in \eqref{lamk}, it implies
\beq\label{limd1}
\begin{aligned}
\limsup_{k\to\infty} \Lambda_k
&\le \limsup_{t\to\infty}\left(\sup_{ S_{(\Phi(t),t]}} u^+\right)+ \limsup_{t\to\infty}\int_{\Phi(t)}^{t} f(\tau)\d\tau\\
&= \limsup_{t\to\infty}\left(\max_{\mathcal S_t} u^+\right)+ \limsup_{t\to\infty}\int_{\Phi(t)}^{t} f(\tau)\d\tau.
\end{aligned}
\eeq
Combining \eqref{limQ1} with \eqref{limd1} yields inequality \eqref{lsup1}.

\ref{dec2} Denote by $\ell$ the right-hand side of \eqref{lasb1}.
Applying part (i) to $f_1:=f$, we obtain
\beq\label{get1}
\limsup_{t\to\infty} \left(\max_{\overline{Q}_t} u^+\right)
\le \frac{2-\bar\eta}{1-\bar\eta}\left[
\limsup_{t\to\infty}\left(\max_{\mathcal S_t} u^+\right)
+\limsup_{t\to\infty}\int_{\Phi(t)}^{t} f(\tau)\d\tau 
\right]\le \ell.
\eeq
Applying part (i) to $(-u)$ and $f_1:=f$, we obtain
\beq\label{get2}
\limsup_{t\to\infty} \left(\max_{\overline{Q}_t} u^-\right)
\le \frac{2-\bar\eta}{1-\bar\eta}\left[
\limsup_{t\to\infty}\left(\max_{\mathcal S_t} u^-\right)
\limsup_{t\to\infty}\int_{\Phi(t)}^{t} f(\tau)\d\tau 
\right]\le \ell.
\eeq
By using property \eqref{maxrel}, one has
 \beqs
\limsup_{t\to\infty} \left(\max_{\overline{Q}_t} |u|\right)
  \le \max\left \{
 \limsup_{t\to\infty} \left(\max_{\overline{Q}_t} u^+\right), 
 \limsup_{t\to\infty} \left(\max_{\overline{Q}_t} u^-\right)
 \right\}
 \le \limsup_{t\to\infty} \left(\max_{\overline{Q}_t} |u|\right).
 \eeqs
 Thus,
 \beq\label{absw}
\limsup_{t\to\infty} \left(\max_{\overline{Q}_t} |u|\right)
= \max\left \{
 \limsup_{t\to\infty} \left(\max_{\overline{Q}_t} u^+\right), 
 \limsup_{t\to\infty} \left(\max_{\overline{Q}_t} u^-\right)
 \right\}.
 \eeq
Combining this with the  estimates \eqref{get1} and \eqref{get2}, 
one obtains \eqref{lasb1}.
\end{proof}

\subsection{Bounded drifts}\label{NHBsec}
We focus on general bounded drifts now.
We recall that the solution can be estimated  for all time by the virtue of \eqref{wv1}.    
However, the following examples show that the solutions in non-cylindrical sets may have unexpected behaviors compared to the cylindrical case.

\begin{example}\label{ubeg}
Given $T\ge 0$, let $Q=\{(x,t)\in \R\times(T,\infty): |x|<\sqrt t\}$. 
Observe that $\gamma(Q)=\emptyset$ which implies $Q=\widetilde Q$ and $\partial Q\setminus \overline{Q}_T=\Gamma(Q)\setminus (\R\times\{T\})$.
Consider the equation 
\beq\label{ufb}
u_t-u_{xx}=f(x,t)\text{ in } Q, \text{ with }  u|_{\partial Q\setminus (\R\times\{T\})}=0.
\eeq
\begin{enumerate}[label=\rnum]
\item  Letting $f(x,t)=3$, we have the function $u(x,t)=t-x^2$ solves \eqref{ufb}.
In spite of the zero boundary condition and $f$ being bounded, the solution $u(x,t)$ can still go to infinity as $t\to\infty$.

\item  
Let $T=1$ and $f(x,t)=x^2/t^2+2/t$. Then $u(x,t)=1-x^2/t$ solves \eqref{ufb}.
In this case, $0\le f(x,t)\le 3/t\to 0$ as $t\to\infty$, but $u(x,t)$ is only bounded, for, $u(0,t)=1$.
\end{enumerate}
\end{example}

The following asymptotic estimates, as $t\to\infty$, are more involved than those in Theorem \ref{NHthm0}. Moreover, they require additional  monotonicity \eqref{mono} or \eqref{mono2} below.

\begin{theorem}\label{NHthm}
Under Assumptions \ref{secondA} and \ref{condB}, consider the three cases 1, 2, 3 in Theorem \ref{STthm3} with $T>0$, and define, for $t\ge T$,
\beqs
\Phi(t)=
\begin{cases}
    t-\delta_*& \text{ in Cases 1 and 2,}\\
t-\delta_* t^z&\text{ in Case 3.}
\end{cases}
\eeqs 
Recall that the number $\theta$ is defined by \eqref{theta}, and the number $\nu$ is defined by \eqref{nudef2} for Case 2, and by \eqref{nudef3} for Case 3.
\begin{enumerate}[label=\tnum]
\item   Suppose there are continuous functions $F_1:[0,\infty)\to [0,\infty)$ and $\Lambda_*:[T,\infty)\to[0,\infty)$  
such that 
    $Lu(x,t)\le F_1(t)$ for all $(x,t)\in \widetilde Q$, and 
\beq\label{Lamstar1}
\max_{\mathcal S_{[\Phi(t),t]}} u^+ +\int_{\Phi(t)}^{t}F_1(\tau)\d \tau\le \Lambda_*(t)
 \text{ for all $t\ge T$.}
\eeq
 Assume additionally the following.
\begin{itemize}
    \item In Case 2, either 
\beq \label{mono}
\text{ the function }
t\mapsto \exp(\nu t^{\theta})\Lambda_*(t)\text{ is monotone in $[T,\infty)$, or}
\eeq 
\beq \label{mono2}
\text{ the function } t\mapsto \Lambda_*(t)\text{ is decreasing in $[T,\infty)$.}
\eeq
    \item In Case 3, condition \eqref{mono2} is satisfied and the limit 
\beqs%\label{ell2}
\lim_{t\to\infty}\left[ t^{K/K_*} \exp(\mu_* t^{z\theta}) \Lambda_*(t)\right] \text{ exists, where }
\mu_*=\nu \left[\frac{3\cdot 2^\frac{z}{1-z}  \delta_*^{1-z}}{ (1-z)^{z}}\right]^\theta.
\eeqs    
\end{itemize}
Define 
\beq\label{ell1}
\ell=\begin{cases}
    {\displaystyle \limsup_{t\to\infty} \left [ t^{1-\theta} \Lambda_*(t)\right]}& \text{ in Cases 1 and 2,}\\
{\displaystyle \lim_{t\to\infty}\left[ t^{K/K_*} \exp(\mu_* t^{z\theta}) \Lambda_*(t)\right]}&\text{ in Case 3.}
\end{cases}.
\eeq
Then one has
\beq \label{limest}
\limsup_{t\to\infty} \left(\max_{\overline{Q}_t}u^+\right)\le C\ell,
\eeq
where $C>0$ is a constant independent of $u$.

\item Suppose there are continuous functions $F:[0,\infty)\to [0,\infty)$ and $\Lambda_*:[T,\infty)\to[0,\infty)$
such that 
    $|Lu(x,t)|\le F(t)$ for all $(x,t)\in \widetilde Q$ and 
\beq\label{Lamstar2}
\max_{\mathcal S_{[\Phi(t),t]}} |u|+\int_{\Phi(t)}^{t}F(\tau)\d \tau\le \Lambda_*(t)
 \text{ for all $t\ge T$.}
\eeq
With $\Lambda_*(t)$ in \eqref{Lamstar2}, continue as in part (i) from \eqref{mono} to \eqref{ell1}.
Then one has
\beqs%\label{limest2}
\limsup_{t\to\infty} \left(\max_{\overline{Q}_t}|u|\right)\le C\ell.
\eeqs
\end{enumerate}
\end{theorem}
\begin{proof}
Note that we can always take $T>0$ sufficiently large.
We use the same $T_k$, $\beta_k$ and $\eta_k$ as in the proof of Theorem \ref{STthm3}.
Recall from \eqref{Tinv0} and \eqref{Tinv} that 
\beq \label{Tinv3}
[T_{k-1},T_k]\subset [\Phi(T_k),T_k]\text{ for all three cases.}
\eeq 

(i) Denote $\Lambda_m=\Lambda_*(T_m)$. We use the same notation as in the proof of Theorem \ref{STthm3}.
By \eqref{Tinv3}, \eqref{Lamstar1}  and \eqref{ell1}, we have
\beq\label{Lamk}
\limsup_{k\to\infty}\Lambda_k \le \limsup_{t\to\infty}\Lambda_*(t) \le \ell.
\eeq
Thanks to \eqref{wv1} and \eqref{Lamk}, we have
\beq \label{limu0}
\limsup_{t\to\infty} \left(\max_{\overline{Q}_t}u^+\right)\le  \limsup_{k\to\infty}J_k+\limsup_{k\to\infty}\Lambda_k
\le \limsup_{k\to\infty}J_k+\ell.
\eeq 
Below, we focus on estimating the limit superior of $J_k$ as $k\to\infty$.
First, we use \eqref{Jkest0}  to estimate $J_k$.

\medskip\noindent
\textbf{Case 1.} Recall from \eqref{the3} that $\eta_j\le \bar \eta$. 
Utilizing this in \eqref{Jkest0} gives 
\beqs
J_k\le     \sum_{m=0}^k \left[ \left(\prod_{j=m+1}^k \eta_j\right) \Lambda_m\right]
  \le \sum_{m=0}^k \left( \bar\eta^{k-m} \Lambda_m\right).
\eeqs 
Applying Lemma \ref{dH} to estimate the last sum, and then using \eqref{Lamk}, we obtain
\beq\label{limu1}
\limsup_{k\to\infty}J_k
\le \frac1{1-\bar\eta}\limsup_{k\to\infty}\Lambda_k
\le  \frac1{1-\bar\eta}\cdot\ell .
\eeq
Hence, we obtain estimate \eqref{limest} from  \eqref{limu0} and \eqref{limu1}.

\medskip\noindent
\textbf{Case 2.} 
Same as \eqref{lneta}--\eqref{nudef2}, we have
\beq\label{lneta2}
\ln \prod_{j=m+1}^k\eta_j
\le -\frac{1}{\delta_0}\int_{T_{m+1}}^{T_{k+1}} \tau^{-\varep}\d\tau
= \frac{1}{\delta_0 (1-\varep)} (-T_{k+1}^{1-\varep}+T_{m+1}^{1-\varep})
=\nu  (-T_{k+1}^\theta+T_{m+1}^\theta).
\eeq
Combining \eqref{Jkest0}  with \eqref{lneta2} gives
\begin{align}\notag
J_k&\le  \sum_{m=0}^k \left[\exp(\nu  T_{m+1}^\theta-\nu T_{k+1}^\theta) \Lambda_m\right]\\
     &=\exp(-\nu T_{k+1}^{\theta})\sum_{m=0}^k \left[\frac{\exp\left(\nu  T_{m+1}^{\theta}\right)}{\exp\left(\nu  T_m^{\theta}\right)}\exp\left(\nu  T_m^{\theta}\right) \Lambda_*(T_m)\right]\label{uTk} .
\end{align}
For $m\ge 0$, recall $T_{m+1}=T_0+\delta_0 (m+1)=T_m+\delta_0$.
With $\theta\in(0,1)$, we estimate $T_{m+1}^\theta\le T_m^\theta+\delta_0^\theta$. Therefore, the quotient in the last sum in \eqref{uTk} can be estimated by
\beq\label{eeTT}
\frac{\exp\left(\nu  T_{m+1}^{\theta}\right)}{\exp\left(\nu  T_m^{\theta}\right)}
\le \frac{\exp\left(\nu  (T_{m}^{\theta}+\delta_0^\theta)\right)}{\exp\left(\nu  T_m^{\theta}\right)}
= C_1\eqdef \exp\left(\nu \delta_0^\theta\right).
\eeq
Utilizing this in \eqref{uTk} yields
\beq\label{IJk}
J_k\le C_1I_k, \text{ where } I_k=\exp(-\nu T_{k+1}^{\theta})\sum_{m=0}^k \left[\exp\left(\nu  T_m^{\theta}\right) \Lambda_*(T_m)\right].
\eeq

To estimate $I_k$,  we consider the following three scenarios corresponding to \eqref{mono} and \eqref{mono2}.

\medskip\noindent
\emph{Case 2a. The function $\exp(\nu t^{\theta})\Lambda_*(t)$ is increasing.} 
 Then 
\begin{align}\notag
    I_k
    &\le \exp(-\nu T_{k+1}^{\theta}) \sum_{m=0}^k \frac1{\tau_{m+1}}\int_{T_m}^{T_{m+1}} \left[\exp\left(\nu  \tau^{\theta}\right) \Lambda_*(\tau)\right]\d\tau\\
    &=  \frac{1}{\delta_0}\exp(-\nu T_{k+1}^{\theta})\int_{T_0}^{T_{k+1}}  \exp(\nu \tau^{\theta})\Lambda_*(\tau)\d\tau
    =\frac1{\delta_0}y(T_{k+1}),\label{Ik2}
\end{align}
where
\beq\label{yeint}
y(t)=\exp(-\nu t^{\theta})
    \int_{T_0}^t  \exp(\nu \tau^{\theta})\Lambda_*(\tau)\d\tau.
\eeq

We investigate $y(t)$ as $t\to\infty$. Note that $y(t)\ge 0$, for $t\ge T_0$, is a solution of the equation
\beqs
y'(t)= -\nu \theta t^{\theta-1} y(t)+\Lambda_*(t)=-h(t)\varphi(y(t))+\Lambda_*(t),\text{ for $t\ge T_0$,}
\eeqs
where  $h(t)=\nu \theta t^{\theta-1}$ and $\varphi={\rm Id}$.
Clearly,  $\int_{T_0}^\infty h(\tau)\d\tau=\infty$. Then we have from \cite[Lemma A.1]{HIKS1} that
\beq\label{limy}
\begin{aligned}
\limsup_{t\to\infty}y(t)
&\le \limsup_{t\to\infty} \varphi^{-1}(\Lambda_*(t)/h(t))
=\limsup_{t\to\infty} \frac{\Lambda_*(t)}{h(t)}\\
&=\frac1{\nu \theta}\limsup_{t\to\infty} (t^{1-\theta}\Lambda_*(t))=\frac{\ell}{\nu\theta}.    
\end{aligned}
\eeq
Thus, \eqref{IJk}, \eqref{Ik2} and \eqref{limy} imply
\beq\label{limJ2}
\limsup_{k\to\infty}I_k
\le \frac1{\delta_0}\limsup_{t\to\infty}y(t)
\le \frac{\ell}{\delta_0 \nu\theta}.
\eeq

\medskip\noindent
\emph{Case 2b. The function $\exp(\nu t^{\theta})\Lambda_*(t)$ is decreasing.} We rewrite
\beq\label{IS1}
    I_k=\exp(-\nu T_{k+1}^{\theta})\left\{\exp\left(\nu  T_0^{\theta}\right) \Lambda_*(T_0)+ \sum_{m=1}^k \left[\exp\left(\nu  T_m^{\theta}\right) \Lambda_*(T_m)\right]\right\}.
\eeq
We simply bound the first term on the right-hand side of \eqref{IS1} by
\beq \label{esim}
\exp(-\nu T_{k+1}^{\theta})\le \exp(-\nu T_{k}^{\theta}),
\eeq 
and, for the last summation in $m$,
\beqs
\sum_{m=1}^k \left[\exp\left(\nu  T_m^{\theta}\right) \Lambda_*(T_m)\right]
\le \sum_{m=1}^k \frac1{\tau_m}\int_{T_{m-1}}^{T_m} \exp\left(\nu  (\tau^{\theta}\right) \Lambda_*(\tau)\d\tau
=\frac1{\delta_0}\int_{T_0}^{T_{k}} \exp\left(\nu  \tau^{\theta}\right) \Lambda_*(\tau)\d\tau.
\eeqs 
These estimates result in 
\beqs
    I_k\le C'_1\exp(-\nu T_{k}^{\theta})+ \frac1{\delta_0}\exp(-\nu T_{k}^{\theta})\int_{T_0}^{T_{k}} \exp\left(\nu  \tau^{\theta}\right) \Lambda_*(\tau)\d\tau
    =C'_1\exp(-\nu T_{k}^{\theta})+ \frac1{\delta_0}y(T_k),
\eeqs
where $C'_1=\exp\left(\nu  T_0^{\theta}\right) \Lambda_*(T_0)$ and $y(t)$ is defined by \eqref{yeint}.
Same as Case 2a,  we use the limit \eqref{limy} to have
\beq\label{limJ3}
\limsup_{k\to\infty}I_k
\le 0+\frac{1}{\delta_0}\limsup_{t\to\infty}y(t)\le \frac{\ell}{\delta_0 \nu\theta}.
\eeq

\medskip\noindent
\emph{Case 2c. The function $\Lambda_*(t)$ is decreasing.}  We use the formula \eqref{IS1} and inequality \eqref{esim} again. We estimate the term  $\exp(\nu T_m^\theta)$ in \eqref{IS1}, with $m\ge 1$,   by applying inequality \eqref{eeTT} to $m:=m-1$ to have $\exp(\nu T_m^\theta)\le C_1 \exp(\nu T_{m-1}^\theta)$.
We then obtain 
\begin{align*}
    I_k&\le\exp(-\nu T_{k}^{\theta})\left\{C'_1+ C_1\sum_{m=1}^k \left[\exp\left(\nu  T_{m-1}^{\theta}\right) \Lambda_*(T_m)\right]\right\}.
\end{align*}
Using the facts $e^{\nu t^\theta}$ is increasing and $\Lambda_*(t)$ is decreasing, we estimate the last summation by 
\beqs
\sum_{m=1}^k \left[\exp\left(\nu  T_{m-1}^{\theta}\right) \Lambda_*(T_m)\right]
\le \sum_{m=1}^k\frac1{\tau_m}\int_{T_{m-1}}^{T_m} \exp\left(\nu  \tau^{\theta}\right) \Lambda_*(\tau)\d\tau
=\frac1{\delta_0}\int_{T_0}^{T_{k}} \exp\left(\nu  \tau^{\theta}\right) \Lambda_*(\tau)\d\tau.
\eeqs 
Therefore,
\beqs 
    I_k
\le  C'_1\exp(-\nu T_{k}^{\theta})+ \frac{C_1}{\delta_0}\exp(-\nu T_{k}^{\theta})\int_{T_0}^{T_{k}} \exp\left(\nu  \tau^{\theta}\right) \Lambda_*(\tau)\d\tau
=C'_1\exp(-\nu T_{k}^{\theta})+ \frac{C_1}{\delta_0}y(T_k).
\eeqs 
Using  the limit \eqref{limy} again yields
\beq\label{limJ4}
\limsup_{k\to\infty}I_k
\le 0+\frac{C_1}{\delta_0}\limsup_{t\to\infty}y(t)\le \frac{C_1 \ell}{\delta_0 \nu\theta}.
\eeq

\medskip
Thus, in all Cases 2a, 2b and 2c, we have from \eqref{limJ2}, \eqref{limJ3} and \eqref{limJ4} that
\beq\label{Iclaim}
\limsup_{k\to\infty}I_k
\le C_*\ell,
\eeq
where $C_*$ is positive constant independent of $u$, $F_1(t)$, $F(t)$, $\Lambda_*(t)$. 
Now, we can combine \eqref{limu0} with \eqref{IJk} and \eqref{Iclaim}  to have
\beqs %\label{limu0}
\limsup_{t\to\infty} \left(\max_{\overline{Q}_t}u^+\right)
\le C_1\limsup_{k\to\infty}I_k+\ell\le C_1 C_*\ell +\ell,
\eeqs 
which proves \eqref{limest}.

\medskip\noindent
\textbf{Case 3.} 
We select $k_0$ sufficiently large that satisfies not only \eqref{rdp0} and \eqref{rdp2} as in the proof of Case 3 of Theorem \ref{STthm3},  but also
\beq\label{ko2}
T_{k_0}\ge 2 T_0\text{ and } (T_{k_0}/2)^z(1-2^{-z})\ge \left( \frac{\delta_0}{\sigma+1}\right)^z.
\eeq
Let $k>k_0$. Applying \eqref{Jkstar} with $k_*=k_0\ge 1$ gives
\beq \label{JP1}
J_k\le \left(\prod_{j=k_0+1}^k \eta_j\right) J_{k_0} +\sum_{m=k_0+1}^k \left[ \left(\prod_{j=m+1}^k \eta_j\right) \Lambda_m\right] \text{ for all $k>k_0\ge 0$.}
\eeq 
By \eqref{exkdec6},
\beq\label{limJ0} 
\prod_{j=k_0}^k \eta_j\to 0\text{ as }k\to\infty.
\eeq 
Recall from \eqref{eTk2} that $\varep=K/K_*$.
We have the same  inequality as \eqref{lneta3} but for $k_0:=m+1$, hence, with $\nu_0$ coming from \eqref{nuc0}, 
\beq\label{lneta4}
\ln \prod_{j=m+1}^k\eta_j\le \nu_0\left[-(k+2)^{1-\varep(\sigma+1)}+(m+2)^{1-\varep(\sigma+1)}\right] .
\eeq 
On the one hand, we find a lower bound of $k+2$ by \eqref{Tkk5}.
On another hand, we find an upper bound of $m+2$ by applying the second inequality in \eqref{Tkk0} to $k:=m+2$ to have
\beq\label{Tkk6}
m+2\le \left[\frac{\sigma+1}{\delta_0 }(T_{m+2}-T_0)\right]^\frac1{\sigma+1}
\le \left[\frac{\sigma+1}{\delta_0 }T_{m+2}\right]^\frac1{\sigma+1}.
\eeq
It then follows from \eqref{lneta4}, \eqref{Tkk5}, \eqref{Tkk6} as well as \eqref{esz}, that 
\beqs
\ln \prod_{j=m+1}^k\eta_j\le \nu \left[-(T_{k+1}-T_0)^\theta +T_{m+2}^\theta\right] ,
\eeqs 
where  $\nu$ is defined in \eqref{nuc0}.
 With  $\theta\in(0,1)$, we have the inequality $(T_{k+1}-T_0)^\theta\ge T_{k+1}^\theta-T_0^\theta$.
Consequently,
\beq\label{lneta5}
\prod_{j=m+1}^k\eta_j
\le \exp\left\{ \nu (-T_{k+1}^\theta+T_0^\theta +T_{m+2}^\theta)\right\}
= C_2\exp \left[-\nu T_{k+1}^\theta+\nu T_{m+2}^\theta\right] ,
\eeq 
 where $C_2=e^{\nu T_0^\theta}$.
Combining \eqref{JP1}, \eqref{limJ0} and \eqref{lneta5}, we obtain 
\begin{align}\notag
\limsup_{k\to\infty} J_k
&\le 0+C_2 \limsup_{k\to\infty} \sum_{m=k_0+1}^k \left[\exp( -\nu T_{k+1}^{\theta}+\nu T_{m+2}^{\theta}) \Lambda_m\right]\\
\label{limJ1}
     &=C_2 \limsup_{k\to\infty} \left\{\exp(-\nu T_{k+1}^{\theta})S_k\right\},
\end{align}
 where 
\beq \label{Sksum}
S_k=\sum_{m=k_0+1}^k \left[\exp\left( \nu T_{m+2}^{\theta}\right) \Lambda_*(T_m)\right].
\eeq 

Let $m$ be an index number of the sum in \eqref{Sksum}.
We rewrite and estimate $T_{m+2}$ in terms of $T_{m-1}$ by 
\begin{align*}
T_{m+2}
&=T_{m-1}+\delta_0[ m^\sigma+(m+1)^\sigma+(m+2)^\sigma]
\le T_{m-1}+3\delta_0(m+2)^\sigma.
\end{align*}
To estimate $m+2$ in terms of $T_{m-1}$, we use the second inequality in \eqref{Tkk0} with $k:=m-1 $ and drop $T_0$ to have
\beqs
T_{m+2}
\le T_{m-1}+3\delta_0 \left[\left(\frac{\sigma+1}{\delta_0 }T_{m-1}\right)^\frac1{\sigma+1}+3\right]^\sigma.
\eeqs
Using the inequality $(x+y)^\sigma\le 2^\sigma(x^\sigma+y^\sigma)$ for $x,y>0$, we continue  estimating
\beq\label{Tm2}
T_{m+2}
\le T_{m-1}+3\delta_0 2^\sigma  \left[\left(\frac{\sigma+1}{\delta_0 }T_{m-1}\right)^\frac\sigma{\sigma+1}+3^\sigma\right]
=T_{m-1}+\mu_1T_{m-1}^z+\mu_2.
\eeq
where $\mu_1=3\delta_0 2^\sigma \left(\frac{\sigma+1}{\delta_0 }\right)^z$ and $\mu_2=3\delta_0 6^\sigma$.
With $\theta\in(0,1)$, it follows from \eqref{Tm2} that
\beqs 
T_{m+2}^{\theta}\le ( T_{m-1}+\mu_1 T_{m-1}^z+\mu_2)^\theta
\le  T_{m-1}^\theta+\mu_1^\theta  T_{m-1}^{z\theta}+\mu_2^\theta.
\eeqs 
Noticing also that $\nu \mu_1^\theta=\mu_*$, we consequently have
\beq\label{Tme}
\exp(\nu T_{m+2}^\theta)
\le e^{\nu \mu_2^\theta} \exp(\nu T_{m-1}^\theta+\nu\mu_1^\theta  T_{m-1}^{z\theta})
= C_3 \exp(\nu T_{m-1}^\theta+\mu_*  T_{m-1}^{z\theta}),
\eeq
where $C_3=e^{\nu \mu_2^\theta}$.
Using inequality \eqref{Tme} to estimate $S_k$, and then combining it with the facts $\exp(\nu t^\theta+\mu_* t^{z\theta})$ is increasing and $\Lambda_*(t)$ is decreasing, we infer
\beq\label{sel}
\begin{aligned}
    S_k&\le C_3\sum_{m=k_0+1}^k \exp(\nu T_{m-1}^\theta+\mu_*  T_{m-1}^{z\theta})\Lambda_*(T_m)\\
    &\le C_3\sum_{m=k_0+1}^k \frac{1}{\tau_m}\int_{T_{m-1}}^{T_m} \exp(\nu \tau^\theta+\mu_* \tau^{z\theta}) \Lambda_*(\tau)\d\tau.
\end{aligned}
\eeq 
Now, applying  the first inequality of \eqref{Tkk0} to $k:=m$, and also applying the inequality $|x-y|^\sigma\ge 2^{-\sigma}x^\sigma -y^\sigma$ for $x,y>0$, we obtain
\begin{align}
\tau_m
=\delta_0 m^\sigma 
&\ge \delta_0  \left\{\left[\frac{\sigma+1}{\delta_0}(T_m-T_0)\right]^\frac1{\sigma+1} -1\right\}^\sigma
\ge \frac{\delta_0}{2^\sigma}  \left\{\left[\frac{\sigma+1}{\delta_0}(T_m-T_0)\right]^\frac\sigma{\sigma+1} -1\right\} \notag \\
&= \frac{\delta_0}{2^\sigma}\left(\frac{\sigma+1}{\delta_0}\right)^z  \left\{(T_m-T_0)^z -\left(\frac{\delta_0}{\sigma+1}\right)^z\right\}.
\label{taum}
\end{align}
Thanks to \eqref{ko2} and the fact $T_m>T_{k_0}$, we have
\beq\label{Tm3}
(T_m-T_0)^z -\left(\frac{\delta_0}{\sigma+1}\right)^z
\ge (T_m/2)^z -\left(\frac{\delta_0}{\sigma+1}\right)^z
\ge (T_m/4)^z.
\eeq
For  $\tau\in(T_{m-1},T_m)$, we have from \eqref{taum} and \eqref{Tm3} that 
$$\tau_m\ge C_4 T_m^z\ge C_4\tau^z,\text{ where }
C_4=\frac{\delta_0^{1-z}(\sigma+1)^z}{2^{\sigma+2z}}.$$
Combining this with \eqref{sel} gives
\beq
    S_k
    \le C_3\sum_{m=k_0+1}^k \int_{T_{m-1}}^{T_m} \frac{\exp(\nu \tau^\theta+\mu_* \tau^{z\theta})}{C_4\tau^z} \Lambda_*(\tau)\d\tau
    = \frac{C_3}{C_4}\int_{T_{k_0}}^{T_k} \frac{\exp(\nu \tau^\theta+\mu_* \tau^{z\theta})}{\tau^z} \Lambda_*(\tau)\d\tau.\label{sel2}
\eeq

We obtain from \eqref{limJ1} and \eqref{sel2} that
\begin{align}\notag
\limsup_{k\to\infty}J_k
&\le \frac{C_2 C_3}{C_4}  \limsup_{k\to\infty} \left\{\exp(-\nu T_k^{\theta}) \int_{T_{k_0}}^{T_k} \frac{\exp(\nu \tau^\theta+\mu_* \tau^{z\theta})}{\tau^z} \Lambda_*(\tau)\d\tau\right\}\\
&\le \frac{C_2 C_3}{C_4}  \limsup_{t\to\infty} \left\{\exp(-\nu t^{\theta}) \int_{T_{k_0}}^{t} \frac{\exp(\nu \tau^\theta+\mu_* \tau^{z\theta})}{\tau^z} \Lambda_*(\tau)\d\tau\right\}.\label{limJ5}
\end{align}

\medskip\noindent\emph{Case 3a. The function $t\mapsto \int_{T_{k_0}}^{t} \exp(\nu \tau^\theta+\mu_* \tau^{z\theta})\tau^{-z} \Lambda_*(\tau)\d\tau$ is bounded.} Then clearly, the last limit in \eqref{limJ5} is zero.

\medskip\noindent\emph{Case 3b. The function $t\mapsto \int_{T_{k_0}}^{t} \exp(\nu \tau^\theta+\mu_* \tau^{z\theta})\tau^{-z} \Lambda_*(\tau)\d\tau$ is unbounded.} Using  L'Hospital Rule and noticing also that $1-z-\theta=K/K_*$, we obtain
\begin{align*}
&\lim_{t\to\infty}\left\{\exp(-\nu t^{\theta}) \int_{T_{k_0}}^{t} \frac{\exp(\nu \tau^\theta+\mu_* \tau^{z\theta})}{\tau^z} \Lambda_*(\tau)\d\tau\right\}
=\lim_{t\to\infty}\frac{ \exp(\nu t^\theta+\mu_* t^{z\theta}) \Lambda_*(t)}{t^z\cdot \exp(\nu t^{\theta})\nu\theta t^{\theta-1}}\\
&=\frac{1}{\nu\theta}\lim_{t\to\infty}\left\{ t^{K/K_*} \exp(\mu_* t^{z\theta}) \Lambda_*(t)\right\}=\frac{\ell}{\nu\theta}.
\end{align*}

\medskip\noindent Therefore, in both Cases 3a and 3b above, we have
\beq \label{limJ6}
\limsup_{k\to\infty}J_k\le \frac{C_2 C_3}{C_4\nu\theta}\ell .
\eeq 
Then the inequality \eqref{limest} follows from \eqref{limu0} and \eqref{limJ6}.

\medskip\noindent
(ii) We apply part (i) to $u$ and $(-u)$, and then use relation \eqref{absw}.
\end{proof}

%\newpage
\appendix
\section{Proof of Lemma \ref{moreE}}\label{apA}

We use the notation in Assumption \ref{GGamcond}.

\medskip\noindent\emph{Part A.} If $t$ is an interior point of $I$, then, by Assumption \ref{GGamcond}\ref{asii} and by taking a sufficiently small number $h>0$, we have the set $F\eqdef Q_{(t-h,t+h)}$ is bounded. Thus, together with the second property in \eqref{qe1}, $\overline{Q}_t=\overline{F}_t\subset \overline{F}$ is bounded.

Consider $t=\min I$. Then $t\in\partial I$. 
Let $(x,t)\in \overline Q$. 
Because $\mathbb P_2$ is an open mapping, $I=\mathbb P_2(\overline{Q})$ and $t=\mathbb P_2 (x,t)$, then the point $(x,t)$ cannot be an interior point of $\overline Q$. Thus, we must have $(x,t)\in \partial Q$.
Then there is a sequence $(x_k,\tau_k)\in Q$ converging to $(x,t)$ as $k\to\infty$. Since $Q$ is open and, again,  $\mathbb P_2$ is an open mapping, each $\tau_k$ is an interior point of $I$, thus, we must have $\tau_k>t$. Fix any $t_0\in I$ and $t_0>t$. For sufficiently large $k$, we have $\tau_k\in(t,t_0)$ which implies $(x_k,\tau_k)\in F\eqdef Q_{(t,t_0)}$. Hence $(x,t)\in  \overline F$. 
Since the set $\overline F$ is independent of $x$, and is bounded thanks to Assumption \ref{GGamcond}\ref{asii}, we have $\overline{Q}_t$ is bounded.

Similarly, if $t=\max I$ then $\overline{Q}_t$ is bounded. 
In summary, we have proved that $\overline{Q}_t$ is bounded for any $t\in I$.

\medskip\noindent\emph{Part B.} 
With $E=Q_{(t_1,t_2)}$ we have, thanks to \eqref{qe1}, $\overline{Q}_{(t_1,t_2)}=\overline{E}_{(t_1,t_2)}\subset \overline{E}$ which is  bounded thanks to Assumption \ref{GGamcond}\ref{asii}.
Together with $\overline{Q}_{t_1}$ and $\overline{Q}_{t_2}$ being bounded thanks to part A above, we then obtain $\overline{Q}_{[t_1,t_2]}=\overline{Q}_{t_1}\cup \overline{Q}_{(t_1,t_2)}\cup \overline{Q}_{t_2}$ is bounded.
    
\section{Proof of Proposition \ref{gset}}\label{apB}

We prove for the case $J=(t_*,\infty)$. The case $J=\R$ is, in fact, simpler and can be proved in a similar way. Without loss of the generality, we assume also that $t_*=0$, that is, $J=(0,\infty)$. 
 
(a) Thanks to \eqref{zeroc}, the function $R$ is not identically zero in $(0,\infty)$. Hence, $Q$ is not empty. Set $F(x,t)=|x-X(t)|-R(t)$. Then $F$ is continuous in $\R^n\times [0,\infty)$.  We clearly have $$Q=\left(F\Big|_{\R^n\times (0,\infty)}\right)^{-1}((-\infty,0))$$ 
 which implies that $Q$ is an open set of $\R^n\times(0,\infty)$. Because the last set is open in $\R^{n+1}$, so is $Q$. Therefore, $Q$ satisfies Assumption \ref{GGamcond}\ref{as0}.

(b) Observe that 
\beq \label{Xtt}
(X(t),t)\in Q\text{ provided $t>0$ and $R(t)>0$.}
\eeq

Let $t\ge 0$.  If $R(t)=0$, we use the fact $t$ is isolated in $R^{-1}(\{0\})$.
If $R(t)>0$, we use the fact (based on the continuity of $R$ at $t$)
\beq\label{Rlimd}
R(t+)\ge R_-(t)=R(t)>0.
\eeq 
They result in 
\beq \label{Rpos}
\text{$R(\tau)>0$ for any $\tau>t$ near $t$. }
\eeq 
For  any number  $\tau$ in \eqref{Rpos}, we have $\tau>0$, and, thanks to \eqref{Xtt}, the point $(X(\tau),\tau)$ belongs to $Q$. By taking $\tau\searrow t$ and using the convergence $(X(\tau),\tau)\to (X(t),t)$, we obtain $(X(t),t)\in\overline Q$. We have proved
\beq\label{Xtall}
(X(t),t)\in \overline Q\text{ for all }t\ge 0.
\eeq
Therefore,  $\mathbb P_2(\overline{Q})=[0,\infty)$ and Assumption \ref{GGamcond}\ref{asi} is satisfied.

(c) By the continuity of $X$ and $R$ in $[0,\infty)$, the boundedness in Assumption \ref{GGamcond} \ref{asii} is satisfied.

(d) We claim that
\beq\label{QbZ}
\overline{Q}=Z\eqdef \{(x,t)\in \R^n\times[0,\infty): |x-X(t)|\le R(t)\},
\eeq

\emph{Step (d1).} To prove \eqref{QbZ}, we first observe, by the continuity of $X$ and $R$, that $\overline{Q}\subset Z$.

\emph{Step (d2).}  Now, consider any $(x,t)\in Z$. We examine two cases.

\noindent\emph{Case  $x=X(t)$.} By \eqref{Xtall}, we have $(x,t)=(X(t),t)\in \overline{Q}$.

\noindent\emph{Case $x\ne X(t)$.} Set $h=|x-X(t)|>0$. Then $0<h\le R(t)$.
Let $\varep$ be an arbitrary number in $(0,h/2)$ . There is $\delta\in(0,\varep)$ such that, for all $\tau\in(t,t+\delta)$, 
\beq\label{XXRe}
|X(t)-X(\tau)|<\varep\text{ and } R(t)-\varep<R(\tau).
\eeq
Choose $y_\varep\in\R^n$ in the line segment $[X(t),x]$ such that $|y_\varep-x|=2\varep$. Then
\beqs
|x-X(t)|=|x-y_\varep|+|y_\varep-X(t)|=2\varep+|y_\varep-X(t)|.
\eeqs 
Then, for any $\tau\in(t,t+\delta)$,  
\beqs
|y_\varep-X(\tau)|\le |y_\varep-X(t)|+|X(t)-X(\tau)|
=|x-X(t)|-2\varep +\varep\le R(t)-\varep<R(\tau).
\eeqs
Pick any $\tau_\varep\in(t,t+\delta)$, then $(y_\varep,\tau_\varep)\in Q$. Recalling $\delta<\varep$ and letting $\varep\to 0$, we have 
$(y_\varep,\tau_\varep)\to (x,t)$. Hence, $(x,t)\in\overline{Q}$.

The above two cases yield  $Z\subset \overline{Q}$. 
Therefore, $\overline{Q}=Z$, that is, the claim \eqref{QbZ} is true.

As a consequence of \eqref{QbZ}, we have, for $t\ge 0$,
\beqs \label{Qbart}
\overline{Q}_t=\{(x,t):x\in\R^n,|x-X(t)|\le R(t)\}.
\eeqs 

(e) From \eqref{QbZ}, we have 
\beq\label{Qpbt}
\partial Q=\overline{Q}\setminus Q=\overline{Q}_0 \cup \{(x,t)\in \R^n\times(0,\infty):|x-X(t)|=R(t)\}.
\eeq
Let $(x,t)$ be any point in $\partial Q$. 

\medskip\noindent\emph{Case $t=0$.} Then clearly $(x,t)\not\in \gamma(Q)$, thus $(x,t)\in \Gamma(Q)$.

\medskip\noindent\emph{Case $t>0$.} Then $|x-X(t)|=R(t)$.  Let $h$ be any positive number.
By the continuity of $X$ and $R$, there is $\delta\in(0,t)\cap (0,h)$ such that 
\beq\label{RX1}
R(\tau)<R(t)+h/4\text{ and } |X(\tau)-X(t)|<h/4 \text{ for all } \tau\in(t-\delta,t).
\eeq
Let $y\in \partial B_{h/2}(x)$ such that 
\beq \label{yX}
|y-X(t)|=|y-x|+|x-X(t)|,\text{ hence }|y-X(t)|= h/2+R(t).
\eeq 
(If $R(t)=0$, then $x=X(t)$, take any $y\in\partial B_{h/2}(x)$. If $R(t)>0$, take $y\in\R^n$ on the ray going from $X(t)$ to $x$ but outside of the line segment $[X(t),x]$.) 
For any $\tau\in (t-\delta,t)$, we have, by the triangle inequality, \eqref{yX}  and \eqref{RX1}, that
\beq \label{yXtau}
|y-X(\tau)|\ge |y-X(t)|-|X(t)-X(\tau)|\ge h/2+R(t)-h/4=R(t)+h/4>R(\tau).
\eeq 
Pick any $\tau\in(t-\delta,t)$. Then $(y,\tau)\in  \mathcal C_{x,h}^{(t-h,t)}$, and, together with \eqref{yXtau}, it implies
\beq \label{ytauC}
(y,\tau)\in \mathcal C_{x,h}^{(t-h,t)}\setminus Q.
\eeq 
Thus the first condition in \eqref{CCpt} fails, which implies $(x,t)\not\in \gamma(Q)$, and hence,  $(x,t)\in \Gamma(Q)$.

\medskip From the above two cases of $t$, we have 
\beq \label{pbequal}
\partial Q=\Gamma(Q).
\eeq
As a consequence of \eqref{pbequal}, one has  
\beq \label{gQem}
\gamma(Q)=\emptyset\text{ and } \widetilde Q=Q.
\eeq 

  (f) Let $E=Q_{(t_1,t_2)}$ with $t_2>t_1$. Without loss of the generality, assume $t_1\ge 0$. 
By \eqref{gEint} and \eqref{gQem}, 
\beq\label{gEint2} 
\gamma(E)_{(t_1,t_2)}=\gamma(Q)_{(t_1,t_2)}=\emptyset.
\eeq

Thanks to \eqref{pbequal},
\beq\label{gecup}
\gamma(E)_{t_2}\subset \overline{Q}_{t_2}=(Q\cup \Gamma(Q))_{t_2}=Q_{t_2}\cup \Gamma(Q)_{t_2}.
\eeq
Let $(x,t_2)\in \Gamma(Q)$. We follow the arguments in part (e) for $t=t_2$ after \eqref{Qpbt} up to \eqref{ytauC}.
In fact, taking $t=t_2$ and $h$ sufficiently  small in \eqref{ytauC}, we obtain 
\beqs 
(y,\tau)\in \mathcal C_{x,h}^{(t_2-h,t_2)}\setminus E.
\eeqs 
Hence any point $(x,t_2)$ in $\Gamma(Q)_{t_2}$ cannot be in $\gamma(E)$.
Utilizing this fact in \eqref{gecup} yields 
\beq\label{Qctn} \gamma(E)_{t_2}\subset Q_{t_2},
\eeq 
which implies $(t_1,t_2)\in\mathcal G(Q)$.
Thus, the  condition \ref{asv} in Assumption \ref{GGamcond} is satisfied. Together with Parts (a)--(c) above, we have that Assumption \ref{GGamcond} is satisfied.

(g) Finally, the last statement $\partial Q=\Gamma (Q)$ in Proposition \ref{gset} was already proved in \eqref{pbequal}. 

\begin{remark}
In fact, we have more information about $E$ than \eqref{Qctn}.
Indeed, since $Q$ is open and $t_2>0$,  we have $(x,t_2)\in \gamma(E)$ for any $(x,t_2)\in Q$. Thus,
\beq \label{Qincl}
Q_{t_2}\subset  \gamma(E)_{t_2}.
\eeq 
From \eqref{Qctn} and \eqref{Qincl}, one has $\gamma(E)_{t_2}=Q_{t_2}$.
Combining this fact with \eqref{gEQ} and  \eqref{gEint2} gives
$\gamma(E)=Q_{t_2}$.
\end{remark}

\section{Proof of Proposition \ref{gset1}}\label{apC}
Same as in the proof of Proposition \ref{gset}, it suffices to consider $J=(0,\infty)$.
Note from Assumption \ref{Rco}\ref{R2} that $R$ is continuous at $0$.

(a) Since $R$ is not identically zero, the set $Q$ is not empty.
Let $(x,t)\in Q$. Set $h=R(t)-|x-X(t)|$. Then $h>0$.
By the condition \eqref{Rmc}, the definition of $R_-$ in \eqref{RRpm}, and the continuity of $X$, there is $\delta\in(0,t)$ such that
\beq\label{RX2}
R(\tau)>R(t)-h/3\text{ and } |X(\tau)-X(t)|<h/3 \text{ for all } \tau\in(t-\delta,t+t).
\eeq
Let $(y,\tau)\in B_{h/3}(x)\times (t-\delta,t+\delta)$. 
Then we have, from the triangle inequality and \eqref{RX2}, 
\beqs 
|y-X(\tau)|\le |y-x|+|x-X(t)|+|X(t)-X(\tau)|< h/3 +(R(t)-h)+h/3<R(\tau).
\eeqs 
Hence, $(y,\tau)\in Q$. Thus, $\mathcal C_{x,h/3}^{(t-\delta,t+\delta)}\subset Q$, which implies $Q$ is open. 
Therefore, $Q$ satisfies Assumption \ref{GGamcond}\ref{as0}.

(b) We proceed in the same way as Step (b) of the proof of Proposition \ref{gset}. Recall that $R$ is continuous at $t=0$. For $t>0$, by using property \eqref{Rmc}, we still have \eqref{Rlimd}.
Then we obtain \eqref{Xtall} which results in $\mathbb P_2(\overline{Q})=[0,\infty)$ and hence,  Assumption \ref{GGamcond}\ref{asi}. 

(c) Let $t_2>t_1\ge 0$ and $E=Q_{(t_1,t_2)}$.  Thanks to the isolation of the possible discontinuity points, there are at most finitely many discontinuity points of $R$ in $[t_1,t_2]$.
Together with the finite one-sided limits of $R$ everywhere, the function $R$ is bounded in $[t_1,t_2]$. Combining this with the boundedness of $X(t)$ in $[t_1,t_2]$, we have $E$ is bounded. Thus, Assumption \ref{GGamcond}\ref{asii} is satisfied.

(d) We claim that
\beq \label{QbZ1}
\overline Q=Z\eqdef \{(x,t)\in \R^n\times[0,\infty):|x-X(t)|\le R_+(t)\}.
\eeq

Firstly, we need to prove $\overline Q\subset Z$. Let $(x,t)\in \overline{Q}$. Then there is a sequence $(x_k,t_k)\in Q$ which converges to $(x,t)$ as $k\to\infty$. We have
\beq\label{xXR}
|x_k-X(t_k)|<R(t_k).
\eeq
Split the sequence $(x_k,t_k)$ into three subsequences corresponding to $t_k>t$, $t_k<t$, $t_k=t$ (whenever possible). Passing to the limit in \eqref{xXR} for each subsequence gives
\beqs
|x-X(t)|\le \max\left\{ R(t+), R(t-),R(t)\right\}
=\max\left \{ R(t+), R(t-)\right\}=R_+(t).
\eeqs
Thus, $(x,t)\in Z$. This proves $\overline Q\subset Z$.

Secondly, we need to prove $Z\subset \overline Q$. Let $(x,t)\in Z$.

\emph{Case  $R_+(t)=R(t+)$.}
We follow Step (d2) in the proof for Proposition \ref{gset}.
By \eqref{Rmc}, we have in this case
$R(t+)=R_+(t)\ge R_-(t)=R(t)$.
Thus,  we still have inequality \eqref{XXRe}.
The end result is $(x,t)\in \overline{Q}$. 

\emph{Case  $R_+(t)=R(t-)$.}
Note that this  cannot be the case $t=0$.
Then following the same proof as in the first case, replacing $R(t+)$ with $ R(t-)$, 
and replacing $\tau,\tau_\varep\in(t,t+\delta)$ with $\tau,\tau_\varep\in(t-\delta,t)$, we again have $(x,t)\in \overline{Q}$.

The conclusions in both cases imply $Z\subset \overline Q$. Then the claim \eqref{QbZ1} holds true.

(e) As a consequence of \eqref{QbZ1},
\beq\label{dQ1}
\partial Q=\overline{Q}\setminus Q=\overline{Q}_0\cup  S,
\eeq
where
\beq \label{SRR1}
S=\{(x,t)\in \R^{n}\times(0,\infty):R(t)\le |x-X(t)|\le R_+(t)\}.
\eeq
One can prove that the upper-base of $Q$ is
\beq\label{gQ1}
\gamma(Q)
=\left\{(x,t)\in \R^{n}\times(0,\infty):
 R(t-)>R(t+), 
R(t)< |x-X(t)|< R_+(t)\right\}.
\eeq
As a consequence of \eqref{dQ1}, \eqref{SRR1} and \eqref{gQ1}, the parabolic boundary is
\begin{align*}
&\Gamma(Q)
=\overline{Q}_0 \cup \left\{(x,t)\in S: R(t-)\le R(t+)\right\}\\
&\quad \cup \left\{(x,t)\in \R^{n}\times(0,\infty):
 R(t-)> R(t+), 
|x-X(t)|\in\{R(t),R_+(t)\} \right\}.
\end{align*}

(f) Consider the following two cases.

\emph{Case 1. $R$ is continuous at $t_2$.} then one can verify that 
\beqs
\gamma(E)_{t_2}=\{(x,t_2):|x-X(t)|<R(t)\}.
\eeqs
Thus, $\gamma(E)_{t_2}=Q_{t_2}\subset \widetilde Q$ which yields $(t_1,t_2)\in \mathcal G(Q)$.

\emph{Case 2. $R$ is discontinuous at $t_2$.}
For $\tau<t_2$ near $t_2$, we have $R$ is continuous at $\tau$ and 
\beq\label{EtauR}
\overline{E}_{\tau}=\overline{Q}_{\tau} =\left\{(x,\tau):x\in\R^n, |x-X(\tau)|\le  R(\tau)\right\}.
\eeq
Note, as $\tau\nearrow t_2$, $X(\tau)\to X(t_2)$ and $R(\tau)\to R(t_2-)$.

$\bullet$ If $ R(t_2-)> R(t_2+)$, then  
\beq\label{EtR1}
\overline{E}_{t_2}=\left\{(x,t_2):x\in\R^n, |x-X(t_2)|\le  R_+(t_2)\right\}.
\eeq
As $\tau\nearrow t_2$,  the radius $R(\tau)$ in \eqref{EtauR} converges to $R(t_2-)$. Note that this limit $R(t_2-)$ is greater than $R_+(t_2)$ which is the radius in \eqref{EtR1}.

$\bullet$ If $ R(t_2-)<  R(t_2+)$, then  
\beq\label{EtR2}
\overline{E}_{t_2}=\left\{(x,t_2):x\in\R^n, |x-X(t_2)|\le  R(t_2)\right\}.
\eeq
Similarly, as $\tau\nearrow t_2$,  the radius $R(\tau)$ in \eqref{EtauR} converges to $R(t_2-)=R(t_2)$ which is the radius in \eqref{EtR2}.

Together with the continuity of $X$, in both cases above,  we obtain
    \beq\label{Hsd}
    \lim_{\tau\nearrow t_2}\left (\sup_{X\in \overline{E}_{t_2}} {\rm dist}(X,\overline{E}_\tau)\right)=0.
    \eeq
For any integer $k\ge 1$, take $\tau_k<t_2$ sufficiently close to $t_2$ and $\tau_k\to t_2$ as $k\to\infty$. 
By Assumption \ref{Rco}\ref{R2}, $R$ is continuous at $\tau_k$. Applying Case 1 to $\tau_k$, we have $(t_1,\tau_k)\in\mathcal G(Q)$.
With this fact, by taking $\tau=\tau_k$ in \eqref{Hsd}, we obtain \eqref{Hsdk}.

Thus, Assumption \ref{GGamcond}\ref{asv} is satisfied.
Together with Parts (a)--(c), we have that Assumption \ref{GGamcond} is satisfied.

\section{Proof of Proposition \ref{gset2}}\label{apD}
We sketch the proof here. Details can be filled in as in the proofs of Propositions \ref{gset} and \ref{gset1}.
Again, same as in the proof of Proposition \ref{gset}, we only consider $J=(0,\infty)$.
For $i=1,2$, denote $R_{i,+}=(R_i)_+$ and $R_{i,-}=(R_i)_-$ as in \eqref{RRpm}.

Define the function $R=R_2-R_1:\R\times[0,\infty)\to [0,\infty)$.
Note from Assumption \ref{Rco}\ref{R2} that $R_1$, $R_2$ and hence $R$ are continuous at $0$.

(a) Again, $R$ not being identically zero implies $Q$ is not empty. 
Let $(x,t)\in Q$. Let 
\beqs
h=\min\{R_2(t)-x,x-R_1(t)\}.
\eeqs
There is sufficiently small $\delta>0$ such that one has, for all $(y,\tau)\in \mathcal C_{x,h/2}^{(t-\delta,t+\delta)}$, that 
\begin{align*}
y&=(y-x)+(x-R_1(t))+(R_1(t)-R_1(\tau))+R_1(\tau)
>-h/2+h-h/2+R_1(\tau)=R_1(\tau),\\
y&=(y-x)+(x-R_2(t))+(R_2(t)-R_2(\tau))+R_2(\tau)
<h/2-h+h/2+R_2(\tau)=R_2(\tau).
\end{align*}
Thus, $(y,\tau)\in Q$.
Hence $\mathcal C_{x,h/2}^{(t-\delta,t+\delta)}\subset Q$, that is, $Q$ is open.
Therefore, $Q$ satisfies Assumption \ref{GGamcond}\ref{as0}.

(b) Consider $t\ge 0$.
Denote $X(t)=\frac12(R_1(t)+R_2(t)).$
%If $t>0$ and $R(t)>0$, then $(X(t),t)\in Q$.

\emph{Case $R(t)>0$.} We have
\beq \label{hh1}
h\eqdef R(t)/2=R_2(t)-X(t)=X(t)-R_1(t)>0.
\eeq 
By the one-sided limits of $R_1(\tau)$ and $R_2(\tau)$ as $\tau\searrow t$ and \eqref{RR12}, there is $\delta>0$ such that it holds, for all $\tau\in(t,t+\delta)$,
\beqs
R_1(\tau)<R_1(t)+h/2, \quad R_2(\tau)>R_1(t)-h/2.
\eeqs
Then 
$R_1(\tau)<X(t)<R_2(\tau)$, hence $(X(t),\tau)\in Q$.  Taking $\tau\to t$, we have $(X(t),t)\in \overline{Q}$.

\emph{Case $R(t)=0$.}  By the isolation of $t$ in the set $R^{-1}(\{0\})$, we have $R(\tau)>0$ for $\tau>t$ near $t$.

$\bullet$ If $ R_1(t+)< X(t)< R_2(t+)$, then same as the case $R(t)>0$, by using 
$$h=\min\{X(t)-R_1(t+),R_2(t+)-X(t)\}>0$$ instead of \eqref{hh1}, we can prove $(X(t),t)\in \overline{Q}$.

$\bullet$ If $ R_i(t+)=X(t)$ for $i=1$ or $2$, thanks to $R(\tau)>0$, take $x_\tau\in (R_1(\tau),R_2(\tau))$ but near $ R_i(\tau)$.
Since $R_i(\tau)\to X(t)$,  as $\tau\searrow t$, one has $x_\tau\to X(t)$. Thus $(x_\tau,\tau)\in Q$ converges to $(X(t),t)$ as $\tau\searrow t$. Hence, $(X(t),t)\in \overline{Q}$.

In conclusion, $(X(t),t)\in \overline{Q}$ for all $t\in[0,\infty)$. Consequently,  $\mathbb P_2(\overline{Q})=[0,\infty)$, that is, Assumption \ref{GGamcond}\ref{asi} is satisfied.

(c) Let $t_2>t_1\ge 0$ and $E=Q_{(t_1,t_2)}$.  Thanks to the isolation of the possible discontinuity points, there are at most finitely many discontinuity points of $R$ in $[t_1,t_2]$.
Together with the finite one-sided limits of $R$ every where, the function $R$ is bounded in $[t_1,t_2]$. Combining this with the boundedness of $X(t)$ in $[t_1,t_2]$, we have $E$ is bounded. Thus, Assumption \ref{GGamcond}\ref{asii} is satisfied.

(d) Same as part (d) in the proof of  Proposition \ref{gset1}, we have 
\beq\label{QbZ2}
\overline{Q}=\{(x,t)\in\R\times[0,\infty):R_{1,-}(t)\le x\le R_{2,+}(t)\}.
\eeq

(e) As a consequence of \eqref{QbZ2},
\beqs
\partial Q=\overline{Q}\setminus Q=\overline{Q}_0\cup S_1\cup S_2,
\eeqs
where
\begin{align*}
S_1&=\left\{(x,t)\in \R\times(0,\infty):R_{1,-}(t)\le x\le R_1(t)\right\},\\
S_2&=\left\{(x,t)\in \R\times(0,\infty):R_{2}(t)\le x\le R_{2,+}(t)\right\}.
\end{align*}
The upper-base of $Q$ is
\beq\label{gQ2}
\gamma(Q)
=\left\{(x,t)\in S_1: R_1(t-)< R_1(t+)\right\}
 \cup \left\{(x,t)\in S_2: R_2(t-)> R_2(t+)\right\}.
\eeq
The parabolic boundary of $Q$ is 
\beqs
\Gamma(Q) 
=\overline{Q}_0 
\cup \left\{(x,t)\in S_1: R_1(t-)\ge  R_1(t+)\right\}\cup  \left\{(x,t)\in S_2: R_2(t-)\le  R_2(t+)\right\}.
\eeqs

(f) Let $t_2>t_1\ge 0$ and $E=Q_{(t_1,t_2)}$. 
For $\tau<t_2$ near $t_2$, the functions $R_1$ and $R_2$ are continuous at $\tau$, and 
\beq\label{Etauseg}
\overline{E}_{\tau}=\overline{Q}_{\tau} =[R_1(\tau),R_2(\tau)]\times\{\tau\}
=[(R_1(\tau),\tau),(R_2(\tau),\tau)],
\eeq
where the last notation $[\cdot,\cdot]$ denotes a line segment in $\R^2$, see \eqref{lineseg}.

Next, the set $\overline{E}_{t_2}$ can be described in all cases below.

$\bullet$ If $ R_1(t_2-)> R_1(t_2+)$ and  $ R_2(t_2-)< R_2(t_2+)$,  then  
$$\overline{E}_{t_2}=[R_1(t_2),R_2(t_2)]\times\{t_2\}
=[(R_1(t_2-),t_2),(R_2(t_2-),t_2)].$$

$\bullet$ If $ R_1(t_2-)\le  R_1(t_2+)$ and  $ R_2(t_2-)\ge  R_2(t_2+)$,  then  
$$\overline{E}_{t_2}=[R_{1,-}(t_2),R_{2,+}(t_2)]\times\{t_2\}
=[(R_1(t_2-),t_2),(R_2(t_2-),t_2)].$$

$\bullet$ If $ R_1(t_2-)> R_1(t_2+)$ and   $ R_2(t_2-)\ge  R_2(t_2+)$,  then  $$\overline{E}_{t_2}=[R_{1}(t_2),R_{2,+}(t_2)]\times\{t_2\}
=[(R_1(t_2-),t_2),(R_2(t_2-),t_2)].$$

$\bullet$ If $ R_1(t_2-)\le  R_1(t_2+)$ and  $ R_2(t_2-)< R_2(t_2+)$, then  
$$\overline{E}_{t_2}=[R_{1,-}(t_2),R_{2}(t_2)]\times\{t_2\}
=[(R_1(t_2-),t_2),(R_2(t_2-),t_2)].$$

In all four cases, the end points of $\overline{E}_{\tau}$ in \eqref{Etauseg} converge  to $(R_1(t_2-),t_2)$ and $(R_2(t_2-),t_2)$, as $t\nearrow t_2$, which are the corresponding end points of $\overline{E}_{t_2}$, thus \eqref{Hsd} is satisfied.
Combining this with Assumption \eqref{Rco}\ref{R2} for both $R_1$ and $R_2$, we deduce that  Assumption \ref{GGamcond}\ref{asv} is satisfied.
Together with Parts (a), (b) and (c) above, we conclude that Assumption \ref{GGamcond} is satisfied.

\bigskip
\noindent\textbf{Data availability.} 
No new data were created or analyzed in this study.

\medskip
\noindent\textbf{Methods.} 
No Artificial Intelligence Generated Content (AIGC) tools are used in developing any portion of this paper. 

\medskip
\noindent\textbf{Funding.} A.I.'s research is supported by OGRI, RAS, grant number 122022800272-4. 

\medskip
\noindent\textbf{Conflict of interest.}
There are no conflicts of interests.

\bibliography{paperbaseall}{}
\bibliographystyle{abbrv}
\end{document}